\documentclass[a4paper,12pt]{article}
\usepackage{array,amsthm,amsfonts,epsfig}
\usepackage{graphicx}
\usepackage{dina42e}
\usepackage{amsmath}
\usepackage{mathdots}
\usepackage{amssymb}
\usepackage{textcomp}
\usepackage[T1]{fontenc}
\usepackage{fancyhdr}
\usepackage[utf8]{inputenc} 
\usepackage{euscript} 
\usepackage{mathrsfs}  
\usepackage{calligra} 

\newtheorem{satz}{Theorem}[section]
\newtheorem{thm}[satz]{Theorem}
\newtheorem{lemma}[satz]{Lemma}
\newtheorem{kor}[satz]{Corollary}
\newtheorem{prop}[satz]{Proposition}

\theoremstyle{definition}
\newtheorem{Def}[satz]{Definition}
\theoremstyle{remark}
\newtheorem{bem}[satz]{Remark}
\newtheorem{bsp}[satz]{Example}

\newcommand{\dq}{\textrm{\textit{\dj}} }                                           
\newcommand{\varf}[2]{#1_1, \ldots, #1_{#2}}                              
\newcommand{\vara}[2]{#1_1 + \ldots + #1_{#2}}                            
\newcommand{\<}[1]{\langle #1 \rangle}                                    
\newcommand{\op}{OP}                                             
\newcommand{\e}{\varepsilon}                                              
\DeclareMathOperator{\ad}{ad}                                             
\newcommand{\supp}{\textrm{supp }}                                        
\newcommand{\adc}[2]{ \ad(-ix)^{#1_1}\ad(D_x)^{#2_1}\ldots \ad(-ix)^{#1_l}\ad(D_x)^{#2_l} }  
\newcommand{\Sallg}[5]{S^{#1}_{#2,#3}(\R^{#4}\times\R^{#5})}              
\newcommand{\Sallgn}[6]{S^{#1}_{#2,#3}(\R^{#4}\times\R^{#5}; #6)}         
\newcommand{\Sn}[3]{\Sallg{#1}{#2}{#3}{n}{n}}                             
\newcommand{\Snn}[4]{S^{#1}_{#2,#3}(\R^{n}\times\R^{n}; #4)}              
\newcommand{\differencequotient}[1]{\p^{#1}_{x_j}}                           
\newcommand{\pa}[1]{\partial_{\xi}^{#1}}                                  
\newcommand{\p}{\partial}
\newcommand{\pax}[2]{\partial_{\xi}^{#1}\partial_{x}^{#2}}                
\newcommand{\skh}[3]{{\langle #1,#2 \rangle}_{#3}}                        
\newcommand{\R}{\mathbb{R}}                                               
\newcommand{\Rn}{\mathbb{R}^n}                                            
\newcommand{\RnRn}{\R^n \times \R^n}                                      
\newcommand{\RnRnx}[2]{\R^n_{#1} \times \R^n_{#2}}                         
\newcommand{\RnRnRn}{\RnRn \times \R^n}                                   
\newcommand{\RnRnRnRn}{\RnRn \times \RnRn}                                
\newcommand{\intr}{\int \limits_{\Rn} }                                   
\newcommand{\osint}{\textrm{Os\hspace*{0,1cm}-}\hspace{-0,15cm}\iint}     
\newcommand{\osiint}{\textrm{Os\hspace*{0,1cm}-}\hspace{-0,15cm}\iiiint}  
\newcommand{\s}{\mathcal{S}(\R^n)}                                        
\newcommand{\sd}{\mathcal{S'}(\R^n)}                                      
\newcommand{\N}{\mathbb{N}}                                               
\newcommand{\Z}{\mathbb{Z}}                                               
\newcommand{\Non}{\mathbb{N}_0^n}                                               
\newcommand{\C}{\mathbb{C}}                                               

\title{Spectral Invariance of Non-Smooth Pseudodifferential Operators}
\author{Helmut Abels and Christine Pfeuffer}

\begin{document}

\maketitle

\begin{abstract}
  In this paper we discuss some spectral invariance results for non-smooth pseudodifferential operators with coefficients in Hölder spaces. In analogy to the proof in the smooth case of Beals and Ueberberg, c.f.\;\cite{Ueberberg}, \cite{Beals}, we use the characterization of non-smooth pseudodifferential operators to get such a result. The main new difficulties are the limited mapping properties of pseudodifferential operators with non-smooth symbols and the fact, that in general the composition of two non-smooth pseudodifferential operators is not a pseudodifferential operator. 
  
  In order to improve these spectral invariance results for certain subsets of non-smooth pseudodifferential operators with coefficients in Hölder spaces, we improve the characterization of non-smooth pseudodifferential operators of A. and P., c.f. \cite{Paper1}.
\end{abstract}

\section{Introduction}

A lot of spectral invariance results of pseudodifferential operators are already known for pseudodifferential operators with smooth symbols e.g.\;of the Hörmander class $\Sn{m}{\rho}{\delta}$. The symbol-class $\Sn{m}{\rho}{\delta}$ with $m \in \R$ and $0 \leq \delta \leq \rho \leq 1$ consists of all smooth functions $p$ such that for all $k \in \N_0$
\begin{align*}
  |p|^{(m)}_k := \max_{|\alpha|,|\beta|\leq k} \sup_{x, \xi \in \R^n}|\pax{\alpha}{\beta} p(x,\xi)|\<{\xi}^{-(m-\rho|\alpha|+\delta|\beta|)} < \infty.
\end{align*}
We define the associated pseudodifferential operator via
\begin{align*}
		\op (p):= p(x,D_x) u(x) := \intr e^{ix \cdot \xi} p(x,\xi) \hat{u}(\xi) \dq \xi \qquad \text{for all } u \in \s , x \in \Rn,
\end{align*}
where $\s$ is
the Schwartz space,i.\,e.\,the space of all rapidly decreasing smooth functions. Moreover $\hat{u}$ and $\mathscr{F}[u]$ denote the Fourier transformation of $u$. 
Additionally $\op \Sn{m}{\rho}{\delta}$ is the set of all pseudodifferential operators with symbols in the symbol-class $\Sn{m}{\rho}{\delta}$.

A fundamental result in the theory of pseudodifferential operators allows one to directly conclude that the inverse of a pseudodifferential operator with a symbol in the Hörmander class $S^0_{\rho, \delta}(\RnRn)$, $0 \leq \delta \leq \rho \leq 1$, $\delta < 1$, which is invertible as an operator on $L^2(\Rn)$, is again a pseudodifferential operator in the same symbol-class. This important statement was derived by Beals \cite{Beals} and Ueberberg \cite{Ueberberg}. Their proof even showed that the same holds for all Bessel potential spaces $H^s_2(\Rn)$, $s \in \R$, see Definition \ref{BesselPotentialSpace} for the definition of these spaces, and that the spectrum is independent of the choice of the space.

Schrohe extended this result for weighted $L^p$-Sobolev spaces in \cite{Schrohe1990} and together with Leopold even for Besov spaces of variable order of differentiation $B_{p,q}^{s,a}(\Rn)$ in \cite{SchroheLeopold1992}. They verified that the spectrum of smooth pseudodifferential operators in certain symbol-classes is independent of the choice of the weighted $L^p$-Sobolev space and of the choice of the Besov space of variable order of differentiation respectively, cf.\;\cite{SchroheLeopold1992} and \cite{Schrohe1990}. 

There are several other results for spectral invariance of smooth pseudodifferential operators in the literature, cf.\;e.g.\;\cite{AlvarezHounie}, \cite{Grubb}, \cite{Kryakvin},\cite{LeopoldTriebel} and \cite{Schrohe1992}.

We show the spectral invariance for non-smooth pseudodifferential operators whose symbols are in the symbol-class $C^{\tau}_{\ast} S^0_{1,0}(\Rn \times \Rn)$ in this paper. Here
\begin{align*}
  C^{\tau}_{\ast}(\Rn):=\left\{ f \in \sd: \|f\|_{C^{\tau}_{\ast} }:= \sup_{j \in \N_0} 2^{js} \|\mathscr{F}^{-1}[\varphi_j \hat{f}] \|_{L^{\infty}} < \infty \right\},
\end{align*}
is the so-called Hölder-Zygmund space, where $\mathscr{F}^{-1}[u]$ denotes the inverse Fourier Transformation of $u \in \sd$, i.e., the dual space of $\s$. Note that $C^{\tau}_{\ast}(\Rn)$ is equal to the Hölder space $C^{\tau}(\Rn)$ if $\tau \notin \N$. The symbol-class $ C^{\tau}_{\ast} \Snn{m}{\rho}{0}{M}$ with $m \in \R$, $\tau>0$, $0 \leq \rho \leq 1$ and $M\in \N_0 \cup \{ \infty \}$ is the set of all functions $p: \RnRn \rightarrow \C$ such that we have for all $\alpha, \beta \in \N_0^n$ with $|\alpha| \leq M$ and $|\beta| \leq m$:
  \begin{itemize}

		\item[i)] $\p_x^{\beta} p(x, .) \in C^M(\Rn)$ for all $x \in \Rn$,

		\item[ii)] $\p_x^{\beta} \pa{\alpha} p \in C^{0}(\R^n_x \times \R^n_{\xi})$,

		\item[iii)] $\| \pa{\alpha} p(.,\xi)  \|_{C^{\tau}_{\ast}(\R^n)} \leq C_{\alpha}\<{\xi}^{\tilde{m}-\rho|\alpha|}$ for all $\xi \in \R^n$.
	\end{itemize}
For a given symbol $p$ of the previous symbol-class we define the associated pseudodifferential operator in the same way as in the smooth case and denote it by $p(x,D_x)$ or $\op(p)$. The set $\op C^{\tau}_{\ast} \Snn{m}{\rho}{0}{M}$ consists of all pseudodifferential operators with symbols in the symbol-class $C^{\tau}_{\ast} \Snn{m}{\rho}{0}{M}$.

In analogy to the proof of the spectral invariance results of Ueberberg in the smooth case, we use the characterization of non-smooth pseudodifferential operators via iterated commutators. The characterization set is defined as follows:  

\begin{Def}
	Let $m\in\R$, $M \in \N_0 \cup \{ \infty \}$ and $ 0\leq \rho \leq 1 $. Additionally let $\tilde{m} \in \N_0 \cup \{ \infty \}$ and $1 < q < \infty$. Then we define 
	$\mathcal{A}^{m,M}_{\rho,0}(\tilde{m},q)$
	as the set of all linear and bounded operators $P: H^m_q(\Rn) \rightarrow L^q(\Rn)$, such that for all $l\in\N$, $\alpha_1, \ldots, \alpha_l \in \Non$ and $ \beta_1, \ldots, \beta_l \in \N_0^n$ with $|\alpha_1| + |\beta_1|= \ldots = |\alpha_l| + |\beta_l|= 1$, $|\alpha| \leq M$ and $|\beta|\leq \tilde{m}$ the iterated commutator of $P$
	\begin{eqnarray*}
		\adc{\alpha}{\beta} P: H^{m-\rho|\alpha|}_q(\Rn) \rightarrow L^q(\Rn)
	\end{eqnarray*}
	is continuous. Here $\alpha:= \alpha_1 + \ldots + \alpha_l$ and $\beta:=\beta_1 + \ldots + \beta_l$. 
\end{Def}

The iterated commutators are defined in Definition \ref{Def:IteratedCommutators} below. In the case $M = \infty$ we write $\mathcal{A}^{m}_{\rho,0}(\tilde{m},q)$ instead of $\mathcal{A}^{m,\infty}_{\rho,0}(\tilde{m},q)$. In \cite{Paper1} we showed that every operator in such a characterization set is a non-smooth pseudodifferential operator if certain conditions hold:

\begin{thm}\label{thm:classificationA10}
  Let  $m\in \R$, $1<q<\infty$ and $\tilde{m} \in \N_0$ with $\tilde{m}>n/q$. Additionally let $M \in \N_0$ with $M >n$. We define $\tilde{M}:=M-(n+1)$. Assuming $P \in \mathcal{A}^{m, M}_{1,0}(\tilde{m},q)$ and $\tilde{M} \geq 1$, we obtain for all $ \tau \in \left(0,\tilde{m}-n/q \right]$ with $\tau \notin \N_0$: 
  $$P \in \op C^{\tau} S^m_{1,0}(\RnRn; \tilde{M}-1) \cap \mathscr{L}(H^m_q(\Rn),L^q(\Rn)).$$
\end{thm}

\begin{thm} \label{thm:classA00}
  Let $m\in \R$, $1<q<\infty$, $\tilde{m} \in \N_0$ with $\tilde{m}>n/q$. Additionally let $M \in \N_0 \cup \{ \infty \}$ with $M >n$. We define $\tilde{M}:=M-(n+1)$. Considering an operator $T \in \mathcal{A}^{m,M}_{0,0}(\tilde{m},q)$ and an $\tilde{M} \geq 1$ we have for $s \in (0, \tilde{m}-n/q]$ with $s\notin \N_0$:
  $$T \in \op C^s S^m_{0,0}(\RnRn; \tilde{M}-1) \cap \mathscr{L}(H^m_q(\Rn),L^q(\Rn)).$$
\end{thm} 

The proofs of those two results are based on the existence of a pointwise convergent subsequence of a bounded set of non-smooth symbols and of the symbol reduction of non-smooth double symbols to a single symbol. We refer to Subsection \ref{Section:DoubleSymbols} for the definition of the non-smooth double symbols. For the convenience of the reader let us repeat these two results now:

\begin{thm}\label{kor:konvergenz}\label{thm:pointwiseConvergence}
  Let $m \in \N_0$, $M \in \N \cup \{ \infty \}$ and $0<s\leq 1$. 
  Furthermore, let $(p_{\e} )_{\e>0} \subseteq C^{m,s}S^0_{0,0}(\RnRn ;M)$ be a bounded sequence. 
  Then there is a subsequence $(p_{\e_l} )_{l \in \N} \subseteq (p_{\e} )_{\e>0}$ with $\e_l \rightarrow 0$ for $l \rightarrow \infty$ and a function $p: \Rn_x \times  \Rn_{\xi} \rightarrow \C$ such that for all $\alpha, \beta \in \N_0^n$ with $|\beta| \leq m$ and $|\alpha| \leq M-1$ we get
  \begin{itemize}
    \item[i)] $\p_x^{\beta}\pa{\alpha} p$ exists and $\p_x^{\beta} \pa{\alpha} p \in C^{0,s}(\RnRn)$,
    \item[ii)] $\p^{\beta}_x \pa{\alpha} p_{\e_l} \xrightarrow[]{l \rightarrow \infty} \p^{\beta}_x \pa{\alpha} p$ is uniformly convergent on each compact set $K \subseteq \RnRn$.
  \end{itemize}
  In particular $p \in C^{m,s}S^0_{0,0}( \RnRn; M-1)$.
\end{thm}

\begin{thm}\label{thm:SymbolReduktionNichtGlatt}
  Let $N \in \N_0 \cup \{ \infty \}$ with $N > n$. We define $\tilde{N}:= N-(n+1)$. Furthermore let $\mathscr{B} \subseteq C^{ \tilde{m} ,s} S^m_{0,0}(\RnRnRn; N)$ be bounded  with $\tilde{m} \in \N_0$, $m \in \R$ and $0<s<1$. For $a \in \mathscr{B}$ we define $a_L:\RnRn \rightarrow \C$ by
  \begin{align*}
    a_L(x,\xi) := \osint e^{-iy \cdot \eta} a(x, \eta + \xi, x+y) dy \dq \eta
  \end{align*}
  for all $x, \xi \in \Rn$. Then $\{a_L: a \in \mathscr{B}\} \subseteq C^{\tilde{m}, s} S^m_{0,0} (\RnRnx{x}{\xi}; \tilde{N})$ is bounded and
  \begin{align*}
    a(x,D_x, x') u = a_L(x,D_x) u \qquad \text{for all } a \in \mathscr{B}, u \in \s.
  \end{align*}
\end{thm}

For more details we refer to \cite[Theorem 4.3, Theorem 4.13 and Theorem 4.15]{Paper1}. In Theorem \ref{thm:classificationA10} and Theorem \ref{thm:classA00} we loose some regularity with respect to $\tilde{m}$.
This loss of regularity can be prevented as we will see in this paper. \\

In the present paper we proceed as follows: In Section \ref{section:Preliminaries} we summarize all notations and results needed later on. In particular we introduce the uniformly local Sobolev Spaces $W^{m,q}_{ uloc}$ and verify some important properties of these spaces, cf.\;Subsection \ref{subsection:HmqUlocAbschaetzung}. Section \ref{Kapitel: PDO} is devoted to the definition and the properties of pseudodifferential operators of certain symbol-classes needed in this paper. We begin with pseudodifferential operators with single symbols in Subsection \ref{Section:PropertiesOfPDO} while pseudodifferential operators with double symbols are treated in Section \ref{Section:DoubleSymbols}. 

The main purpose of Section \ref{ImprovementCharacterization} is to improve the characterization of non-smooth pseudodifferential operators with coefficients in Hölder spaces, cf. Theorem \ref{thm:classificationA10} and Theorem \ref{thm:classA00}. To this end we improve Theorem \ref{thm:pointwiseConvergence} in Subsection \ref{subsection:pointwiseConvergenceHmqUloc} and verify a result for the symbol reduction of pseudodifferential operators with coefficients in an uniformly local Sobolev space in Subsection \ref{subsection:SymbolReductionHmqUloc}. With those results at hand, we are able to improve the characterization of non-smooth pseudodifferential operators in Subsection \ref{CharactPDO}.

By means of this characterization we show several spectral invariance results for non-smooth pseudodifferential operators in Section \ref{SpectralInvariance}. Subsection \ref{section:SpectralInvarianceHölderCase} is devoted to the inverse of a non-smooth pseudodifferential operator $P$ in the symbol-class $C^{\tau} S^0_{0,0}(\Rn \times \Rn)$. We show that $P^{-1}$ is also a non-smooth pseudodifferential operator of the symbol-class $C^{s} S^0_{0,0}(\Rn \times \Rn)$, where $s < \tau$. Unfortunately, in contrast to the smooth case, we loose some smoothness of the coefficients. 
Our next goal is to prove the spectral invariance of non-smooth pseudodifferential operators of the class $C^{\tau} \Snn{0}{1}{0}{N}$ for sufficiently large $N$. To be more precise, we arrive at the following statement: The inverse of a non-smooth pseudodifferential operator of the order zero with coefficients in the Hölder space $C^{\tilde{m},\tau}(\Rn)$ is also a non-smooth pseudodifferential operator if its inverse is an element of $\mathscr{L}(H_q^r(\Rn))$ for one $|r| < \tilde{m}+ \tau$. This is the topic of Subsection 
\ref{section:SpectralInvarianceHölderCaseS10}. Beyond 
the characterization of non-smooth pseudodifferential operators we also use the technique of approximation with difference quotients for the proof of the above mentioned statement. We introduce this technique in Subsection \ref{section:DifferenceQuotients}. We are able to improve the results of Subsection \ref{section:SpectralInvarianceHölderCaseS10} in Subsection \ref{section:SpectralInvarianceWmqUlocCase} for certain subclasses of the non-smooth pseudodifferential operators with coefficients in Hölder spaces.\\
The present paper is based on a part of the PhD-thesis, cf. \cite{Diss}, of the second author in this paper advised by the first author.\\
\textbf{Acknowledgement:} We would like to thank Prof. Dr. Schrohe for his helpful suggestions of improvements.

\section{Preliminaries}\label{section:Preliminaries}

We assume $n \in \N$ throughout the whole paper unless otherwise noted. In particular $n \neq 0$. For $x \in \R$ we define 
\begin{align*}
  x^+:=\max \{0;x \} \qquad \text{and} \qquad \lfloor x \rfloor:= \max \{k \in \Z : k \leq x \}.
\end{align*}
Additionally $$\<{x}:=(1+|x|^2)^{1/2} \quad \text{for all } x \in \Rn \qquad \text { and } \qquad  \dq \xi:= (2 \pi)^{-n} d \xi.$$
Partial derivatives with respect to a variable $x\in \Rn$ scaled with the factor $-i$ are denoted by 
$$D_x^{\alpha}:= (-i)^{|\alpha|} \p^{\alpha}_x := (-i)^{|\alpha|} \p^{\alpha_1}_{x_1} \ldots \p^{\alpha_n}_{x_n} .$$ 
Here $\alpha =(\alpha_1, \ldots, \alpha_n) \in \Non$ is a \textit{multi-index}. 
For $j\in \{1,\ldots, n\}$ we define  $e_j \in \N^n_0$ as the $j$-the canonical unit vector, i.e., $(e_j)_k = 1$ if $k=j$ and $(e_j)_k = 0$ else.

Considering two Banach spaces $X,Y$ the set $\mathscr{L}(X,Y)$
contains of all linear and bounded operators $A:X \rightarrow Y$. If $X=Y$, we also just write $\mathscr{L}(X)$. \\

Iterated commutators of linear operators are defined in the usual way:

\begin{Def}\label{Def:IteratedCommutators}
  Let $X,Y \in \{ \s, \sd \}$ and $T: X \rightarrow Y$ be linear. We define the linear operators $\ad(-ix_j) T: X \rightarrow Y$ and $\ad(D_{x_j}) T: X \rightarrow Y$ for all $j\in \{ 1, \ldots, n\}$ and $u \in X$ by
  \begin{eqnarray*}
    \ad(-ix_j) T u:= -ix_j Tu + T \left( ix_j u \right) \quad \textrm{and} \quad \ad(D_{x_j}) T u:= D_{x_j} \left( T u \right) - T \left( D_{x_j} u \right).
  \end{eqnarray*}
  For arbitrary multi-indices $\alpha, \beta \in \N_0^n$ we denote the \textit{iterated commutator} of $T$ as
  \begin{eqnarray*}
    \ad(-ix)^{\alpha}\ad(D_{x})^{\beta} T := [\ad(-ix_1)]^{\alpha_1} \ldots [\ad(-ix_n)]^{\alpha_n} [\ad(D_{x_1})]^{\beta_1} \ldots [\ad(D_{x_n})]^{\beta_n} T.
  \end{eqnarray*}
\end{Def}

On account of the properties of the Fourier transformation we obtain:

\begin{bem}\label{bem:SymbolOfIteratedCommutatorNonSmooth}\label{bem:SymbolOfIteratedCommutator}
  Let $\tilde{m}\in \N_0$, $M \in \N_0 \cup \{ \infty \}$, $0< \tau \leq 1$, $m\in \R$ and $0 \leq \rho \leq 1$. We assume that $ p \in C^{\tilde{m}, \tau} S^m_{\rho,0} (\RnRn; M)$. Moreover, let $l \in \N$, $\alpha_1, \ldots, \alpha_l \in \Non$ and $\beta_1, \ldots, \beta_l \in \Non$ with $|\alpha_j + \beta_j| = 1$ for all $j \in \{ 1, \ldots, n\}$, $|\alpha| \leq M$ and $|\beta| \leq \tilde{m}$. Here $\alpha$ and $\beta$ are defined by $\alpha:= \alpha_1 + \ldots + \alpha_l$ and $\beta := \beta_1 + \ldots + \beta_l$. Then the operator 
  $$\ad(-ix)^{\alpha_1} \ad(D_x)^{\beta_1} \ldots \ad(-ix)^{\alpha_l} \ad(D_x)^{\beta_l} p(x,D_x)$$ 
  is a pseudodifferential operator with the symbol 
  $$\pa{\alpha} D^{\beta}_x  p(x,\xi) \in C^{\tilde{m}- |\beta|, \tau} S^{m-\rho |\alpha|}_{\rho,0} (\RnRnx{x}{\xi}; M-|\alpha|).$$
  If we even have $ p \in S^m_{\rho,0} (\RnRn)$, then $\pa{\alpha} D^{\beta}_x  p(x,\xi) \in S^{m-\rho |\alpha| }_{\rho,0} (\RnRn)$.
\end{bem}

Finally let us mention that the dual space of a topological vector space $V$ is denoted by $V'$.
In view of $V$ being a Banach space the duality product $V$ is denoted by $\skh{.}{.}{V; V'}$.

\subsection{Functions on $\Rn$ and Function Spaces}

For the convenience of the reader we introduce all functions and function spaces needed later on in this subsection. We start 
with the \textit{Hölder space}  of the order $m \in \N_0$ with Hölder continuity exponent $s \in (0,1]$ which is denoted by $C^{m,s}(\Rn)$ and also by $C^{m+s}(\Rn)$.  
Moreover for $s \in \R$, $1<p<\infty$ and $\<{D_x}^s:=\op(\<{\xi}^s) $ the set 
\begin{align}\label{BesselPotentialSpace} 
  H^s_p(\Rn):= \{ f \in \sd: \<{D_x}^s f \in L^p(\Rn) < \infty \}
\end{align}
is called  \textit{Bessel Potential space}.\\

For $y \in \Rn$ the translation function $\tau_y(g): \Rn \rightarrow \C$ of $g \in L^1(\Rn)$ is defined  as $\tau_y(g)(x):=g(x-y)$ for all $ x\in \Rn$.
    
Moreover a function $f : \Rn \rightarrow \C$ is \textit{homogeneous of degree $d\in \R$ (for $|x| \geq 1$)} if
    $f(rx) = r^d f(x) $ $x \in \Rn \text{ for all } | x | \geq 1 \text{ and } r \geq 1.$
Note that the derivatives $\p_x^{\alpha} f$, $\alpha \in  \Non$ with $|\alpha| \leq k$, of functions $f \in C^k(\Rn)$ which are homogeneous of degree $d$ are homogeneous of degree $d-|\alpha|$.\\

A frequently used ingredient for verifying several results in this paper is the dyadic partition of unity, i.e.\,, a partition of unity $(\varphi_j)_{j \in \N_0}$ on $\Rn$ which fulfilles the properties
\begin{align*}
  \supp \varphi_0 \subseteq \overline{B_2(0)} \qquad \text{and} \qquad \supp \varphi_j \subseteq \{ \xi \in \Rn: 2^{j-1} \leq |\xi| \leq 2^{j+1}\} 
\end{align*}
for all $ j \in \N$. Here $B_r(0)$ denotes the open ball with radius $r>0$ and center $0$.
A dyadic partition of unity can be constructed in the following way: We take $\varphi_0 \in C^{\infty}(\Rn)$ with $\varphi_0(\xi) =1$ for all $|\xi| \leq 1$ and $\varphi_0(\xi) =0$ for $|\xi| \geq 2$. Then we set $\varphi_j(\xi):= \varphi_0(2^{-j}\xi) - \varphi_0(2^{-j+1} \xi)$ for all $\xi \in \Rn$ and $j \in \N$.\\

We also will need the next statement, cf. \cite[Lemma 2.1]{Paper1}:

\begin{lemma}\label{lemma:CharacterizationOfBesselPotentialSpaces}
  Let $1<p<\infty$, $s < 0$ and $m := -\lfloor s \rfloor$. Then for each $f \in H^s_p(\Rn)$ there are functions $g_{\alpha} \in H^{s-\lfloor s \rfloor}_p(\Rn)$, where $\alpha \in \Non$ with $|\alpha| \leq m$, such that
  \begin{itemize}
    \item $f= \sum \limits_{|\alpha| \leq m} \p^{\alpha}_x g_{\alpha}$, 
    \item $\sum \limits_{|\alpha| \leq m} \| g_{\alpha} \|_{H^{s-\lfloor s \rfloor}_p} \leq C \| f \|_{H^s_p}$,
  \end{itemize}
  where $C$ is independent of $f$, $g_{\alpha}$.
\end{lemma}

\subsection{Uniformly Local Sobolev Spaces} \label{subsection:HmqUlocAbschaetzung}

\begin{Def}
  Let $1 \leq q \leq \infty$, $m \in \N_0$, $U \subseteq \Rn$ be open
  and $X$ be a Banach space. Then the space of all functions, which belong \textit{uniformly local} to $L^q(U;X)$ or $W^m_q(U;X)$ is denoted by 
  \begin{align*}
      L^q_{uloc}(U;X) &:= \{ f \in L^q_{loc}(U;X) : \| f \|_{L^q_{uloc}(U;X)}  < \infty \}, \\ 
      W^{m,q}_{uloc}(U;X) &:= \{ f \in L^q_{uloc}(U;X) : \p_x^{\alpha} f \in
    L^q_{uloc}(U;X) \text{ for all } |\alpha| \leq m \},
  \end{align*}
  respectively, where 
  \begin{align*}
    \| f \|_{L^q_{uloc}(U;X)} &:= \sup_{x \in U} \| f \|_{L^q(B_1(x) \cap U;X)} &  &\text{for } f \in L^q_{uloc}(U; X), \\
    \| f \|_{W^{m,q}_{uloc}(U;X)} &:= \sum_{|\alpha| \leq m} \| \p_x^{\alpha} f \|_{L^{q}_{uloc}(U;X)}  &  &\text{for } f \in W^{m,q}_{uloc}(U; X).
  \end{align*}
  If $X=\C$, we write $L^q_{uloc}(U)$
  instead of $ L^q_{uloc}(U;\C)$ and $W^{m,q}_{uloc}(U)$ 
  instead of $W^{m,q}_{uloc}(U;\C)$. Moreover, we also write $\|. \|_{L^q_{uloc}}$ 
  and $\|. \|_{W^{m,q}_{uloc}}$
  instead of $ \| . \|_{L^q_{uloc}(\Rn;\C)} $ and $\| . \|_{W^{m,q}_{uloc}(\Rn;\C)}$.
\end{Def}

The spaces $L^q_{uloc}(U;X)$ and $W^{m,q}_{uloc}(U;X)$ are Banach spaces. This can be verified by using the fact that $L^q(U;X)$ is a Banach space.
The norm $\| . \|_{W^{m,q}_{uloc}(U;X)}$ is equivalent to the norm $\| . \|'_{W^{m,q}_{uloc}(U;X)} $ defined by 
$$\| f \|'_{W^{m,q}_{uloc}(U;X)}:= \max_{|\alpha| \leq m} \| \p_x^{\alpha} f \|_{L^{q}_{uloc}(U;X)} \qquad \text{for } f \in W^{m,q}_{uloc}(U;X) .$$

\begin{lemma}\label{lemma:HölderInequalityForLq_ulocSpaces} \label{lemma:SobolevHoelderInequalityFor_LqUlocSpaces}
  Let $1 < q \leq \infty$, $m \in \N_0$, $X$ be a Banach space and $0< \tau \leq m-n/q$ with $\tau \notin \N$. Additionally let $U \subseteq \Rn$ be open and $V\subseteq U$ be compact. Then
  \begin{enumerate}
    \item $\| f \|_{L^q(V;X)} \leq C \| f \|_{L^q_{uloc}(U;X)}  \text{ for all } f \in L^q_{uloc}(U;X)$,
    \item $W^{m}_{q}(U;X) \subseteq W^{m,q}_{uloc}(U;X)$, 
    \item $\| D_{x_j} f \|_{W^{m,q}_{uloc}(U;X) } \leq \| f \|_{W^{m+1,q}_{uloc}(U;X)} \text{ for all } f \in W^{m+1,q}_{uloc}(U;X)$,
    \item $L^r_{uloc}(U;X) \subseteq L^q_{uloc}(U;X)$ for all $q \leq r < \infty$.
	  The embedding also holds true for $q=1$.
    \item $W^{m,q}_{uloc}(\Rn; X) \hookrightarrow C^{\tau}(\Rn; X)$ is continuous.
  \end{enumerate}
\end{lemma}

\begin{proof}
  The properties (i)-(iii) immediately follow from the definition of the uniformly local Sobolev Spaces. An application of the H\"older inequality for bounded sets yields (iv). 
  
  In order to prove the continuous embedding (v) let $\alpha \in \Non$ with $|\alpha| \leq \lfloor \tau \rfloor$ be arbitrary. Then we obtain 
  \begin{align}\label{eq33}
    \sup_{x \neq y} \frac{\| \p^{\alpha}_xf(x)- \p^{\alpha}_y f(y)\|_{X}}{ |x-y|^{\tau- \lfloor \tau \rfloor}} 
    \leq \sup_{x \in \Rn} \|\p^{\alpha}_x f \|_{C^{\tau- \lfloor \tau \rfloor}(\overline{B_1(x)};X)} + 2 \|\p^{\alpha}_x f \|_{C_b^0(\Rn;X)}
  \end{align}
  if we split the supremum of the left side in the supremum over $|x-y|>1$ and the rest. Using inequality (\ref{eq33}) and $\|f\|_{C_b^{\lfloor \tau \rfloor}(\Rn;X) } \leq \sup_{x\in \Rn} \|f\|_{C^{\tau}(\overline{B_1(x)};X )}$ for all functions $f \in C^{\tau}(\Rn;X)$, we get 
  \begin{align}\label{eq:sofort}
    \|f\|_{C^{\tau}(\Rn;X)} 
    &\leq C \sup_{x \in \Rn} \|f\|_{C^{\tau}(\overline{B_1(x)};X)} \qquad \text{for all } f \in C^{\tau}(\Rn;X).
  \end{align}
  Making use of Corollary 4.3 in \cite{AmannPaper} and the embeddings (3.1)-(3.3) and (3.6) in \cite{AmannPaper} yields the continuous embedding $W^m_q(B_1(0); X) \hookrightarrow C^{\tau}(\overline{B_1(0)};X)$. Together with inequality (\ref{eq:sofort}) we obtain
  \begin{align*}
    \|f\|_{C^{\tau}(\Rn;X)} 
    &\leq \sup_{x \in \Rn} \|f\|_{C^{\tau}(\overline{B_1(x)};X)}
    = \sup_{x \in \Rn} \|f(x+.)\|_{C^{\tau}(\overline{B_1(0)};X)}\\
    &\leq C \sup_{x \in \Rn} \|f(x+.)\|_{W^{m,q}(B_1(0);X)}
    =C \|f\|_{W^{m,q}_{uloc}(\Rn;X)}\raisebox{-1mm}{.}
  \end{align*}
  \vspace*{-1cm}

\end{proof}

Now we want to discuss the case $X=L^q_{uloc}(\R^m; \C)$: 

\begin{prop}\label{prop:MessbarkeitBezglProduktmas}
  Let $1 \leq q < \infty$ and $m \in \N$. For $a \in
  L^q_{uloc}(\Rn; L^q_{uloc}(\R^m;\C))$ we define 
  $\tilde{a}:\R^m \times \Rn \rightarrow \C$ via 
  $$\tilde{a}(x,y):=a (y)(x) \qquad \text{for all } y \in \Rn \text{ and } x \in \R^m.$$
  Then $\tilde{a}$ is measurable with respect to the product measure
  of $\R^m \times \Rn$.
\end{prop}

\begin{proof}
  With the properties of measurability at hand, cf.\,e.\,g.\,Chapter 3 of \cite{Elstrodt}, one can prove the claim. For more details we refer to \cite[Proposition 4.16]{Diss}. 
\end{proof}

\begin{bem}\label{bem:EinbettungVonLq_ulocRaumen}
  Let $1 \leq q \leq r < \infty$ and $m \in \N$. Then the 
embedding 
  \begin{align*}
    L^q_{uloc}(\Rn; L^q_{uloc}(\R^m)) \hookrightarrow L^q_{uloc}(\Rn \times \R^m) \qquad \text{is continuous.}
  \end{align*}
\end{bem}

\begin{proof}
  Since $L^q_{uloc}(\R^m)$ is a Banach space, we obtain the claim by
Proposition \ref{prop:MessbarkeitBezglProduktmas}, the definition of
these spaces and by $B_1(x,y) \subseteq B_1(x) \times B_1(y)$ for all $x \in \R^m$ and $y \in \Rn$.
\end{proof}

\begin{lemma}\label{lemma:AbleitungWiederInWmq_uloc}
  Let $m, \tilde{m} \in \N_0$, $1 < q < \infty$ and $\alpha \in \Non$ with
$|\alpha| \leq \tilde{m}$.  Assuming $a \in  W^{m,q}_{uloc}(\Rn;
W^{\tilde{m},q}_{uloc} (\Rn))$, we obtain $$ \p_x^{\alpha} a \in
W^{m,q}_{uloc}(\Rn; W^{\tilde{m}-|\alpha|,q}_{uloc} (\Rn_x)).$$
\end{lemma}

This lemma is a direct consequence of the next proposition:

\begin{prop}\label{prop:TfInWmq_uloc}
  Let $1 < q < \infty$, $U \subseteq \Rn$ be an open subset and $X,Y$ be two
Banach spaces. Moreover, let $T: X \rightarrow Y$ be a linear bounded operator
and $f \in W^{m,q}_{uloc}(U,X)$. If we define $\tilde{T}f: U
\rightarrow Y$ by 
    $(\tilde{T}f)(x):= T(f(x))$ for all $ x \in U$,
  we obtain $\tilde{T}f \in W^{m,q}_{uloc}(U,Y)$.
\end{prop}

\begin{proof}
  Let $\alpha \in \Non$ with $|\alpha| \leq m$ and  $\varphi \in
C_c^{\infty}(U)$. 
  Due to Lemma \ref{lemma:HölderInequalityForLq_ulocSpaces} we have  
  $\left( \p_x^{\alpha} f \right) \varphi, f \left( \p_x^{\alpha} \varphi \right) \in L^1(U;X)$. 
  Together with $f \in W^{m,q}_{uloc}(U;X)$ and 
the linearity of $T$ we obtain
  \begin{align*}
    &\int\limits_U \tilde{T} \left( \p_x^{\alpha} f\right)(x) \varphi(x) dx 
    = T\int\limits_U \left( \p_x^{\alpha} f(x) \right) \varphi(x) dx 
    =(-1)^{|\alpha|} T\int\limits_U f(x) \p_x^{\alpha} \varphi(x) dx\\ 
    &\qquad =(-1)^{|\alpha|} \int\limits_U T \left[ f(x) \p_x^{\alpha} \varphi(x)
\right] dx 
    =(-1)^{|\alpha|} \int\limits_U (\tilde{T} f) (x) \p_x^{\alpha} \varphi(x) dx.
  \end{align*}
  Consequently the $\alpha$-th weak derivative of $\tilde{T}f$ exists and is given by $\p_x^{\alpha} (\tilde{T} f) =
\tilde{T}(\p_x^{\alpha} f)$. Hence $\tilde{T}f \in W^{m,q}_{uloc}(U;Y)$
because of
  \begin{align*}
    \| \tilde{T}f \|_{ W^{m,q}_{uloc}(U;Y)} 
    = \max_{|\alpha| \leq m} \| \tilde{T} ( \p_x^{\alpha} f ) \|_{ L^{q}_{uloc}(U;Y)} 
    \leq C \max_{|\alpha| \leq m} \sup_{x \in U } \int \limits_{B_1(x) \cap U} \hspace{-0.1cm} \| \p_x^{\alpha} f(x) \|_X dx 
    \leq C.
  \end{align*}
  \vspace*{-1cm}
  
\end{proof}

\begin{bem}\label{bem:FubiniBeiWmq_loc_Wmq_loc_Wertig}
    Let $1 < q < \infty$ and $\varphi, \psi \in
C^{\infty}_c(\Rn)$. Moreover, let $f \in L^{q}_{uloc}(\Rn_y;
L^{q}_{uloc} (\Rn_x))$. An easy consequence of Fubini's Theorem, Lemma \ref{lemma:HölderInequalityForLq_ulocSpaces} and Remark \ref{bem:EinbettungVonLq_ulocRaumen} is:
    \begin{align*}
      \iint f(x,y) \varphi(x) \psi(y) dx dy = \iint f(x,y) \varphi(x) \psi(y) dy
dx.
    \end{align*}
\end{bem}

\begin{bem}\label{bem:VertauschbarkeitDerAbleitungInHmq_uloc}
  Let $m, \tilde{m} \in \N_0$, $1 < q < \infty$ and $a \in  W^{m,q}_{uloc}(\Rn;
W^{\tilde{m},q}_{uloc} (\Rn))$. We choose  $\varf{\alpha}{l} \in \Non$ and
$\varf{\beta}{l} \in \Non$ with $|\alpha_1| + |\beta_1| = \ldots = |\alpha_l| +
|\beta_l| =1$, $|\alpha| \leq \tilde{m}$, $|\beta| \leq m$ and $l \in \N$ and define
$\alpha := \vara{\alpha}{l}$ and $\beta := \vara{\beta}{l}$. Then
  \begin{align*}
    \p_x^{\alpha_1} \p_y^{\beta_1} \ldots \p_x^{\alpha_l} \p_y^{\beta_l} a(x,y)
= \p_x^{\alpha} \p_y^{\beta} a(x,y) \qquad \text{for almost all } x,y \in \Rn.
  \end{align*}
\end{bem}

For the proof we refer to \cite[Remark 4.22]{Diss}.

\begin{lemma}\label{lemma:FunktionInWmq_ulocWmq_uloc_wertig}
  Let $m \in \N_0$, $1 < q < \infty$ and $N \in \N_0$. We consider a measurable function $a:\RnRn \rightarrow \C$ such that for each $\alpha
\in \Non$ with $|\alpha| \leq N$ we have
  \begin{itemize}
    \item $\p_y^{\alpha} a(., y) \in W^{m,q}_{uloc}(\Rn)$ for all $y \in \Rn$,
    \item $a(x,.) \in W^{N,q}_{uloc}(\Rn)$ for all $x \in \Rn$, 
    \item $\sup\limits_{y \in \Rn} \| \p_y^{\alpha} a(., y)
\|_{W^{m,q}_{uloc}(\Rn)} < C_{\alpha}$ for a constant $C_{\alpha}>0$.
  \end{itemize}
  Then $a \in W^{N,q}_{uloc}(\Rn_y; W^{m,q}_{uloc} (\Rn_x))$.
\end{lemma}

\begin{proof}
  Because of $a(x,.) \in W^{N,q}_{uloc}(\Rn)$,
  the $\alpha$-the weak derivative of $a$ in the sense of $\mathscr{D}'(\Rn_y; W^{m,q}_{uloc}(\Rn_x))$ exists for all $\alpha \in \Non$
with $|\alpha| \leq N$.  
  Hence the claim holds:
  \begin{align*}
    \| a \|_{ W^{N,q}_{uloc}(\Rn_y; W^{m,q}_{uloc} (\Rn_x)) } 
    &\leq \sum_{|\alpha | \leq N} \sup_{y \in \Rn} \left\{ \int_{B_1(y)} \left\|
\p_y^{\alpha} a(.,z) \right\|^q_{  W^{m,q}_{uloc} (\Rn) } dz
\right\}^{1/q} \\
    &\leq \sum_{|\alpha | \leq N} C_{\alpha, q} \sup_{y \in \Rn}
|B_1(y)|^{1/q} \leq C_{N,q,n}.
  \end{align*}
  \vspace*{-1cm}

\end{proof}

With all these results at hand, we are able to show: 

\begin{lemma}\label{lemma:AbleitungenInLq_uloc}
  Let $1 < q < \infty$ and $\tilde{m}, N \in \N_0$. Furthermore, let $\mathscr{B}$ be a set of all measurable functions $a: \RnRnRn
\rightarrow \C$ such that 
  \begin{itemize}
    \item $\p_y^{\alpha} a(., \xi, y) \in W^{\tilde{m},q}_{uloc}(\Rn)$ for all
$\xi, y \in \Rn$ and each $\alpha \in \Non$ with $|\alpha| \leq N$,
    \item $a(x,\xi,.) \in W^{N, q}_{uloc}(\Rn)$ for all $x, \xi \in \Rn$.
  \end{itemize}
  Additionally let $m \in \N_0$ such that 
  for every $\alpha \in \Non$ with $|\alpha| \leq N$ we have
  \begin{align*}
    \sup_{y\in \Rn} \| \p_y^{\alpha} a(., \xi, y)
\|_{W^{\tilde{m},q}_{uloc}(\Rn)} \leq C_{\alpha, q} \<{\xi}^m \qquad \text{for all } \xi \in \Rn, a \in\mathscr{B}.
  \end{align*}
If we define $b: \RnRnRn \rightarrow \C$ for a fixed but arbitrary $a \in
\mathscr{B}$ by
  \begin{align*}
    b(x, \xi, y) := a(x, \xi, x+y) \qquad \text{for all } x, \xi, y \in \Rn, 
  \end{align*}
  we get for each $\alpha, \beta \in \Non$ with $|\beta| \leq \tilde{m}$ and $|\alpha| +
|\beta| \leq N$ and for all $\xi \in \Rn$:
  \begin{align*}
    \| \p_y^{\alpha} \p_x^{\beta} b(x, \xi, y) \|_{L^q_{uloc}(\RnRnx{x}{y}) } <
C_{\alpha, q} \<{\xi}^m,
  \end{align*}
  where $C_{\alpha, q}$ is independent of $a \in \mathscr{B}$ and $\xi \in \Rn$.
\end{lemma}

\begin{proof}
  Lemma \ref{lemma:FunktionInWmq_ulocWmq_uloc_wertig} and Remark
\ref{lemma:AbleitungWiederInWmq_uloc} imply 
  $$\p_y^{\alpha} \p_x^{\beta} a (x,\xi,y) \in W^{N-|\alpha|,q}_{uloc}
\left(\Rn_y; W^{\tilde{m}-|\beta|, q}_{uloc}(\Rn_x) \right) \subseteq
L^q_{uloc}(\Rn_y; L^q_{uloc}(\Rn_x))$$
  for all $\xi \in \Rn$ and $\alpha, \beta \in \Non$ with $|\alpha| \leq N$ and
$|\beta| \leq \tilde{m}$. On account of Proposition
\ref{prop:MessbarkeitBezglProduktmas} we derive the measurability of
$\p_y^{\alpha} \p_x^{\beta} a (x,\xi,y)$ with respect to the product measure of $\RnRnx{x}{y}$ for a fixed $\xi \in \Rn$. 
  Using Tonelli's theorem twice and substituting $\hat{y}:= z + \tilde{y}$,
we obtain for each  $\alpha, \beta \in \Non$ with $|\alpha|
\leq N$ and $|\beta| \leq \tilde{m}:$
  \begin{align*}
    &\int \limits_{B_1(x,y)} \left| \left(\p_y^{\alpha} \p_x^{\beta} a\right)
(z, \xi, z+ \tilde{y}) \right|^q d(z, \tilde{y}) 
    \leq \int \limits_{B_1(x) \times B_1(y)} \left| \left(\p_y^{\alpha}
\p_x^{\beta} a\right) (z, \xi, z+ \tilde{y}) \right|^q d(z, \tilde{y}) \\
    &\qquad = \int \limits_{B_1(x)} \int \limits_{B_1(y)} \left|
\left(\p_y^{\alpha} \p_x^{\beta} a\right) (z, \xi, z+ \tilde{y}) \right|^q d
\tilde{y} \; dz
    = \int \limits_{B_1(x)} \int \limits_{B_1(y+z)} \left|
\left(\p_y^{\alpha} \p_x^{\beta} a\right) (z, \xi, \hat{y}) \right|^q d \hat{y}
\; dz \\
    &\qquad \leq \int \limits_{B_1(x)} \int \limits_{B_2(y+x)} \left|
\p_{\hat{y}}^{\alpha} \p_z^{\beta} a (z, \xi, \hat{y}) \right|^q d \hat{y} \; dz
    = \int \limits_{B_2(y+x)} \int \limits_{B_1(x)} \left|
\p_{\hat{y}}^{\alpha} \p_z^{\beta} a (z, \xi, \hat{y}) \right|^q dz  \; d
\hat{y} \\
    &\qquad \leq \int \limits_{B_2(y+x)} \left\| \p_{\hat{y}}^{\alpha} \p_x^{\beta} a
(x, \xi, \hat{y}) \right\|^q_{L^q_{uloc}(\Rn_x)}  \; d \hat{y} 
   \leq \hspace{-0.1cm} \int \limits_{B_2(y+x)} \sup_{\hat{y} \in \Rn}
\left\| \p_{\hat{y}}^{\alpha} a (x, \xi, \hat{y})
\right\|^q_{W^{\tilde{m},q}_{uloc}(\Rn_x)}  \; d \hat{y} \\
     &\qquad \leq C_{\alpha, q, n} \<{\xi}^m,
  \end{align*}
 for all $x,\xi, y \in \Rn$ and $a \in
\mathscr{B}$. 
Finally, let $\alpha, \beta \in \Non$ with $|\beta| \leq \tilde{m}$ and $|\alpha|+|\beta| \leq N$ be arbitrary. 
An application of Remark \ref{bem:VertauschbarkeitDerAbleitungInHmq_uloc},
the Leibniz rule and the previous inequality provides:
  \begin{align*}
    \|\p_y^{\alpha} \p_x^{\beta} b(x, \xi, y)\|_{L^q_{uloc}(\RnRnx{x}{y})} 
    &\leq \hspace{-0.27cm} \sum_{\beta_1 + \beta_2 = \beta} \sup_{x,y \in \Rn}
\left\{ \int_{B_1(x,y)} \hspace{-0.5cm} \left| \left( \p_y^{\alpha + \beta_1} \p_x^{\beta_2} a
\right) (z, \xi, z+ \tilde{y}) \right|^q \right\}^{1/q} \\
    &\leq C_{\alpha, \beta, q, n} \<{\xi}^m \qquad \text{for all } \xi \in \Rn \text{ and } a \in \mathscr{B}.
  \end{align*}
  \vspace*{-1cm}

\end{proof}

\begin{lemma}\label{lemma:AbschaetzungInHmq_uloc}
  Let $1 < q < \infty$, $m \in \R$ and $\tilde{m} \in \N_0 $
with $\tilde{m}>n/q$. Moreover, let $\mathscr{B}$ be a set of measurable
functions $a: \RnRnRn \rightarrow \C$ with the following property: 
  \begin{itemize}
    \item $\p_y^{\alpha} a(., \xi, y) \in W^{\tilde{m},q}_{uloc}(\Rn)$ for all
$\xi, y \in \Rn$ and each $\alpha \in \Non$ with $|\alpha| \leq 2\tilde{m}$,
    \item $a(x,\xi,.) \in W^{2\tilde{m}, q}_{uloc}(\Rn)$ for all $x, \xi \in
\Rn$.
  \end{itemize}
  Additionally we assume that for all $\alpha \in \Non$ with $|\alpha| \leq 2\tilde{m}$ there is a constant $C_{\alpha, q}$ such that
  \begin{align*}
    \sup_{y\in \Rn} \| \p_y^{\alpha} a(., \xi, y)
\|_{W^{\tilde{m},q}_{uloc}(\Rn)} \leq C_{\alpha, q} \<{\xi}^m \qquad \text{for all } \xi \in \Rn, a\in \mathscr{B}.
  \end{align*}
Then we have for some $C_{m, q} < \infty$:
  \begin{align*}
    \sup_{y\in \Rn} \| a(x, \xi, x+y) \|_{W^{\tilde{m},q}_{uloc}(\Rn_x)} \leq
C_{m, q} \<{\xi}^m \qquad \text{for all } \xi \in \Rn, a\in \mathscr{B}.
  \end{align*}
\end{lemma}

\begin{proof}
  First of all 
  we choose a finite cover
$(U_i)_{i=1}^N$ of $B_2(0,0)$ with open balls of radius 1. Denoting for
all $i \in \{ 1, \ldots, n\}$ and each $x,y \in \Rn$ the set $U_i(x,y)$ as the
translation of $U_i$ by $(x,y)$, $(U_i(x,y))_{i=1}^N$ is a finite cover
of $B_2(x,y)$ with open balls of radius 1. For all $\alpha, \beta
\in \Non$ with $|\beta| \leq 2 \tilde{m}$ and $|\alpha| \leq \tilde{m}$ we obtain:
  \begin{align}\label{p67}
    &\sup_{x,y \in \Rn} \left[ \int_{B_2(x,y)} \left| \p^{\beta}_{\tilde{y}} \p^{\alpha}_z a(z,\xi, z+ \tilde{y}) \right|^q d(\tilde{y},z) \right]^{1/q} \notag \\ 
    &\qquad \qquad \qquad \leq C_q \sum_{i=1}^N \sup_{x,y \in \Rn} \left[ \int_{U_i(x,y)} \left| \p^{\beta}_{\tilde{y}} \p^{\alpha}_z a(z,\xi, z+ \tilde{y}) \right|^q d(\tilde{y},z) \right]^{1/q} \notag \\ 
    &\qquad \qquad \qquad  \leq C_q \|\p^{\beta}_{y} \p^{\alpha}_x a(x,\xi, x+ y) \|_{L^q_{uloc}(\Rn_x \times \Rn_{y})} \text{ for all } a \in \mathscr{B}, \xi \in \Rn.
  \end{align}
  Due to Lemma \ref{lemma:FunktionInWmq_ulocWmq_uloc_wertig}  and Remark
\ref{lemma:AbleitungWiederInWmq_uloc} we get
  $$\p_y^{\alpha} \p_x^{\beta} a (x,\xi,y) \in W^{2\tilde{m}-|\alpha|,q}_{uloc}
\left(\Rn_y; W^{\tilde{m}-|\beta|, q}_{uloc}(\Rn_x) \right) \subseteq
L^q_{uloc}(\Rn_y; L^q_{uloc}(\Rn_x))$$
  for all $\xi \in \Rn$ and $\alpha, \beta \in \Non$ with $|\alpha| \leq 2
\tilde{m}$ and $|\beta| \leq \tilde{m}$. This implies the measurability of
$\p_y^{\alpha} \p_x^{\beta} a (x,\xi,y)$ with respect to the product measure of
$\RnRnx{x}{y}$ for every fixed $\xi \in \Rn$ as stated in Proposition
\ref{prop:MessbarkeitBezglProduktmas}. 
 We define 
$b(x, \xi, y) := a(x, \xi, x+y)$ for all $x, \xi, y \in \Rn$. Using the Sobolev
embedding theorem and Tonelli's theorem, we
obtain for each $\alpha \in \Non$ with $|\alpha| \leq \tilde{m}$:
  \begin{align}\label{p68}
    \int \limits_{B_1(x)} \sup_{\tilde{y} \in B_1(y) } \left| \p_z^{\alpha}
b(z,\xi, \tilde{y}) \right|^q \; d z 
    &\leq C_q \int \limits_{B_1(x)}  \left\| \p_z^{\alpha} b(z,\xi, .)
\right\|^q_{ W^{\tilde{m} }_q( B_1(y) ) } d z \notag\\
    &\leq C_q \sum_{|\beta| \leq \tilde{m}} \; \int \limits_{B_1(x) \times B_1(y)}  \left|
\p^{\beta}_{\tilde{y}} \p_z^{\alpha} b(z,\xi, \tilde{y}) \right|^q d (\tilde{y},
 z) \notag\\
    &\leq C_q \sum_{|\beta| \leq \tilde{m}} \; \int \limits_{B_2(x,y)}  \left|
\p^{\beta}_{\tilde{y}} \p_z^{\alpha} b(z,\xi, \tilde{y}) \right|^q d (\tilde{y},
 z)
  \end{align}  
  for all $a \in \mathscr{B}$ and $x,y,\xi \in \Rn$. 
Therefore (\ref{p68}),
(\ref{p67}) and Lemma \ref{lemma:AbleitungenInLq_uloc} yield
  \begin{align*}
    &\sup_{y \in \Rn} \| a(x, \xi, x+y) \|_{W^{\tilde{m},q}_{uloc}(\Rn_x)} 
    \leq \sum_{|\alpha| \leq \tilde{m}} \sup_{y \in \Rn}  \left\| \p_x^{\alpha} b(x,
\xi, y) \right\|_{L^{q}_{uloc}(\Rn_x)} \\
    &\qquad \leq \sum_{|\alpha| \leq \tilde{m}} \sup_{x,y \in \Rn} \left\{ \int_{B_1(x)}
\sup_{\tilde{y} \in B_1(y)} \left|\p_z^{\alpha} b(z, \xi, \tilde{y}) \right|^{q} dz
\right\}^{1/q} \\
    &\qquad \leq C_q \sum_{|\alpha| \leq \tilde{m}} \sum_{|\beta| \leq \tilde{m}} \sup_{x,y \in
\Rn} \left\{ \int_{B_2(x,y)}  \left| \p^{\beta}_{\tilde{y}} \p_z^{\alpha}
b(z,\xi, \tilde{y}) \right|^q d (\tilde{y},  z) \right\}^{1/q} \\
    &\qquad \leq C_q \sum_{|\alpha| \leq \tilde{m}} \sum_{|\beta| \leq \tilde{m}}  \left\|
\p^{\beta}_{y} \p_x^{\alpha} b(x,\xi, y)
\right\|_{L^q_{uloc}(\RnRnx{x}{y})} 
    \leq C_{m,q,n} \<{\xi}^m  
  \end{align*}
  for all $a \in \mathscr{B}$ and $\xi \in \Rn$.
\end{proof}

\subsection{Extension of the Space of Amplitudes}\label{section:ExtensionSpaceOfAmplitudes}

While verifying results in the field of pseudodifferential operators one often uses the properties of oscillatory integrals defined by
\begin{align*}
  \osint e^{-iy \cdot \eta} a(y,\eta) dy \dq \eta := \lim_{\e \rightarrow 0} \iint \chi(\e y, \e \eta) e^{-iy \cdot \eta} a(y,\eta) dy \dq \eta
\end{align*}
for all elements $a$ of the \textit{extension of the space of amplitudes} $\mathscr{A}^{m,N}_{\tau}(\RnRn)$ $(m,\tau \in \R, N \in \N_0 \cup \{ \infty \})$, the set of all functions $a:\Rn \times \Rn \rightarrow \C$ with the following properties: For all
$\alpha, \beta \in \Non$ with $|\alpha| \leq N$ we have
  \begin{enumerate}
    \item[i)] $\p^{\alpha}_{\eta} \p^{\beta}_{y} a(y,\eta) \in C^0(\RnRnx{y}{\eta})$,
    \item[ii)] $\left|\p^{\alpha}_{\eta} \p^{\beta}_{y} a(y, \eta) \right| \leq C_{\alpha, \beta} (1 + |\eta|)^m (1 + |y|)^{\tau}$ for all $y, \eta \in \Rn$,
  \end{enumerate} 

If $N=\infty$ we also write $\mathscr{A}^{m}_{\tau}(\RnRn)$ instead of $\mathscr{A}^{m,\infty}_{\tau}(\RnRn)$.\\

Now we summarize the properties of the oscillatory integral we need later on. For more details we refer to \cite[Subsection 2.3]{Paper1}. In the following we use for all $m \in \N$ the next definition:
\begin{align*}
  A^m(D_{x},\xi) &:= \<{\xi}^{-m} \<{D_x}^{m} \quad &\text{ if } m \text{ is even},\\
  A^m(D_{x},\xi) &:= \<{\xi}^{-m-1} \<{D_x}^{m-1} -\sum_{j=1}^n \<{\xi}^{-m} \frac{\xi_j}{\<{\xi} } \<{D_x}^{m-1} D_{x_j} \quad &\text{ else}.
\end{align*}

\begin{thm}\label{thm:ExistenceOfOscillatoryIntegral}
  Let $m, \tau \in \R$ and $N \in \N_0 \cup \{ \infty\}$ with $N  > n+ \tau$. Moreover, let $\chi \in
\mathcal{S}(\RnRn)$ with $\chi(0,0)=1$ be arbitrary. Then the
\textbf{oscillatory integral}
  \begin{align*}
    \osint e^{-iy \cdot \eta} a(y,\eta) dy \dq \eta := \lim_{\e \rightarrow 0}
\iint \chi(\e y, \e \eta) e^{-iy \cdot \eta} a(y,\eta) dy \dq \eta
  \end{align*}
  exists for each $a \in \mathscr{A}^{m,N}_{\tau}(\RnRn)$. Additionally for all $l,l' \in
\N_0$ with $l > n+m$ and $N \geq l' > n + \tau$ we have
  \begin{align*}
    \osint e^{-iy \cdot \eta} a(y,\eta) dy \dq \eta = \iint e^{-iy \cdot \eta}
A^{l'}(D_{\eta},y) [ A^{l}(D_{y},\eta) a(y,\eta) ]
dy \dq \eta.
  \end{align*}
  Therefore the definition does not depend on the choice of $\chi$.
\end{thm}

\begin{thm}\label{thm:VertauschenVonOsziIntUndAbleitungen}
  Let $m,\tau \in \R$, $N \in \N_0 \cup \{ \infty \}$ and $k \in  \N$ with $N  > k + \tau$. We set $\tilde{\tau}:= \tau$ if $\tau \geq -k$, $\tilde{\tau}:= -k-0.5$ if $\tau \in \Z$ and $\tau < -k$ and $\tilde{\tau}:= -k-(|\tau| - \lfloor -\tau \rfloor)/2$ else. We define $\hat{\tau}:= \tau_+$ if $\tau \geq -k$ and $\hat{\tau}:= \tau- \tilde{\tau}$ else. For $a \in \mathscr{A}^{m,N}_{\tau}(\R^{n+k} \times
\R^{n+k})$ we define 
  \begin{align*}
    b(y,\eta) := \osint e^{-iy' \cdot \eta'} a(y,y',\eta,\eta') dy' \dq \eta' \qquad \text{for all } y, \eta \in \Rn.
  \end{align*}
  Let $M:= \max\{ m \in \N_0: N-m \geq l > k+ \tilde{\tau} \text{ for one } l \in \N_0 \}$. Then $b$ is an element of $\mathscr{A}^{m_+, M}_{\hat{\tau}}(\R^{n} \times \R^{n})$  and for each  $\alpha, \beta \in \Non$ with $|\beta| \leq M$ we have:
  \begin{align*}
    \p_y^{\alpha} \p_{\eta}^{\beta} b(y,\eta) = \osint  e^{-iy' \cdot \eta'}
\p_y^{\alpha} \p_{\eta}^{\beta} a(y,y',\eta,\eta') dy' \dq \eta' \qquad \text{for all } y, \eta \in \Rn.
  \end{align*}
\end{thm}

\begin{thm}\label{thm:OscillatoryIntegralGleichung}
 Let $m,\tau \in \R$ and $N \in  \N_0 \cup \{ \infty \}$ with $N > n + \tau$. 
 Moreover, let $l_0, \tilde{l}_0 \in \N_0$ with $\tilde{l}_0 \leq N$. Then
 \begin{align*}
    \osint e^{-iy \cdot \eta} a(y,\eta) dy \dq \eta = \osint e^{-iy \cdot \eta}
 A^{\tilde{l}_0}(D_{\eta},y) A^{l_0}(D_y, \eta) a(y,\eta) 
dy \dq \eta
 \end{align*}
 for every $a \in \mathscr{A}^{m, N}_{\tau}(\R^{n} \times \R^{n})$.
\end{thm}

\section{Pseudodifferential Operators} \label{Kapitel: PDO}

During the whole section we assume $X^{\tilde{m}}_q \in \left\{ W^{\tilde{m},q}_{uloc}; H^{\tilde{m}}_q \right\}$ for $1<q<\infty$, $\tilde{m} \in \R$ if $X^{\tilde{m}}_q = H^{\tilde{m}}_q$ and $\tilde{m} \in \N_0$ else unless otherwise noted.
\subsection{Pseudodifferential Operators with Single Symbols} \label{Section:PropertiesOfPDO} 
The non-smooth symbol-class with coefficients in 
$X^{\tilde{m}}_q$ was already introduced in \cite{Marschall2}. 
Analogous to the definition of $C^{m,s}\Sallgn{\tilde{m}}{\rho}{\delta}{n}{n}{M}$ we define

\begin{Def}
	Let $1 < q < \infty$, $m \in \R$, $\tilde{m} \in \R$ if $X^{\tilde{m}}_q = H^{\tilde{m}}_q$ and $\tilde{m} \in \N_0$ else with $\tilde{m}>n/q$. Moreover, let $M \in \N_0 \cup \{ \infty \}$ and $0 \leq \rho \leq 1$. Then the \textit{symbol-class} $X^{\tilde{m}}_{q} \Sallgn{m}{\rho}{0}{n}{n}{M}$
	is the set of all functions $p:\R^n_x \times\R^n_{\xi} \rightarrow \C$ such that
	\begin{itemize}
		\item[i)] $\p_x^{\beta} p(x, .) \in C^M(\Rn)$ for all $x \in \Rn$,
		\item[ii)] $\p_x^{\beta} \pa{\alpha} p \in C^{0}(\R^n_x \times \R^n_{\xi})$,
		\item[iii)] $\p^{\alpha}_{\xi} p(.,\xi) \in X^{\tilde{m}_q}$ and $\| \pa{\alpha} p(.,\xi)  \|_{X^{\tilde{m}}_{q}} \leq C_{\alpha}\<{\xi}^{m-\rho|\alpha|}$ for all $\xi \in \R^n$
	\end{itemize}
  holds for all $\alpha, \beta \in \N_0^n$ with $|\alpha| \leq M$ and $|\beta| < \tilde{m} -n/q$. 
  The function $p$ is called 
  \textit{(non-smooth) symbol} and $m$ is called \textit{order} of $p$. If $M= \infty$, the symbols are smooth in $\xi$. In this case we write $X^{\tilde{m}}_{ q} \Sallg{m}{\rho}{0}{n}{n}$ instead of $X^{\tilde{m}}_{q}\Sallgn{m}{\rho}{0}{n}{n}{ \infty }$.
\end{Def}

We define for all $k \in \N$ with $k \leq M$ the semi-norms
  \begin{align*}
    |a|^{(\tilde{m})}_k:= \sup_{\xi \in \R^n} \max_{|\alpha| \leq k} \| \pa{\alpha} p(.,\xi)  \|_{X^{\tilde{m}}_{q}} \<{\xi}^{-\tilde{m}+\rho|\alpha|} 
  \end{align*}
  for all $a \in X^{\tilde{m}}_{ q}\Sallgn{\tilde{m}}{\rho}{0}{n}{n}{ M }$. Equipped with the family of semi-norms $(|.|^{(\tilde{m})}_{k})_{k \in \{0,\ldots, M\}}$ the symbol-class 
  $X^{\tilde{m}}_{ q}\Sallgn{\tilde{m}}{\rho}{0}{n}{n}{ M }$ is a topological vector space.\\

Using Lemma \ref{lemma:SobolevHoelderInequalityFor_LqUlocSpaces} if $X^{\tilde{m}}_{ q} = W^{\tilde{m},q}_{uloc}$ and the continuous embedding $H^{\tilde{m}}_q(\Rn) \hookrightarrow C^{\tau}(\Rn)$ else, which is an consequence of Corollary 6.13 in \cite{PDO}, Lemma 6.5 in \cite{PDO}, Theorem 6.15 in \cite{PDO} and Remark 6.4 in \cite{PDO}, we obtain:

\begin{lemma}\label{lemma:einbettungSymbolklassen}
  Let $1 < q < \infty$, $m \in \R$ and $\tilde{m} \in \R$ if $X^{\tilde{m}}_q = H^{\tilde{m}}_q$ and $\tilde{m} \in \N_0$ else with $\tilde{m}>n/q$. Moreover, let $M \in \N_0 \cup \{ \infty \}$ and $0 \leq \rho \leq 1$. Assuming $0 < \tau \leq \tilde{m}-n/q$, $\tau \notin \N$ we have
  \begin{align*}
    X^{\tilde{m}}_{q} \Sallgn{m}{\rho}{0}{n}{n}{M} \subseteq C^{\tau} \Sallgn{m}{\rho}{0}{n}{n}{M}.
  \end{align*}
\end{lemma}

Due to the last lemma we already defined
the associated pseudodifferential operator $p(x,D_x)$ to a non-smooth symbol $p \in X^{\tilde{m}}_{ q} \Sallgn{m}{\rho}{0}{n}{n}{M}$. The set of all non-smooth pseudodifferential operators with symbols in $X^{\tilde{m}}_{q}\Sallgn{m}{\rho}{0}{n}{n}{M}$ is denoted by $\op X^{\tilde{m}}_{q}\Sallgn{m}{\rho}{0}{n}{n}{M}$. 
If $M=\infty$ we write $\op X^{\tilde{m}}_{q}\Sallg{m}{\rho}{0}{n}{n}$ 
instead of $\op X^{\tilde{m}}_{q}\Sallgn{m}{\rho}{0}{n}{n}{\infty}$.\\

The iterated commutators of a non-smooth pseudodifferential operator with coefficients in $X^{\tilde{m}}_{q}$ are also non-smooth pseudodifferential operators. For the proof we refer to \cite[Remark 4.28]{Diss} and \cite[Remark 4.31]{Diss}.

\begin{bem}\label{bem:SymbolOfIteratedCommutatorNonSmoothInUniformlyLocallySobolevSpaces} \label{bem:SymbolOfIteratedCommutatorNonSmoothInSobolevSpaces}
  Let $1 < q < \infty$, $M \in \N_0 \cup \{ \infty\}$ and $\tilde{m} \in \R$ if $X^{\tilde{m}}_q = H^{\tilde{m}}_q$ and $\tilde{m} \in \N_0$ else with $\tilde{m}>n/q$. Additionally let $m\in \R$, $0 \leq \rho \leq 1$ and $ p \in X^{\tilde{m}}_{q} S^m_{\rho,0} (\RnRn; M)$. Moreover, let $l \in \N$, $\alpha_1, \ldots, \alpha_l \in \Non$ and $\beta_1, \ldots, \beta_l \in \Non$ with $|\alpha_j + \beta_j| = 1$ for all $j \in \{ 1, \ldots, n\}$, $|\alpha| \leq M$ and $|\beta| < \tilde{m} - n/q$, where 
  $\alpha:= \alpha_1 + \ldots + \alpha_l$ and $\beta := \beta_1 + \ldots + \beta_l$. Then
  $$\ad(-ix)^{\alpha_1} \ad(D_x)^{\beta_1} \ldots \ad(-ix)^{\alpha_l} \ad(D_x)^{\beta_l} p(x,D_x)$$ 
  is a pseudodifferential operator with the symbol 
  $$\pa{\alpha} D^{\beta}_x  p(x,\xi) \in X^{\tilde{m}- |\beta|}_{q} S^{m-\rho |\alpha| }_{\rho,0} ( \RnRnx{x}{\xi}; M-|\alpha|).$$
\end{bem}


In applications to partial differential equations, many pseudodifferential operators are classical ones. They are defined in the following way:

\begin{Def}
  Let $m \in \R$, $1<q<\infty$ and $\tilde{m} \in \N_0$ with $\tilde{m}>n/q$. 
  Then $p \in W^{\tilde{m},q}_{uloc} S^m_{1,0}(\RnRn)$ is a \textit{classical symbol of the order m} if $p$ has a asymptotic expansion
  \begin{align*}
    p(x, \xi) \sim \sum_{j \in \N_0} p_j(x, \xi),
  \end{align*}
  where $p_j$ are homogeneous of degree $m-j$ in $\xi$ (for $|\xi| \geq 1$ ) for all $j\in \N_0$
  in the sense, that for all $N \in \N$ we have 
  \begin{align*}
    p(x,\xi) - \sum_{ j < N} p_j(x, \xi) \in W^{\tilde{m},q}_{uloc} S^{m-N}_{1,0}(\RnRn).
  \end{align*}
  The set of all classical symbols of the order m is denoted by $W^{\tilde{m},q}_{uloc} S^m_{cl}(\RnRn)$.
\end{Def}

For a more general definition we refer to \cite{Taylor2}.
As an immediate consequence of the previous definition we obtain:
$$W^{\tilde{m},q}_{uloc} S^m_{cl}(\RnRn) \subseteq W^{\tilde{m},q}_{uloc} S^m_{1,0}(\RnRn).$$ 
By means of Remark \ref{bem:SymbolOfIteratedCommutatorNonSmoothInUniformlyLocallySobolevSpaces} and the definition of classical symbols we are able to verify the next remark. We refer to \cite[Remark 4.35]{Diss}for more details.

\begin{bem}\label{bem:SymbolOfIteratedCommutatorNonSmoothInUniformlyLocallySobolevSpacesClassical}
  Let $1 < q < \infty$, $\tilde{m} \in \N_0$ with $\tilde{m}>n/q$ and $m\in \R$. We assume $ p \in W^{\tilde{m}, q}_{uloc} S^m_{cl} (\RnRn)$. Moreover, let $l \in \N$, $\alpha_1, \ldots, \alpha_l \in \Non$ and $\beta_1, \ldots, \beta_l \in \Non$ with $|\alpha_j + \beta_j| = 1$ for all $j \in \{ 1, \ldots, n\}$ and $|\beta| < \tilde{m} - n/q$. Here 
  $\alpha:= \alpha_1 + \ldots + \alpha_l$ and $\beta := \beta_1 + \ldots + \beta_l$. Then
  $$\ad(-ix)^{\alpha_1} \ad(D_x)^{\beta_1} \ldots \ad(-ix)^{\alpha_l} \ad(D_x)^{\beta_l} p(x,D_x)$$ 
  is a pseudodifferential operator with the symbol 
  $$\pa{\alpha} D^{\beta}_x  p(x,\xi) \in W^{\tilde{m}- |\beta|, q}_{uloc} S^{m- |\alpha| }_{cl} (\Rn_x \times \Rn_{\xi}).$$
\end{bem}


Smooth pseudodifferential operators are bounded as maps between several function spaces, cf. Theorem 3.6 in \cite{PDO}, Theorem 6.19 in \cite{PDO} and Theorem 5.20 in \cite{PDO}.

\begin{thm} \label{thm:stetigAufS} \label{thm:stetigInHolderZygmundRaum} \label{thm:stetigInBesselPotRaum}
  Let $p \in \Sn{m}{1}{0}$ with $m \in \R$ and $1<q<\infty$. Then $p(x,D_x)$ extends to a bounded linear operator
  \begin{itemize}
    \item $p(x,D_x): \s \rightarrow \s$,
    \item $p(x,D_x): C_*^{s+m}(\Rn) \rightarrow C^s_*(\Rn)$ for all $ s >0$  with $ s+m>0$,
    \item $p(x,D_x): H_q^{s+m}(\Rn) \rightarrow H^s_q(\Rn)$ for all $ s \in \R$.
  \end{itemize}
  More precisely, for every $k \in \N_0$ there is a constant $C_k>0$ such that
  \begin{align*}
    |p(x,D_x)f|_{k,\mathcal{S}} \leq C_k |p|^{(m)}_k |f|_{\tilde{m}, \mathcal{S}} \qquad \textrm{for all } f \in \s,\\
  \end{align*}
  where 
  $\tilde{m}:= \max \{ 0, m + 2(n+1) + k\}$ if $m \in \Z$ and $\tilde{m}:= \max \{ 0, \lfloor m \rfloor + 2n+3 + k\}$ else.
  Moreover there are some $C_{s,m,q}>0$ and $k \in \N_0$, independent of $p$, such that
  \begin{align*}
    \| p(x,D_x) \|_{ \mathscr{L} (H_q^{s+m}, H^s_q) } \leq C_{s,m,q} |p|^{(m)}_k \qquad \textrm{for all } s \in \R.
  \end{align*}
\end{thm}

The previous theorem enables us to prove a characterization of the Bessel potential spaces. The last missing piece towards this result is

\begin{prop}\label{prop:CommutatorEstimationOfPartitionOfUnityWithDerivative}
  Let $1 < p < \infty$ and $s <0$. We consider a partition of unity $( \psi_j )_{j \in \Z^n} \subseteq C^{\infty}_c(\Rn)$ with
    $\psi_j(x) = \psi_0(x-j)$ for all $x\in \Rn$ and  $j \in \Z^n$.
  Then for all $\alpha \in \Non$ with $|\alpha| \leq -\lfloor s \rfloor$ we have
  \begin{align}\label{p106}
    \| [\p_x^{\alpha}, \psi_j] f \|_{H^s_p} \leq C \| f \|_{ H^{s-\lfloor s \rfloor}_p } \qquad \text{for all } f \in H^{s-\lfloor s \rfloor}_p (\Rn) \text{ and } j \in \Z^n.
  \end{align}
\end{prop}

\begin{proof}
  With $\p_x^{\alpha} \in \op \Sn{|\alpha|}{1}{0}$ and the boundedness of $\{ \psi_j \}_{j \in \Z^n} \subseteq \Sn{0}{1}{0}$ at hand the claim is a consequence of Theorem \ref{thm:stetigInBesselPotRaum}. For more details we refer to \cite[Proposition 3.19]{Diss}. 
\end{proof}

\begin{prop}\label{prop:NormequivalenceOfSobolevSpace}
  Let $1 < p < \infty$ and $s \in \R$. Moreover, let $( \psi_j )_{j \in \Z^n} \subseteq C^{\infty}_c(\Rn)$ be a partition of unity with
  \begin{itemize}
   \item $\psi_0 (x) =1 $ for all $x \in [0,1]^n$,
   \item $\psi_j(x) = \psi_0(x-j)$ for all $x\in \Rn$ and $j \in \Z^n$.
  \end{itemize} 
  Then we obtain
  \begin{align*}
    \| f \|_{H^s_p(\Rn)} \simeq \left( \sum_{j \in \Z^n} \| \psi_j f \|_{H^s_p(\Rn)}^p \right)^{1/p}.
  \end{align*}
\end{prop}

\begin{proof}
  The case $s \geq 0$ follows directly from \cite[Theorem 1.3]{Marschall2}. Therefore let $s<0$. In order to show 
  \begin{align}\label{p107}
    \| f \|_{H^s_p(\Rn)} \leq C \left( \sum_{j \in \Z^n} \| \psi_j f \|_{H^s_p(\Rn)}^p \right)^{1/p} \qquad \text{for all } f \in  H^s_p(\Rn)
  \end{align}
  let $f \in H^s_p(\Rn)$ and $g \in H^{-s}_q(\Rn)$ with $ 1/p + 1/q = 1$ be arbitrary. We define 
  \begin{align*}
    \eta_0 = \sum_{k \in Z} \psi_k, \qquad \text{where } Z:= \{ k \in \Z^n : \supp \psi_0 \cap \supp \psi_k \neq \varnothing  \}
  \end{align*}
  and $\eta_j (x) := \eta_0 (x-j)$ for all $x \in \Rn$ and $j \in \Z^n$. An application of the partition of unity and H\"older's inequality for sequence spaces first and the case $-s >0$ afterwards provides
  \begin{align*}
    \left|\skh{f}{g}{H^s_p; H^{-s}_q}\right| 
    &\leq \sum_{j \in \Z^n} \left|\skh{\eta_j \psi_j f}{g}{H^s_p; H^{-s}_q} \right| 
    \leq  \sum_{j \in \Z^n} \| \psi_j f \|_{H^s_p} \| \eta_j g \|_{ H^{-s}_q } \\
    &\leq  \left( \sum_{j \in \Z^n} \| \psi_j f \|^p_{H^s_p} \right)^{1/p} \left( \sum_{j \in \Z^n} \| \eta_j g \|^q_{ H^{-s}_q } \right)^{1/q} \\
    &\leq  C_{q, Z} \left( \sum_{j \in \Z^n} \| \psi_j f \|^p_{H^s_p} \right)^{1/p} \left( \sum_{j \in \Z^n} \| \psi_j g \|^q_{ H^{-s}_q } \right)^{1/q} \\
    &\leq  C_{q, Z} \left( \sum_{j \in \Z^n} \| \psi_j f \|^p_{H^s_p} \right)^{1/p}  \| g \|_{ H^{-s}_q }.
  \end{align*}
  Consequently we get (\ref{p107}) by duality and the previous inequality.\\
  Now we define $\eta_j$ for every $j \in \Z^n$ as before and $m:=-\lfloor s \rfloor$. Additionally we choose an arbitrary $f \in H^s_p(\Rn)$. The existence of $g_{\alpha} \in H^{s - \lfloor s \rfloor}_p(\Rn)$, $\alpha \in \Non$ with $|\alpha| \leq m$, fulfilling  $f= \sum_{|\alpha| \leq m} \p^{\alpha}_x g_{\alpha}$ and $\sum_{|\alpha| \leq m} \| g_{\alpha} \|_{H^{s-\lfloor s \rfloor}_p} \leq C \| f \|_{H^s_p}$ was verified in Lemma \ref{lemma:CharacterizationOfBesselPotentialSpaces}. Together with Proposition \ref{prop:CommutatorEstimationOfPartitionOfUnityWithDerivative}, $\eta_j \equiv 1$ on $\supp \psi_j$ and Theorem  \ref{thm:stetigInBesselPotRaum} we obtain
  \begin{align*}
    \| \psi_j f \|^p_{H^s_p}
    &\leq \sum_{|\alpha| \leq m} \| \psi_j  \p^{\alpha}_x g_{\alpha} \|^p_{H^s_p}
    = \sum_{|\alpha| \leq m} \| \psi_j  \p^{\alpha}_x \{ g_{\alpha} \eta_j \} \|^p_{H^s_p} \\
    &\leq \sum_{|\alpha| \leq m} \left\{ \|  \p^{\alpha}_x \{ \psi_j  g_{\alpha} \eta_j \} \|_{H^s_p} + \| [\p^{\alpha}_x, \psi_j] ( g_{\alpha} \eta_j ) \|_{H^s_p} \right\}^p \\
    &\leq \sum_{|\alpha| \leq m} \left\{ \|  \p^{\alpha}_x \{ \psi_j  g_{\alpha} \} \|_{H^s_p} + C \|  g_{\alpha} \eta_j  \|_{H^{s- \lfloor s \rfloor }_p} \right\}^p \\
    &\leq \sum_{k \in Z +j} \sum_{|\alpha| \leq m} \left\{ \|  \p^{\alpha}_x \{ \psi_k  g_{\alpha} \} \|_{H^s_p} + C \|  g_{\alpha} \psi_k  \|_{H^{s- \lfloor s \rfloor }_p} \right\}^p\\
    &\leq C \sum_{k \in Z +j} \sum_{|\alpha| \leq m} \| \psi_k  g_{\alpha} \|_{H^{s- \lfloor s \rfloor }_p}^p.
  \end{align*}
  Using the case $s \geq 0$ provides:
  \begin{align*}
    \sum_{j \in \Z^n} \| \psi_j f \|_{H^s_p}^p
    &\leq C \sum_{j \in \Z^n} \sum_{k \in Z +j} \sum_{|\alpha| \leq m} \| \psi_k  g_{\alpha} \|_{H^{s- \lfloor s \rfloor }_p}^p
    \leq C \sum_{|\alpha| \leq m} \sum_{j \in \Z^n} \| \psi_j  g_{\alpha} \|_{H^{s- \lfloor s \rfloor }_p}^p \\
    &\leq C \sum_{|\alpha| \leq m} \| g_{\alpha} \|_{H^{s- \lfloor s \rfloor }_p}^p
    \leq C \| f \|_{H^s_p}^p.
  \end{align*}
  \vspace*{-1cm}

\end{proof}

There are also  
boundedness results for non-smooth pseudodifferential operators,  
cf. Lemma 3.4 in \cite{Paper1}, Theorem 3.7 in \cite{Paper1}, Theorem 2.1 in \cite{Marschall},  Lemma 2.9 in \cite{Marschall}:

\begin{lemma}\label{lemma:stetigkeitInS}
  Let $1 < q < \infty$, $\tilde{m}\in \N_0$ with $\tilde{m}>n/q$ and $\tau>0$. We consider $X \in \{ C^{\tau}(\Rn),C^{\tau}_{\ast}(\Rn),  X^{\tilde{m},q} \}$. Let $m\in \R$, $M \in \N_0 \cup \{ \infty \}$ and $\delta=0$ in the case $X \notin \{  C^{\tilde{m},\tau}, C^{\tilde{m} + \tau}_{\ast} \}$ and $0 \leq \rho,\delta \leq 1$ else. Assuming a bounded subset $\mathscr{B} \subseteq X\Sallgn{m}{\rho}{\delta}{n}{n}{M}$,  we obtain the boundedness of $\{ p(x, D_x): p \in \mathscr{B} \} $ as a subset of $ \mathscr{L}(\s, X)$. 
\end{lemma}

\begin{thm} \label{thm:stetigInHoelderRaum}
  Let $m \in \R$, $0 \leq \delta \leq \rho \leq 1$ with $\rho > 0$ and $1 < p < \infty $. Additionally let $\tau > \frac{1- \rho}{1-\delta} \cdot \frac{n}{2}$ if $\rho <1$ and $\tau>0$ if $\rho =1$ respectively. Moreover, let $N \in \N \cup \{ \infty \}$ with $N> n/2$ for $2 \leq p < \infty$ and $N>n/p$ else, $k_p := (1- \rho)n \left| 1/2 - 1/p \right|$ and let $\mathscr{B} \subseteq C^{\tau}_{\ast} \Snn{m- k_p}{\rho}{\delta}{N}$ be bounded. Then for each $s \in \R$ with
  $$(1-\rho)\frac{n}{p}-(1-\delta)\tau < s < \tau$$
  there is some $C_s>0$ such that
  \begin{align*}
    \| a(x, D_x)f \|_{ H_p^{s} } \leq C_s \| f \|_{ H_p^{s+m} } \qquad \text{for all } f \in H_p^{s+m}(\Rn) \text{ and } a \in \mathscr{B}.
  \end{align*}
\end{thm}

\begin{thm} \label{thm:1stetigInHoelderRaum00}
  Let $m \in \R$ and $\tau > \frac{n}{2}$. Moreover, let $N \in \N \cup \{ \infty \}$ with $N> n/2$ and $a \in C^{\tau}_{\ast} \Snn{m}{0}{0}{N}$. Then for each $s \in \R$ with
  $n/2-\tau < s < \tau$
  we have
  \begin{align*}
    \| a(x, D_x)f \|_{ H_2^{s} } \leq C_s \| f \|_{ H_2^{s+m} } \qquad \text{for all } f \in H_2^{s+m}(\Rn).
  \end{align*}
\end{thm}

\begin{thm} \label{thm:stetigInHoelderRaum00}
  Let $m \in \R$, $N> n/2$, $\tau > 0$. Moreover let $P$ be an element of $ \op C^{\tau}_{\ast} \Snn{m-n/2}{0}{0}{N}$. Then
  \begin{align*}
    P: H_2^{s+m}(\Rn) \rightarrow H_2^s(\Rn) \qquad \textrm{is continuous for all } -\tau < s < \tau.
  \end{align*}
\end{thm}

On account of the following lemma, cf. Lemma 3.5 in \cite{Paper1}, we can prove similar boundedness results for non-smooth pseudodifferential operators.

\begin{lemma}\label{lemma:UnabhangigkeitVomSymbol}
  Let $N \in \N_0 \cup \{ \infty \}$ and $X$ be a Banach space with $C^{\infty}_c(\Rn) \subseteq X \subseteq C^0(\Rn)$. Additionally let $m$, $\rho$ and $\delta$ be as in the last lemma. 
  We consider that $\mathscr{B}$ is the topological vector space $S^m_{\rho, \delta} (\RnRn)$ or $X S^m_{\rho, \delta} (\RnRn;N)$. In the case $\mathscr{B}=S^m_{\rho, \delta} (\RnRn)$ we set $N:=\infty$. Moreover, let $X_1,X_2$ be two Banach spaces such that
  \begin{enumerate}
    \item[i)] $\s \subseteq X_1, X_2 \subseteq \sd$,
    \item[ii)] $\s$ is dense in $X_1$ and in $ X'_2$,
    \item[iii)] $a(x,D_x) \in \mathscr{L}(X_1, X_2)$ for all $a \in \mathscr{B}$. 
  \end{enumerate}
  Then there is a $k \in \N$ with $k \leq N$  such that
  \begin{align*}
    \| a(x, D_x)f \|_{\mathscr{L}(X_1;X_2)} \leq C |a|^{(m)}_{k} \qquad \text{for all } a \in \mathscr{B}.
  \end{align*}
\end{lemma}

\begin{thm}\label{thm:Marschall,Thm2.2}
  Let $1 < p,q < \infty$ and $m, \tilde{m} \in \R$ with $\tilde{m} > n/q$. Moreover, let $\mathscr{B} \subseteq H^{ \tilde{m} }_q \Sn{m}{1}{0}$ be bounded. Then for each $s \in \R$ with
  \begin{align*}
    n(1/p + 1/q - 1)^{+} - \tilde{m} < s \leq \tilde{m} -n(1/q- 1/p)^{+}
  \end{align*}
  we have
  \begin{align*}
    \| a(x,D_x) f \|_{H^s_p} \leq C_s \| f \|_{ H^{s+m}_p } \qquad \text{for all } f \in H^{s+m}_p(\Rn) \text{ and all } a \in \mathscr{B}.
  \end{align*}
\end{thm}

\begin{proof}
  Let $s$ be as in the assumptions. Due to \cite[Theorem 2.2]{Marschall2} we get $a(x,D_x) \in \mathscr{L}(H^{s+m}_p, H^{s}_p)$. 
  Thus it remains to verify whether $C_s$ is independent of $a \in \mathscr{B}$. We define $p'$ by $1/p + 1/p' = 1$. Since $\s$ is dense in $H^{s+m}_p(\Rn)$ and $H^{-s}_{p'}(\Rn)$,
  the theorem holds due to Lemma \ref{lemma:UnabhangigkeitVomSymbol}.
\end{proof}

Note that the last theorem even holds for $0<p \leq \infty$ and $q \in \{1, \infty\}$ if $\sharp \mathscr{B} = 1$, cf.\;\cite[Theorem 2.2]{Marschall2}. A similar result holds for pseudodifferential operators of the class $H^{ \tilde{m} }_q \Sn{m}{1}{\delta}$, $0 \leq \delta \leq 1$, if $\sharp \mathscr{B} = 1$, cf.\;\cite[Theorem 2.2]{Marschall2}. \\

As a consequence of the previous theorem we obtain the next result. For more details we refer to \cite[Lemma 4.53]{Diss}.

\begin{lemma}\label{lemma:BesselPotentialEstimation}
  Let $1< p,q < \infty$ and $m \in \R$ with $m > n/q$. Assuming $$n \left(1/p + 1/q -1 \right)^{+} - m < s \leq m - n \left( 1/q - 1/p \right)^{+},$$ there is some $C_s>0$ such that
  \begin{align}\label{p79}
    \|a b\|_{H^s_p} \leq C_s \| a \|_{H^m_q} \| b \|_{H^s_p} \qquad \text{for all } a \in H^m_q(\Rn), b \in H^s_p(\Rn).
  \end{align}  
\end{lemma}

\begin{thm}\label{thm:boundednessOfSymbolsWithCoefficientsInUniformlyBoundedSobolevSpaces}
  Let $m \in \R$, $1 < p,q < \infty$ and $\tilde{m} \in \N$ with $\tilde{m}>n/q$. We assume that $\mathscr{B} \subseteq W^{\tilde{m}, q}_{uloc} \Sn{m}{1}{0}$ is bounded. Then for each $s \in \R$ with
  \begin{align}\label{eq57}
    -  \tilde{m} + n/q < s \leq \tilde{m} - n(1/q - 1/p)^+
  \end{align}
  there is some $C_s>0$ such that
  \begin{align*}
    \| a(x, D_x)f \|_{H^{s}_p} \leq C_s \| f \|_{H^{s+m}_p} \qquad \text{for all } f \in  H^{s+m}_p(\Rn) \text{ and } a \in \mathscr{B}.
  \end{align*}
\end{thm}

\begin{proof}
  On account of \cite[Theorem 2.6]{Marschall2} we have $a(x,D_x) \in \mathscr{L}(H^{s+m}_p, H^{s}_p)$ 
  for each $a \in \mathscr{B}$. Therefore it remains to show that $C_s$ is independent of $a \in \mathscr{B}$. We define $p'$ by $1/p + 1/p' = 1$. Since $\s$ is dense in $H^{s+m}_p(\Rn)$ and $H^{-s}_{p'}(\Rn)$,
  the theorem holds true because of Lemma \ref{lemma:UnabhangigkeitVomSymbol}.
\end{proof}

We remark that the previous theorem even holds for $0<p < \infty$ and $q =1$ if $\sharp \mathscr{B} = 1$ and if $s \in \R$ fulfills
\begin{align*}
    n( \max\{ 1, 1/p\} -1 ) -  \tilde{m} + n/q < s \leq \tilde{m} - n(1/q - 1/p)^+
  \end{align*}
instead of the assumption (\ref{eq57}) due to \cite[Theorem 2.6]{Marschall2}. \\

Next we want to improve the previous statement for classical pseudodifferential operators of the symbol-class $W^{\tilde{m}, q}_{uloc} \Sn{m}{1}{0}$. To this end we have to develop some further tools.

\begin{Def}
   Let $\Sigma$ be $[\sigma, \infty)$ or $(\sigma, \infty)$ with $\sigma \in \R \cup \{ -\infty \}$. Then a family of Banach spaces $\{ X^s: s \in \Sigma \}$ is called \textit{scale}, if
   \begin{enumerate}
      \item $C^{\infty}_c(\Rn) \subseteq X_s \subseteq \sd$ for all $s \in \Sigma$,
      \item $X^s \subseteq X^t$ for all $s,t \in \Sigma$ with $t<s$,
      \item $P: X^{s+m} \rightarrow X^s$ for all differential operators $P$ of the order $m \in \N_0$ with smooth coefficients and all $s \in \Sigma$.
   \end{enumerate}
\end{Def}

\begin{Def}
  A scale $\{ X^s: s \in \Sigma \}$ is called \textit{microlocalizable} if for every symbol $p \in S^m_{1,0}(\RnRn)$ the operator
  \begin{align*}
    p(x, D_x): X^{s+m} \rightarrow X^s
  \end{align*}
  is bounded for all $s \in \Sigma$ with $s+m \in \Sigma$.
\end{Def}

\begin{bsp}\label{bsp:microlocalizable}
  Let $1 < p < \infty$. Due to Theorem \ref{thm:stetigInBesselPotRaum} the sets $\{H^s_p(\Rn):  s \in \R \}$ and $\{ C^s_*(\Rn): 0< s < \infty \}$ are microlocalizable.
\end{bsp}

On account of \cite[Proposition 1.1B]{Taylor2} the next statement holds:

\begin{prop}\label{prop:BoundednessOfClassicalSymbols}
  Let  
  $\tilde{m} \in \N_0$
  and $1<q < \infty$. If $\{  X^s :  s \in \Sigma \}$ is microlocalizable and $p \in W^{\tilde{m}, q}_{uloc} S^m_{cl}(\RnRn)$, then 
  \begin{align*}
    p(x, D_x): X^{s+m} \rightarrow X^s
  \end{align*}
  is bounded for $s \in \Sigma$, provided that $s+m \in \Sigma$ and
  $$\| ab \|_{X^s} \leq C \| a \|_{X^s} \| b \|_{W^{\tilde{m}, q}_{uloc}} \qquad \text{ for all } a \in X^s \text{ and } b \in W^{\tilde{m}, q}_{uloc}(\Rn).$$
\end{prop}

Making use of the previous proposition and of Lemma \ref{lemma:UnabhangigkeitVomSymbol} we get:

\begin{thm}\label{thm:boundednessOfClassicalSymbolsInSobolevSpaceUniformly}\label{thm:boundednessOfClassicalSymbolsInSobolevSpace}
  Let $m \in \R$, $1 < p,q < \infty$ and $\tilde{m} \in \N$ with $\tilde{m}>n/q$. Assuming a bounded subset $\mathscr{B} \subseteq W^{\tilde{m},q}_{uloc} S^m_{cl}(\RnRn)$ we get for each $s \in \R$ with
  \begin{align*}
    n \left(1/p + 1/q -1 \right)^{+} - \tilde{m} < s \leq \tilde{m} - n \left( 1/q - 1/p \right)^{+}
  \end{align*}
  the existence of a constant $C_s>0$ such that
  \begin{align}\label{Stern}
    \| a(x, D_x)f \|_{ H_p^{s} } \leq C_s \| f \|_{ H_p^{s+m} } \qquad \text{for all } f \in H_p^{s+m}(\Rn) \text{ and } a \in \mathscr{B}.
  \end{align}
\end{thm}

\begin{proof}
  Due to Example \ref{bsp:microlocalizable}  inequality (\ref{Stern}) is a  
  direct consequence of Proposition \ref{prop:BoundednessOfClassicalSymbols} for each fixed $a \in  \mathscr{B}$ if we have for every $1<p<\infty$ and $s \in \R$ fulfilling the assumptions:
  \begin{align}\label{p77}
    \| fg \|_{H^s_p} \leq C_s \| f \|_{H^s_p} \| g \|_{W^{\tilde{m},q}_{uloc}} \qquad \text{for all } f \in H^s_p(\Rn), g \in W^{\tilde{m},q}_{uloc}(\Rn).
  \end{align}
  Let $s$ and $p$ be as in the assumptions. In view of Proposition \ref{prop:NormequivalenceOfSobolevSpace} we may choose a partition of unity $( \psi_j )_{j \in \Z^n} \subseteq C^{\infty}_c(\Rn)$ with the properties $\supp \psi_0 \subseteq [-\e, \e ]^n$ for one fixed $\e>0$ and $\psi_j(x) = \psi_0(x-j)$ for all $x\in \Rn$ and $j \in \Z^n$. With $Z_j:= \{k\in \Z^n: \supp \psi_k \cap \supp \psi_j \neq \varnothing\}$, we define
  \begin{align*}
    \eta_j (x) := \sum_{k \in Z_j} \psi_k(x) \qquad \text{for all } x \in \Rn \text{ and every } j \in \Z^n.
  \end{align*}
  Choosing a finite cover $\left( B_1(x_i) \right)_{i=1}^N$, $ N \in \N$, of $\supp \eta_0$ with open balls of radius $1$ provides a finite cover $\left( B_1(x_i + j) \right)_{i=1}^N$ of $\supp \eta_j$ with open balls of radius $1$. Hence $N$ is independent of $j \in \Z^n$. By means of the Leibniz rule we obtain:
  \begin{align}\label{p78}
    \|\eta_j g\|^q_{H^{\tilde{m}}_{q}(\Rn)} 
    &\leq \sum_{|\alpha| \leq \tilde{m}} \sum_{\alpha_1 + \alpha_2 = \alpha} C_{\alpha_1, \alpha_2} \int_{\supp \eta_j} |\p^{\alpha_1}_x 	     \eta_j(x)|^q | \p^{\alpha_2}_x g(x) |^q dx \notag \\
    &\leq \sum_{|\alpha| \leq \tilde{m}} \sum_{\alpha_1 + \alpha_2 = \alpha} C_{\alpha_1, \alpha_2} \sum_{i=1}^N \int_{B_1(x_i + j)} |\p^{\alpha_2}_x g(x) |^q dx \notag\\
    &\leq C_{\tilde{m}} \| g \|^q_{W^{\tilde{m},q}_{uloc}(\Rn)}  
  \end{align}
  for all $g \in W^{\tilde{m},q}_{uloc}(\Rn)$ and $j \in \Z^n$.  
  Together with Proposition \ref{prop:NormequivalenceOfSobolevSpace} and Lemma \ref{lemma:BesselPotentialEstimation} we conclude inequality (\ref{p77}):
  \begin{align*}
    \|f g\|^p_{H^s_p} 
    &\leq C_s \sum_{j \in \Z^n } \| (\psi_j f)(\eta_j g) \|_{H^s_p}^p 
    \leq C_{s,\tilde{m}} \sum_{j \in \Z^n } \| \psi_j f \|_{H^s_p}^p \| \eta_j g \|^p_{H^{\tilde{m}}_q} \\
    &\leq C_{s,\tilde{m}} \| g \|^p_{ W^{\tilde{m},q}_{uloc} } \sum_{j \in \Z^n } \| \psi_j f \|_{H^s_p}^p 
    \leq C_{s,\tilde{m}} \| g \|^p_{ W^{\tilde{m},q}_{uloc} } \| f \|_{H^s_p}^p  
  \end{align*}
  for all $f \in H^s_p(\Rn)$ and $ g \in W^{\tilde{m},q}_{uloc}(\Rn)$.
  It remains to verify whether the constant $C_s$ is independent of $a \in \mathscr{B}$. We define $p'$ by $1/p + 1/p' = 1$. Since $\s$ is dense in $H^{s+m}_p(\Rn)$ and in $H^{-s}_{p'}(\Rn)$ 
  the theorem holds because of Lemma \ref{lemma:UnabhangigkeitVomSymbol}.
\end{proof}

\subsection{Kernel Representation}\label{section:KernelRepresentationSmoothCase}

The present subsection is devoted to  the kernel representation of a non-smooth pseudodifferential operator $p(x,D_x)$, whose symbol is in the class $X\Sn{m}{1}{0}$ for a Banach space $X$ with $C^{\infty}_c(\Rn) \subseteq X \subseteq C^0(\Rn)$, we refer to \cite[Chapter 1]{Taylor2} for the definition of these symbol-classes. In particular we can choose $X \in \{ C^{\tilde{m},\tau}, C^{\tilde{m} + \tau}_{\ast}, H^{\tilde{m}}_q, W^{\tilde{m}, q}_{uloc} \}$ with $\tilde{m} \in \N_0$, $0< \tau \leq 1$ and $1<q < \infty$.

\begin{thm}\label{thm:KernelRepresentationNonsmoothCase}
  Let $p \in X\Sn{m}{1}{0}$, where $C^{\infty}_c(\Rn) \subseteq X \subseteq C^0(\Rn)$ is a Banach space and $m \in \R$. Then there is a function $k: \Rn \times (\Rn \backslash \{ 0 \}) \rightarrow \C$ such that $k(x, .) \in C^{\infty}( \Rn \backslash \{ 0 \} )$ for all $x \in \Rn$ and 
  \begin{align*}
    p(x, D_x) u(x) = \int k(x, x-y)u(y)dy \qquad \text{for all } x \notin \supp u
  \end{align*}
  for all $u \in \s$. Moreover, for every $\alpha \in \Non$ and each $N \in \N_0$  the kernel $k$ satisfies
  \begin{align*}
    \| \p_z^{\alpha} k(.,z) \|_X 
    \leq  \left\{  \begin{array}{l l}
		      C_{\alpha, N} |z|^{-n-m-|\alpha|} \<{z}^{-N} & \text{ if } n+m+|\alpha| >0, \\ 
		      C_{\alpha, N} (1 + \left| \log|z| \right| )\<{z}^{-N} & \text{ if } n+m+|\alpha| =0, \\
		      C_{\alpha, N} \<{z}^{-N} & \text{ if } n+m+|\alpha| <0			   
                   \end{array}
	  \right.
  \end{align*}
  uniformly in $z \in \Rn \backslash \{ 0 \}$.
\end{thm}

\begin{proof}
  We are able to prove the statements in a similar way as in \cite[Theorem 5.12]{PDO}. The main idea of the proof is to decompose
  \begin{align*}
    p(x, D_x) f = \sum_{j=0}^{\infty} p(x, D_x) \varphi_j(D_x) f \qquad \text{for all } f \in \s
  \end{align*}
  where $(\varphi_j)_{j \in \N_0}$ is a dyadic partition of unity. The series converges in $X$ due to Lemma \ref{lemma:stetigkeitInS}. First of all we want to construct a kernel $k_j$ of $p_j(x, D_x):= p(x, D_x) \varphi_j(D_x)$ for each $j \in \N_0$. This can be made in the same way as in the smooth case. We just have to use $\|\pa{\alpha} p(., \xi)\|_{X} \leq C_{\alpha} \<{\xi}^{m- |\alpha|}$ instead of $ |\pa{\alpha} p(., \xi)| \leq C_{\alpha} \<{\xi}^{m- |\alpha|}$ for all $\alpha \in  \Non$ and all $\xi \in \Rn$. Afterwards we use this kernel decompositions to construct the kernel of $p(x, D_x)$ as in the smooth case. By means of $X \subseteq C^0(\Rn)$ we get the absolute and uniform convergence of $k(x,z)=\sum_{j = 0}^{\infty} k_j(x, z)$.
\end{proof}

\begin{bem}\label{bem:KernelRepresentationSmoothCase}
  If we even have $p \in S^m_{1,0}(\RnRn)$ in the previous theorem, we can show that $k(.,z)$ is smooth for all $z  \in \Rn$ while applying Theorem \ref{thm:KernelRepresentationNonsmoothCase} with $X= C^{\tau}_{\ast}$ for all $\tau \in \R$. This result already was shown in \;e.g.\;\cite[Theorem 5.12]{PDO}.
\end{bem}

\subsection{Double Symbols}\label{Section:DoubleSymbols}

\begin{Def}
	Let $0<s \leq 1$, $m \in \N_0$, $1<q<\infty$ and $\tilde{m},m'\in \R$. Furthermore, let $N \in \N_0 \cup \{ \infty \}$, $0 \leq \rho \leq 1$ and $X \in \{ C^{m,s}, W^{m,q}_{uloc} \}$. Additionally let $m>n/q$ in the case $X=W^{m,q}_{uloc}$. Then the space of \textit{non-smooth double symbols} $X S^{\tilde{m},m'}_{\rho, 0}(\RnRn \times \RnRn;N)$ 
	is the set of all $p:\R^n_x \times\R^n_{\xi} \times \R^n_{x'} \times\R^n_{\xi'} \rightarrow \C$ such that
	\begin{itemize}
		\item[i)] $\pa{\alpha} \p^{\beta'}_{x'} \p_{\xi'}^{\alpha'} p(.\xi, x', \xi') \in X$ for all $\xi,x',\xi' \in \Rn$ and $\p_x^{\beta} \pa{\alpha} \p^{\beta'}_{x'} \p_{\xi'}^{\alpha'} p \in C^{0}(\R^n_x \times \R^n_{\xi} \times \R^n_{x'} \times\R^n_{\xi'})$,
		\item[ii)] $\| \pa{\alpha}  \p^{\beta'}_{x'} \p_{\xi'}^{\alpha'} p(.,\xi, x', \xi')  \|_{X} \leq C_{\alpha, \beta', \alpha'} \<{\xi}^{\tilde{m}-\rho|\alpha|} \<{\xi'}^{m'-\rho|\alpha'|}$ for all $\xi, x', \xi' \in \Rn$ 
	\end{itemize}
  and $\beta, \alpha, \beta', \alpha' \in \N_0^n$ with $|\beta| \leq m$ and $|\alpha| \leq N$.  
  In the case $N= \infty$ we write $X S^{\tilde{m},m'}_{\rho, 0}(\RnRn \times \RnRn)$
  instead of $X S^{\tilde{m},m'}_{\rho, 0}(\RnRn \times \RnRn;\infty)$. Furthermore, we define the set of semi-norms $\{|.|^{\tilde{m},m'}_k : k \in \N_0 \}$ by
  \begin{align*}
    |p|^{\tilde{m},m'}_k \hspace{-1.2mm}:= \hspace{-0.2mm} \max_{\substack{ |\alpha| + |\beta'| + |\alpha'| \leq k \\ |\alpha| \leq N} } \sup_{\xi, x', \xi' \in \Rn} \hspace{-2mm} \| \pa{\alpha}  \p^{\beta'}_{x'} \p_{\xi'}^{\alpha'} p(.,\xi, x', \xi')  \|_{X} \<{\xi}^{-(\tilde{m}-\rho|\alpha|)} \<{\xi'}^{-(m'-\rho|\alpha'|)}.
  \end{align*}
\end{Def}

Because of the previous definition $p \in X \Sallg{\tilde{m}}{\rho}{\delta}{n}{N}$ is often called a \textit{non-smooth single symbol}. \\

We define the associated operator of a non-smooth double symbol as follows:

\begin{Def}
  Let $0<s \leq 1$, $m \in \N_0$, $1<q<\infty$ and $\tilde{m},m'\in \R$. Furthermore, let $N \in \N_0 \cup \{ \infty \}$, $0 \leq \rho \leq 1$ and $X \in \{ C^{m,s}, W^{m,q}_{uloc} \}$. Additionally let $m>n/q$ in the case $X=W^{m,q}_{uloc}$.
  Assuming $ p \in X S^{\tilde{m},m'}_{\rho, 0}(\RnRn \times \RnRn; N) $,  we define the pseudodifferential operator $P = p(x,D_x, x', D_{x'})$ 
  such that for all $u \in \s$ and $x \in \Rn$
  \begin{align*}
    P u(x) := \osiint e^{-i(y \cdot \xi + y' \xi')} p(x,\xi,x+y,\xi') u(x+y+y')dy dy' \dq \xi \dq \xi' .
  \end{align*}
  The existence of the previous oscillatory integral is a consequence of the properties of such integrals. For more details we refer to \cite[Lemma 4.64]{Diss}.
\end{Def}

The set $\op X S^{\tilde{m},m'}_{\rho, 0}(\RnRn \times \RnRn; N) $ consists of all non-smooth pseudodifferential operators whose double symbols are in  $ X S^{\tilde{m},m'}_{\rho, 0}(\RnRn \times \RnRn; N) $.
Moreover for $0<s \leq 1$, $m \in \N_0$, $N \in \N_0 \cup \{ \infty \}$ and $\tilde{m} \in \R$ we denote the space $X S^{\tilde{m}}_{\rho, 0}(\RnRn \times \Rn;N)$ 
as the set of all $p \in X S^{\tilde{m},0}_{\rho, 0}(\RnRn \times \RnRn;N)$ which are independent of $\xi'$.
The pseudodifferential operator $p(x,D_x,x')$ is defined
by $$p(x,D_x,x') := p(x,D_x, x', D_{x'}).$$
 Additionally $\op X S^{\tilde{m}}_{\rho, 0}(\RnRn \times \Rn;N)$ is the set of all non-smooth pseudodifferential operators whose double symbols are in $X S^{\tilde{m}}_{\rho, 0}(\RnRn \times \Rn;N)$.\\

On account of Lemma \ref{lemma:SobolevHoelderInequalityFor_LqUlocSpaces} and the definition of the non-smooth symbol-classes we obtain:

\begin{bem}\label{bem:einbettungSymbolklassenDoppelsymbol}
  Let $1 < q < \infty$, $m,m' \in \R$ and $\tilde{m}\in \N_0$ with $\tilde{m}>n/q$. Moreover, let $N \in \N_0 \cup \{ \infty \}$ and $0 \leq \rho \leq 1$. Assuming $0 < \tau \leq \tilde{m}-n/q$, $\tau \notin \N$, we have
  \begin{align*}
    W^{\tilde{m}, q}_{uloc} S^{m,m'}_{\rho, 0}(\RnRn \times \RnRn;N) \subseteq C^{\tau} S^{m,m'}_{\rho, 0}(\RnRn \times \RnRn;N).
  \end{align*}
\end{bem}

Later we will need the following statement:

\begin{bem}\label{bem:beschrTeilmengeVonWmq_ulocSymbolen}
  Let $m \in \R$, $1 < q < \infty$ and $\tilde{m} \in \N_0$ with $\tilde{m} > n/q$. Moreover, let $0 \leq \rho \leq 1$  and $N \in \N_0 \cup \{ \infty \}$. If
$\mathscr{B} \subseteq W^{\tilde{m}, q}_{uloc} S^m_{\rho,0}(\RnRnRn;N)$ is a
bounded subset, we get the boundedness of
  \begin{align*}
    \mathscr{B}':= \left\{ \p^{\gamma}_y \pa{\delta} a: a \in \mathscr{B}
\right\} \subseteq W^{\tilde{m},q}_{uloc} S^{m- \rho
|\delta|}_{\rho,0}(\RnRnRn;N-|\delta|)
  \end{align*}
  for each $\gamma, \delta \in \Non$ with $|\delta| \leq N$.
\end{bem}

By means of  Remark \ref{bem:beschrTeilmengeVonWmq_ulocSymbolen} and Lemma \ref{lemma:AbschaetzungInHmq_uloc} we get:

\begin{lemma}\label{lemma:SymbolAbschatzungHmqUloc}
  Let $m \in \R$, $1 < q < \infty$ and $\tilde{m} \in \N_0$ with $\tilde{m} > n/q$. Moreover, let $0 \leq \rho \leq 1$ and $N \in \N_0 \cup \{ \infty \}$.
  Assuming a bounded subset $\mathscr{B}$ of $W^{\tilde{m}, q}_{uloc}
S^m_{\rho,0}(\RnRnRn;N)$, we can show for each $\gamma, \delta \in \Non$ with
$|\delta| \leq N$ there is some $C_{\tilde{m}, q, \gamma, \delta}$ such that
  \begin{align*}
    \sup_{y\in \Rn} \| \p^{\gamma}_y \pa{\delta} a(x, \xi, x+y)
\|_{W^{\tilde{m}, q}_{uloc}(\Rn_{x})} \leq C_{\tilde{m}, q, \gamma, \delta}
\<{\xi}^{m- \rho |\delta|} \qquad \text{for all } a\in \mathscr{B}, \xi \in \Rn.
  \end{align*}
\end{lemma}

\section{Improvement of the Characterization} \label{ImprovementCharacterization}

In this section we show that the operator $T$ of Theorem \ref{thm:classificationA10} is even an element of $\op W^{\tilde{m}, q}_{uloc} \Snn{m}{1}{0}{\tilde{M}-1}$. The proof of this statement is essentially the same as the one of Theorem \ref{thm:classificationA10}. We just have to replace the results for pseudodifferential operators with coefficients in a Hölder space with analogous ones for pseudodifferential operators with coefficients in an uniformly local Sobolev space. 

The main difficulty comes along with the symbol reduction of non-smooth double symbols of the class $X S^m_{0,0} (\RnRnRnRn; M)$ to non-smooth single symbols with coefficients in $X $, where $X=W^{\tilde{m}, q}_{uloc}$. Both cases, $X=C^{\tilde{m},\tau}$ and $X=W^{\tilde{m}, q}_{uloc}$ make use of
\begin{align*}
  \sup_{y \in \Rn} \| \p_y^{\gamma} \p_{\xi}^{\delta} a(.,\xi, .+y)\|_{X} \leq C \<{\xi}^m,
\end{align*}
where $a \in X S^m_{0,0} (\RnRnRn; M)$ and $\gamma, \delta \in \Non$ with $|\delta| \leq M$. While this estimate directly follows from the definition of the symbol-class in the case $X=C^{\tilde{m},\tau}$, this proof turned out to be rather tedious for the case $X=W^{\tilde{m}, q}_{uloc}$ in Subsection \ref{subsection:HmqUlocAbschaetzung}. The symbol reduction for uniformly local Sobolev spaces is the subject of Subsection \ref{subsection:SymbolReductionHmqUloc}. 
But let us begin with the improvement of Theorem \ref{thm:pointwiseConvergence}.

\subsection{Pointwise Convergence in $W^{\tilde{m}, q}_{uloc} S^0_{0,0}$}\label{subsection:pointwiseConvergenceHmqUloc}

\begin{thm}\label{thm:pointwiseConvergenceHmqUloc}
  Let $M \in \N_0 \cup \{ \infty \}$, $\tilde{m} \in \N_0$ and $1< q < \infty$.  
  Furthermore,  let $( p_{\e} )_{\e > 0 } \subseteq W^{\tilde{m}, q}_{uloc} \Snn{0}{0}{0}{M}$ be bounded. Then there is a subsequence $( p_{\e_l} )_{l \in \N } \subseteq ( p_{\e} )_{\e > 0 }$ with $\e_l \rightarrow 0$ for $l \rightarrow \infty$ and a $p:\RnRn \rightarrow \C$ such that
  \begin{itemize}
    \item[i)] $p(x,.) \in C^{M-1}(\Rn)$ for all $x \in \Rn$,
    \item[ii)] $\p_x^{\beta} \pa{\alpha} p \in C^0(\RnRnx{x}{\xi})$,
    \item[iii)] $\p_x^{\beta} \pa{\alpha} p_{\e_l} \xrightarrow[]{l \rightarrow \infty} \p_x^{\beta} \pa{\alpha} p$ uniformly on each compact set of $\RnRn$ 
  \end{itemize}
  for every $\alpha, \beta \in \Non$ with $|\alpha| \leq M-1$ and $|\beta| < \tilde{m}-n/q$. Moreover, 
  $$p \in  W^{\tilde{m}, q}_{uloc} \Snn{0}{0}{0}{M-1}.$$
\end{thm}

\begin{proof}
  Let $\tau \leq \tilde{m} - n/q$ with $\tau \notin \N$. 
  On account of Lemma \ref{lemma:einbettungSymbolklassen}
  we are able to apply Theorem \ref{thm:pointwiseConvergence} and get 
  the existence of a subsequence $( p_{\e_l} )_{l \in \N } \subseteq ( p_{\e} )_{\e > 0 }$ with $\e_l \rightarrow 0$ for $l \rightarrow \infty$ which fulfills the properties $i)$, $ii)$ and $iii)$ for some $p \in C^{\tau} \Snn{0}{0}{0}{M-1}$. Thus it remains to show $p \in W^{\tilde{m}, q}_{uloc} \Snn{0}{0}{0}{M-1}$. By means of $i)$ and $ii)$, we just have to check $\p_{\xi}^{\alpha} p(.,\xi)  \in W^{\tilde{m}, q}_{uloc}(\Rn) $ for all $\xi \in \Rn$ and 
  \begin{align*}
    \| \pa{\alpha} p(.,\xi) \|_{W^{\tilde{m}, q}_{uloc}(\Rn) } \leq C_{\alpha} \qquad \text{for all } \xi \in \Rn,
  \end{align*}
  for all $\alpha \in \Non$ with $|\alpha| \leq M-1$. Let $\alpha \in \Non$ with $|\alpha| \leq M-1$ and $\xi \in \Rn$ be arbitrary but fixed. Moreover let  $z_j \in \left\{ n^{-1/2} z : z \in \Z^n \right\}$ for each $j \in \N$ such that $\Rn= \bigcup_{j \in \N} B_1(z_j)$. The boundedness of $( p_{\e_l} )_{l \in \N } \subseteq W^{\tilde{m}, q}_{uloc} \Snn{0}{0}{0}{M}$ yields
  \begin{align}\label{p109}
    \| \pa{\alpha} p_{\e_l}(., \xi) \|_{H^{\tilde{m}}_q (B_1(z_j))} 
    \leq \| \pa{\alpha} p_{\e_l}(., \xi) \|_{W^{\tilde{m}, q}_{uloc}(\Rn) } 
    \leq C_{\alpha},
  \end{align}
  for all $j,l \in \N$ and $\xi \in \Rn$. 
  Let $j \in \N$ be arbitrary but fixed. Because of the reflexivity of $H^{\tilde{m}}_q(B_1(z_j))$
  there is a subsequence $( \pa{\alpha} p_{{\e_l}_m} )_{m \in \N} $ of $ ( \pa{\alpha} p_{\e_l} )_{l \in \N}$ such that
  \begin{align*}
    \pa{\alpha} p_{{\e_l}_m} (., \xi) \rightharpoonup q_{\alpha,\xi,j} \qquad \text{in } H^{\tilde{m}}_q(B_1(z_j))
  \end{align*}
  for $m \rightarrow \infty$. The compact embedding $H^{\tilde{m}}_q(B_1(z_j)) \hookrightarrow \hookrightarrow C^{0} (\overline{B_1(z_j)})$ even gives us 
  \begin{align*}
    \pa{\alpha} p_{{\e_l}_m} (., \xi) \xrightarrow[]{m \rightarrow \infty} q_{\alpha,\xi,j} \qquad \text{in } C^0(\overline{B_1(z_j)}).
  \end{align*}
  Together with $iii)$ the uniqueness of the limit provides $q_{\alpha,\xi,j} = \pa{\alpha} p(., \xi)$. Consequently every arbitrary weak convergent subsequence of $ ( \pa{\alpha} p_{\e_l} )_{l \in \N}$ has the same weakly limit. Hence an application of \cite[Chapter 3, Lemma 0.3]{Ruzicka} implies
  \begin{align*}
    \pa{\alpha} p_{\e_l}(., \xi)  \rightharpoonup   \pa{\alpha} p(., \xi) \qquad \text{in } H^{\tilde{m}}_q(B_1(z_j))
  \end{align*}
  for $m \rightarrow \infty$. Using the previous weak convergence and (\ref{p109}), we get for all $j \in \N$:
  \begin{align}\label{eq2}
    \| \pa{\alpha} p(., \xi) \|_{H^{\tilde{m}}_q (B_1(z_j))} 
    &\leq \liminf_{l \rightarrow \infty} \| \pa{\alpha} p_{\e_l}(., \xi) \|_{H^{\tilde{m}}_q (B_1(z_j))}
    \leq \sup_{l \in \N} \| \pa{\alpha} p_{\e_l}(., \xi) \|_{H^{\tilde{m}}_q (B_1(z_j))} \notag\\
    &\leq C_{\alpha} \qquad \text{ for all } j \in \N \text{ and } \xi \in \Rn.
  \end{align}
  Since there is an $N \in \N$, independent of $x_0 \in \Rn$, with
    $B_1(x_0) \subseteq \bigcup_{k=1}^N B_1(z_{j_k})$
  for $j_1, \ldots , j_N \in \N$,
  we get together with inequality (\ref{eq2})
  \begin{align*}
    \sup_{\xi \in \Rn} \| \pa{\alpha} p(.,\xi) \|_{ W^{\tilde{m},q}_{uloc} (\Rn) }  
    = \sup_{\xi \in \Rn} \sup_{x_0 \in \Rn } \| \pa{\alpha} p(.,\xi) \|_{ H^{\tilde{m}}_{q} (B_1(x_0)) }
    \leq C_{\alpha}.
  \end{align*}
  \vspace*{-1cm}

\end{proof}

\subsection{Symbol Reduction of Double Symbols in $W^{\tilde{m}, q}_{uloc} S^m_{0,0}$}\label{subsection:SymbolReductionHmqUloc}

The last missing piece towards a better characterization is the improvement of Theorem \ref{thm:SymbolReduktionNichtGlatt}. For this we need the next proposition:

\begin{prop}\label{prop:IntAbschatzungHmqUloc}
  Let $1< q < \infty$, $\tilde{m} \in \N_0$ with $\tilde{m} > n/q$ and $m \in \R$. Additionally let $N \in \N_0 \cup \{ \infty \}$ with $n < N$. Moreover, let $\mathscr{B} \subseteq W^{\tilde{m},q}_{uloc} S^{m}_{0,0}(\RnRnRn; N)$ be bounded and $a \in \mathscr{B}$. Considering $l_0 \in \N_0$ with $n < l_0 \leq N$, we define
  \begin{align*}
    r(x,\xi,\eta,y):= A^{l_0}(D_{\eta},y) a(x,\xi+\eta,x+y) \qquad \text{for all } x, \xi, \eta, y \in \Rn.
  \end{align*}
  Then $\int e^{-iy\cdot \eta} r(x,\xi, \eta, y) dy \in L^{1}(\Rn_{\eta})$ for all $x, \xi \in \Rn$. If we define
  \begin{align*}
    I(x, \xi) := \int \left[ \int e^{-iy\cdot \eta } r(x,\xi,\eta,y) dy \right] \dq \eta
  \end{align*}
  for each $x, \xi \in \Rn$, then there is a constant $C$, independent of $\xi \in \Rn$ and $a \in \mathscr{B}$, such that
  \begin{align*}
    \left\| I(., \xi) \right\|_{ W^{\tilde{m},q}_{uloc} } \leq C \<{\xi}^m \qquad \text{for all } \xi \in \Rn.
  \end{align*}
\end{prop}

The proof of the previous proposition is based on the following Proposition, c.f. Proposition 4.6 of \cite{Paper1}:

 \begin{prop}\label{prop:HilfslemmaIntAbschatzung}
  Let $m \in \R$ and $X$ be a Banach space with $X \hookrightarrow L^{\infty}(\Rn)$.
  Considering an $ l_0 \in \N_0$ with $-l_0 < -n$, let $\mathscr{B}$ be a set of functions $r: \RnRnRnRn \rightarrow \C$ which are smooth with respect to the fourth variable such that for all $l \in \N_0$ there is some $C_l>0 $ such that
  \begin{align}\label{SternStern}
    \| \<{D_y}^{2l} r(.,\xi,\eta,y) \|_X \leq C_{l} \<{y}^{-l_0} \<{\xi + \eta}^m \qquad \text{for all } \xi, \eta, y \in \Rn, r \in \mathscr{B}.
  \end{align}
  Then $ \int e^{-iy\cdot \eta } r(x,\xi, \eta,y) dy \in L^1(\Rn_{\eta})$ for all $x,\xi \in \Rn$. If we define  
  \begin{align*}
    I(x,\xi) := \int \left[ \int e^{-iy\cdot \eta } r(x,\xi,\eta,y) dy \right] \dq \eta 
  \end{align*}
  for $x, \xi \in \Rn$ and $r \in \mathscr{B}$ we have for some $C>0$
  \begin{align*}
    \left\| I(.,\xi) \right\|_{X} \leq C \<{\xi}^m \qquad \text{ for all } \xi \in \Rn \text{ and } r \in \mathscr{B}.
  \end{align*}
\end{prop}

\begin{proof}[Proof of Proposition \ref{prop:IntAbschatzungHmqUloc}]
  Due to Lemma \ref{lemma:SymbolAbschatzungHmqUloc} 
  and of Proposition \ref{prop:HilfslemmaIntAbschatzung} the claim holds. 
\end{proof}

In the same manner as in the proof of Lemma 4.9 of \cite{Paper1} the previous result enables us to prove the next lemma:

\begin{lemma}\label{lemma:AbschVonaLHmqUloc}
  Let $1< q < \infty$, $\tilde{m} \in \N_0$ with $\tilde{m} > n/q$ and $m \in \R$. Additionally let $N \in \N_0 \cup \{ \infty \}$ with $n < N$.  We define $\tilde{N}:= N-(n+1)$. Moreover, let $\mathscr{B} \subseteq W^{\tilde{m},q}_{uloc} S^{m}_{0,0}(\RnRnRn; N)$ be bounded. If we define for each $a \in \mathscr{B}$ the function $a_L: \RnRn \rightarrow \C$ as in Theorem \ref{thm:SymbolReduktionNichtGlatt}
  we have for each $\gamma \in \Non $ with $|\gamma| \leq \tilde{N}$ 
  \begin{align} \label{eq:Hilfsabsch}
    \| \p^{\gamma}_{\xi} a_L (., \xi) \|_{W^{\tilde{m},q}_{uloc}(\Rn)} \leq C_{\gamma} \<{\xi}^m \qquad \text{for all } \xi \in \Rn \text{ and } a \in \mathscr{B}
  \end{align}
  for some $C_{\gamma}>0$.
\end{lemma}

\begin{proof}
  With Lemma \ref{lemma:einbettungSymbolklassen} and Theorem \ref{thm:SymbolReduktionNichtGlatt} at hand, it remains to show (\ref{eq:Hilfsabsch}).
  Due to 
  $a \in \mathscr{A}_0^{m,N}(\RnRnx{y}{\eta})$ 
  and $N-\tilde{N} =k >n$
  we can apply Theorem \ref{thm:OscillatoryIntegralGleichung} and Proposition \ref{prop:IntAbschatzungHmqUloc} and get for each $x,\xi \in \Rn$, $a \in \mathscr{B}$ and $ l_0 \in \N_0$ with $n < l_0 \leq N$:
  \begin{align*}
    \|a_L(x,\xi)\|_{  W^{\tilde{m},q}_{uloc}(\Rn_x) } &= \left\| \osint  e^{-iy \cdot \eta} A^{l_0}(D_{\eta},y)  a(x,\xi+\eta,x+y) dy \dq \eta \right\|_{  W^{\tilde{m},q}_{uloc}(\Rn_{x}) } \\
    &=\left\| \iint e^{-iy \cdot \eta} A^{l_0}(D_{\eta},y) a(x,\xi+\eta,x+y) dy \dq \eta \right\|_{  W^{\tilde{m},q}_{uloc}(\Rn_{x})  } \\
    &\leq C \<{\xi}^m \qquad \text{for all } \xi \in \Rn \text{ and } a \in \mathscr{B}.
  \end{align*}
  For more details concerning the second equality we refer to Proposition 4.8 of \cite{Paper1}.
  Thus the theorem holds for $\gamma =0$. Now let $\gamma \in \Non$ with $|\gamma| \leq \tilde{N}$. 
  Because of $N-\tilde{N} =2k >n$ for a $k \in \N_0$ and $a \in \mathscr{A}_0^{m,N}(\RnRnx{y}{\eta})$ we can apply Theorem \ref{thm:VertauschenVonOsziIntUndAbleitungen} and get
  \begin{align*}
    \p^{\gamma}_{\xi} a_L(x, \xi) = \osint e^{-iy \cdot \eta} \p^{\gamma}_{\xi} a(x, \eta + \xi, x+y) dy \dq \eta.
  \end{align*}
  On account of
  Remark \ref{bem:beschrTeilmengeVonWmq_ulocSymbolen} the first case, applied on the set $\mathscr{B}^{\gamma}$, yields 
  (\ref{eq:Hilfsabsch}).
\end{proof}

 The previous lemma enables us to show the improvement of the symbol reduction in the non-smooth case:

\begin{thm}\label{thm:SymbolReduktionNichtGlattHmqUloc}
  Let $1<q< \infty$, $\tilde{m} \in \N$ with $\tilde{m} > n/q$ and $m \in \R$. Additionally let $N \in \N_0 \cup \{ \infty \}$ with $N > n$. We define $\tilde{N}:= N-(n+1)$. Furthermore, let $\mathscr{B}$ be a bounded subset of $ W^{\tilde{m},q}_{uloc} S^{m}_{0,0}(\RnRnRn; N)$. For $a \in \mathscr{B}$ we define $a_L:\RnRn \rightarrow \C$ as in Theorem \ref{thm:SymbolReduktionNichtGlatt}.
  Then $\{a_L: a \in \mathscr{B} \} \subseteq  W^{\tilde{m},q}_{uloc} S^m_{0,0} (\RnRnx{x}{\xi}; \tilde{N})$ is bounded and we have
  \begin{align}\label{eq18}
    a(x,D_x, x') u = a_L(x,D_x) u \qquad \text{for all } a \in \mathscr{B} \text{ and } u \in \s.
  \end{align}
\end{thm}

\begin{proof}
  On account of Remark \ref{bem:einbettungSymbolklassenDoppelsymbol}, Theorem \ref{thm:SymbolReduktionNichtGlatt} and Lemma \ref{lemma:AbschVonaLHmqUloc} the claim holds.
\end{proof}

\subsection[Characterization of Non-Smooth Operators$]{Characterization of Non-Smooth Pseudodifferential Operators}\label{CharactPDO}

With all the work done in the last subsections we are now in the position to improve Theorem \ref{thm:classificationA10} and Theorem \ref{thm:classA00}.

\begin{thm} \label{thm:classA00HmqUloc}
  Let $m\in \R$, $1<q<\infty$, $\tilde{m} \in \N_0$ with $\tilde{m}> n/q$. Additionally let $M \in \N_0 \cup \{ \infty \}$ with $M  >n$. We define $\tilde{M}:=M-(n+1)$. Considering $T \in \mathcal{A}^{m,M}_{0,0}(\tilde{m},q)$ and $\tilde{M} \geq 1$ we have 
  $$T \in \op W^{\tilde{m},q}_{uloc} S^m_{0,0}(\RnRn; \tilde{M}-1) \cap \mathscr{L}(H^m_q(\Rn),L^q(\Rn)).$$
\end{thm}

\begin{proof}
  The proof of the claim is essentially the same as that one of Theorem \ref{thm:classA00}, cf. Subsection 4.4 in \cite{Paper1}. In the case $m=0$ we just have to replace the results for pseudodifferential operators with coefficients in  Hölder spaces with corresponding ones for pseudodifferential operators with coefficients in uniformly local Sobolev spaces. \\
  Therefore we have to use the continuous embedding $H^{\tilde{m}}_q(\Rn) \hookrightarrow W^{\tilde{m},q}_{uloc}(\Rn)$  instead of the continuous embedding $H^{\tilde{m}}_q(\Rn) \hookrightarrow C^{\tau}(\Rn)$ in step two. Additionally we have to apply Theorem \ref{thm:SymbolReduktionNichtGlattHmqUloc} instead of Theorem \ref{thm:SymbolReduktionNichtGlatt} in the second step. \\
  In step three we have to replace Theorem \ref{thm:pointwiseConvergence} with Theorem \ref{thm:pointwiseConvergenceHmqUloc}. We also use the fact that $W^{\tilde{m},q}_{uloc} S^0_{0,0}(\RnRn; \tilde{M})$ is a subset of $ C^{\tau} S^0_{0,0}(\RnRn; \tilde{M})$ for every $\tau \in(0,  \tilde{m}-n/q]$ with $\tau \notin \N$, cf.\;Lemma \ref{lemma:einbettungSymbolklassen}. \\
  The general case $m \in \R$ can be verified in the same way as to the general case of Theorem \ref{thm:classA00}, cf. Lemma 4.21 of \cite{Paper1} by means of the order reducing pseudodifferential operator $\Lambda^{-m}$ with symbol $\<{\xi}^{-m}$ and the case $m=0$. Instead of corresponding results for pseudodifferential operators with coefficients in the Hölder spaces we use that $P \Lambda^{-m} \in  \op W^{\tilde{m},q}_{uloc} S^0_{0,0}(\RnRn; M) \cap \mathscr{L}(L^q(\Rn))$ implies $ P \in \op W^{\tilde{m},q}_{uloc} S^m_{0,0}(\RnRn; M) \cap \mathscr{L}(H^m_q(\Rn),L^q(\Rn))$ here. For more details we refer to \cite[Section 5.6.3]{Diss}. 
\end{proof}

\begin{thm}\label{thm:classificationA10HmqUloc}
  Let  $m\in \R$, $1<q<\infty$, $\tilde{m} \in \N_0$ with $\tilde{m}>n/q$. Additionally let $M \in \N_0$ with $M >n$. We define $\tilde{M}:=M-(n+1)$. Considering $P \in \mathcal{A}^{m, M}_{1,0}(\tilde{m},q)$ and $\tilde{M} \geq 1$ we obtain
  $$P \in \op W^{\tilde{m},q}_{uloc} S^m_{1,0}(\RnRn; \tilde{M}-1) \cap \mathscr{L}(H^m_q(\Rn),L^q(\Rn)).$$
\end{thm}

\begin{proof}
  Let $\tau \in (0, \tilde{m}-n/q]$, $\tau \notin \N$. An application of Theorem \ref{thm:classificationA10} yields
  $$P \in \op C^{\tau}S^m_{1,0}(\RnRn; \tilde{M}-1) \cap \mathscr{L}(H^m_q(\Rn),L^q(\Rn)).$$
  Let $\alpha \in \Non$ with $|\alpha| \leq \tilde{M}-1$ be arbitrary. Due to the proof of Theorem \ref{thm:classificationA10}, cf.\,\cite[Theorem 4.22]{Paper1} we know that $\ad(-ix)^{\alpha} P \in \mathcal{A}_{1,0}^{m-|\alpha|, M - |\alpha|}(\tilde{m},q)$. Hence an application of Theorem \ref{thm:classA00HmqUloc} provides
  \begin{align*}
	\ad(-ix)^{\alpha} P \in \op W^{\tilde{m},q}_{uloc} S^{m-|\alpha|}_{0,0}(\RnRn; \tilde{M}-|\alpha|-1).
  \end{align*}
  On account of the proof of Theorem \ref{thm:classificationA10} the symbol of $\ad(-ix)^{\alpha} P$ is $\pa{\alpha} p(x,\xi)$. This implies $P \in \op W^{\tilde{m},q}_{uloc}S^m_{1,0}(\RnRn; \tilde{M}-1)$, since
  \begin{align*}
    \| \pa{\alpha} p(.,\xi) \|_{  W^{\tilde{m},q}_{uloc}(\Rn) } \leq C_{\alpha} \<{\xi}^{m-|\alpha|} \qquad \textrm{for all } \xi \in \R^n.
  \end{align*}
  and $|\alpha| \leq \tilde{M}-1$ is arbitrary.
\end{proof}

If we use Theorem \ref{thm:classificationA10HmqUloc} and Theorem \ref{thm:classA00HmqUloc} instead of Theorem \ref{thm:classificationA10} and Theorem \ref{thm:classA00} in the proof of Theorem 5.1 in \cite{Paper1}, we can improve the result for the symbol composition of non-smooth pseudodifferential operators:  

\begin{thm}\label{thm:symbolComposition}
  Let $m_i \in \R$, $M_i \in \N \cup \{ \infty \}$ and $\rho_i \in \{0,1\}$ for $i \in \{ 1,2 \}$. Additionally let $0< \tau_i < 1$ and $\tilde{m}_i \in \N_0$ be such that $\tau_i + \tilde{m}_i > (1-\rho_i)n/2$ for $i \in \{ 1,2 \}$. We define $k_{i}:=(1-\rho_i)n/2$ for $i \in \{ 1,2\}$, $\rho:= \min\{\rho_1; \rho_2 \}$ and $m:= m_1+ m_2 + k_{1} + k_{2}$.   
  Moreover, let $\tilde{m}, M \in \N$ and $1< q < \infty$ be such that
  \begin{itemize}
    \item[i)] $M \leq \min \left\{ M_i- \max \{ n/q; n/2 \}: i \in \{1,2\} \right\}$,
    \item[ii)] $n/q < \tilde{m} \leq \min\{ \tilde{m}_1 ; \tilde{m}_2\}$,
    \item[iii)] $\tilde{m} < \tilde{m}_2 + \tau_2  - m_1 - k_{1}$,
    \item[iv)] $\rho M + \tilde{m} < \tilde{m}_2 + \tau_2 + m_1  + k_{1}$,
    \item[v)] $\tilde{M} \geq 1$, where $\tilde{M}:= M-(n+1)$, 
    \item[vi)] $q=2$ in the case $(\rho_1, \rho_2) \neq (1,1)$.
  \end{itemize}
  Considering two symbols $p_i \in C^{\tilde{m}_i, \tau_i} \Snn{m_i}{\rho_i}{0}{M_i}$, $i \in \{1,2\}$, we obtain
  \begin{align*}
    p_1(x,D_x) p_2(x,D_x) \in \op W^{\tilde{m},q }_{uloc} \Snn{m}{\rho}{0}{\tilde{M}-1}.
  \end{align*}
\end{thm}

We are even able to improve this result for non-smooth pseudodifferential operators with coefficients in the uniformly local Sobolev Spaces:

\begin{thm}\label{thm:SymbolCompositionWmqUloc}
  Let $m_i \in \R$ and $1 < q_i < \infty$ for $i \in \{ 1,2 \}$. Additionally let $\tilde{m}_i \in \N_0$ with $\tilde{m}_i > n/{q_i}$ for $i \in \{ 1,2 \}$. We define $m:= m_1+ m_2$.   
  Moreover let $\tilde{m}, M \in \N$ and $1< q < \infty$ be such that
  \begin{itemize}
    \item[i)] $n/q < \tilde{m} < \min\{ \tilde{m}_1-n/{q_1}; \tilde{m}_2-n/{q_2} \}$,
    \item[ii)] $\tilde{m} \leq \tilde{m}_2 -n(1/q_2- 1/q)^+ - m_1$,
    \item[iii)] $M + \tilde{m} < \tilde{m}_2 -n/q_2 + m_1$,
    \item[iv)] $\tilde{M} \geq 1$, where $\tilde{M}:= M-(n+1)$. 
  \end{itemize}
  Considering two symbols $p_i \in W^{\tilde{m}_i, q_i}_{uloc} \Sn{m_i}{1}{0}$, $i \in \{1,2\}$, we obtain
  \begin{align*}
    p_1(x,D_x) p_2(x,D_x) \in \op W^{\tilde{m},q}_{uloc} \Snn{m}{1}{0}{\tilde{M}-1}.
  \end{align*}
\end{thm}

\begin{proof}
  The proof of the theorem is essentially the same as that one of Theorem 5.1 in \cite{Paper1}. We just have to replace Remark \ref{bem:SymbolOfIteratedCommutatorNonSmooth} with Remark \ref{bem:SymbolOfIteratedCommutatorNonSmoothInUniformlyLocallySobolevSpaces} and Theorem \ref{thm:stetigInHoelderRaum} with Theorem \ref{thm:boundednessOfSymbolsWithCoefficientsInUniformlyBoundedSobolevSpaces}.
\end{proof}

In the same way as the statement of Theorem \ref{thm:symbolComposition} it should be possible to verify a similar result for the composition of two pseudodifferential operators of the symbol-class $H^{\tilde{m}}_{q} S^m_{1,0} (\RnRn)$ by using Remark \ref{bem:SymbolOfIteratedCommutatorNonSmoothInSobolevSpaces} and Theorem \ref{thm:Marschall,Thm2.2} instead of Remark \ref{bem:SymbolOfIteratedCommutatorNonSmooth} and Theorem \ref{thm:stetigInHoelderRaum}.

We also could consider the composition of two non-smooth pseudodifferential operators whose coefficients are either in a Hölder space, in a Bessel potential space or in an uniformly local Sobolev spaces, however in different spaces. Adapting the proof of Theorem \ref{thm:symbolComposition} one could obtain similar results for these cases. 
 
\section{Spectral Invariance}\label{SpectralInvariance}

\subsection[The Inverse of an Operator with Symbol in $C^{\tau} S^{0}_{0,0} $]{The Inverse of a Pseudodifferential Operator in the Symbol-Class $C^{\tau} S^{0}_{0,0} $}\label{section:SpectralInvarianceHölderCase}

In the present subsection

\begin{thm}\label{thm:SpectralinvarianceS000}
  Let $\tilde{m} \in \N_0$ and $0 < \tau < 1$. We assume $$\hat{m}:=\max \{ k \in \N_0: \tilde{m}+\tau-k> n/2 \}>n/2.$$  For every $p \in C^{\tilde{m},\tau} \Sn{0}{0}{0} $ with $p(x,D_x)^{-1} \in \mathscr{L}(L^2(\Rn))$ we get
  \begin{align*}
    p(x,D_x)^{-1} \in \op W^{\hat{m},2}_{uloc} S^0_{0,0}(\RnRn) \subseteq \op C^s \Sn{0}{0}{0}
  \end{align*}
   for all $s \in (0, \hat{m} - n/2]$ with $ s \notin \N$.
\end{thm}

Ueberberg proved a similar result for the smooth case, cf.\;\cite[Theorem 4.3]{Ueberberg}:

\begin{thm}\label{thm:SpektralinvarianzGlatterFall}
  Let $1< q < \infty$ and $0 \leq \delta \leq \rho \leq 1$ with $\delta <1$.
  \begin{itemize}
    \item[i)] Considering a symbol $p \in \Sn{0}{\rho}{\delta}$ where $p(x, D_x)^{-1} \in \mathscr{L}(L^2(\Rn))$ we obtain $p(x, D_x)^{-1} \in \op \Sn{0}{\rho}{\delta}$.
    \item[ii)] Assuming a symbol $p \in \Sn{0}{1}{\delta}$ where $p(x, D_x)^{-1} \in \mathscr{L}(L^q(\Rn))$ we get $p(x, D_x)^{-1} \in \op \Sn{0}{1}{\delta}$.
  \end{itemize}
\end{thm}

In order to verify Theorem \ref{thm:SpectralinvarianceS000}, we use the main idea of the proof in the smooth case: We want to apply the characterization of pseudodifferential operators. Thus we just have to show the boundedness of certain iterated commutators of $p(x, D_x)^{-1}$. Since we already know that the iterated commutators of $p(x, D_x)$ have these mapping properties, we try to write the iterated commutators of $p(x, D_x)^{-1}$ as a sum and compositions of $p(x, D_x)^{-1}$ and the iterated commutators of $p(x, D_x)$. Unfortunately, non-smooth pseudodifferential operators are in general not bounded as operators from $\s$ to $\s$ like the smooth ones. Therefore we have to prove the formal identities for the iterated commutators rigorously.

\begin{bem}[Formal identities for the iterated commutators]\label{bem:Formal identity} \hspace{1cm} \\
  Let $m,s \in \R$, $1 < q < \infty$ and $M, \tilde{m} \in \N_0$ with $\tilde{m}+ M \geq 1$. We assume that $P \in \mathscr{L}(H^{s+m}_q, H^s_q)$ with  $P^{-1} \in \mathscr{L}(H^s_q,H^{s+m}_q)$ and  
  $$ \ad(-ix)^{\alpha_1} \ad(D_x)^{\beta_1} \ldots \ad(-ix)^{\alpha_l} \ad(D_x)^{\beta_l} P \in \mathscr{L}(H^{s+m}_q, H^s_q)$$
  for all $l \in \N$, $\alpha_1, \ldots, \alpha_l \in \Non$ and $\beta_1, \ldots, \beta_l \in \Non$ with $|\alpha_j + \beta_j| = 1$  for all $j \in \{ 1, \ldots, l\}$, $|\alpha_1|+ \ldots + |\alpha_l| \leq M$ and $|\beta_1|+ \ldots + |\beta_l| \leq \tilde{m}$.
  For $\alpha, \beta \in \Non$ with $|\alpha + \beta| = 1$ we have 
  $ \ad(-ix)^{\alpha} \ad(D_x)^{\beta} P^{-1} : \s \rightarrow \sd$. We consider $|\beta| = 0$ and $\alpha = e_j$ for an arbitrary $j \in \{ 1, \ldots, n\}$ first. On account of $\ad(-ix_j)P\in \mathscr{L}(H^{s+m}_q, H^s_q)$, we know that 
  \begin{align}\label{eq61}
    \ad(-ix_j) P u = -ix_j Pu + P(ix_j u) \in H^s_q(\Rn) \qquad \text{for all } u \in \s. 
  \end{align}
  If $u \in \s \subseteq H^{m+s}_q(\Rn)$, we obtain $P(ix_j u) \in H^s_q(\Rn)$. Together with (\ref{eq61}) this implies 
  \begin{align}\label{eq62}
     -ix_j Pu \in  H^s_q(\Rn) \qquad \text{for all } u \in \s.
  \end{align}
  Now we define $\mathscr{D}:= \{ Pu: u \in \s\} \subseteq H^{s}_q(\Rn)$. To show the density of $\mathscr{D}$ in $H^{s}_q(\Rn)$ we choose an arbitrary $v \in H^s_q(\Rn)$. On account of $P^{-1} \in  \mathscr{L}(H^s_q,H^{s+m}_q)$ we have $u:= P^{-1}v \in H^{s+m}_q(\Rn)$ and therefore $v=Pu$. Considering a sequence $(u_j)_{j \in \N_0} \subseteq \s$, which converges to $u$ in $H^{s+m}_q(\Rn)$, we define $v_j:= Pu_j$ for each $j \in \N_0$. 
  Due to $P \in \mathscr{L}(H^{s+m}_q,H^{s}_q)$ the sequence $(v_j)_{j \in \N}$ converges to $v$.
  This implies the density of $\mathscr{D}$ in $H^{s}_q(\Rn)$.
  Next we define $Q: \mathscr{D} \rightarrow H^{s+m}_q(\Rn)$ by $Qu:= -ix_j P^{-1} u  + P^{-1} (ix_j u)$ for all $u \in \mathscr{D}$. Due to (\ref{eq62})  
  $Q$ is well-defined and we obtain for all $u \in \s$:
  \begin{align}\label{eq63}
    Q(Pu)
    = -ix_j  u  + P^{-1} (ix_j Pu)
    =-P^{-1} [\ad(-ix_j)P]u.
  \end{align}
  With $\ad(-ix_j)P \in \mathscr{L}(H^{s+m}_q, H^s_q)$ and $P^{-1} \in \mathscr{L}(H^s_q,H^{s+m}_q)$ we get:
  \begin{align*}
    \|Q(Pu)\|_{H^{s+m}_q} 
    \leq C \|u\|_{H^{s+m}_q} 
    \leq C \|Pu\|_{H^{s}_q} \quad \text{for all } u \in H^s_q(\Rn).
  \end{align*}
  Due to the density of $\mathscr{D}$ in $H^{s}_q(\Rn)$ this implies $Q \in  \mathscr{L}(H^s_q,H^{s+m}_q)$. As a direct consequence we obtain 
  \begin{align*}
    \ad(-ix_j)P^{-1} \in \mathscr{L}(H^s_q,H^{s+m}_q)
  \end{align*}
  since $Qu=\ad(-ix_j)P^{-1}u$ for all $u \in \s$. Together with $\mathscr{ D} \subseteq H^s_q(\Rn)$ and (\ref{eq63}) we get
  \begin{align*}
    [\ad(-ix_j)P^{-1}]Pu = -P^{-1}[\ad(-ix_j)P]u \qquad \text{for all } u \in \s.
  \end{align*}
  On account of $[\ad(-ix_j)P^{-1}]P \in \mathscr{L}(H_q^{s+m})$ and $P^{-1}[\ad(-ix_j)P] \in \mathscr{L}(H_q^{s+m})$ the previous equality holds for all $u \in H^{s+m}_q(\Rn)$.
  The surjectivity of $P \in \mathscr{L}(H^{s+m}_q; H^s_q)$ yields for all $v \in H^{s}_q(\Rn)$:
  \begin{align}\label{p92}
    \ad(-ix)^{\alpha} \ad(D_x)^{\beta} P^{-1} v
    &=[\ad(-ix_j)P^{-1}]v = -P^{-1}[\ad(-ix_j)P]P^{-1}v \notag\\
    &=-P^{-1} [\ad(-ix)^{\alpha} \ad(D_x)^{\beta} P] P^{-1} v.
  \end{align}

  In the case $\beta = e_j$,  $j \in \{ 1, \ldots, n\}$ and $|\alpha| = 0$ we get (\ref{p92}) for all $u \in \s$ in the same way as before. Moreover, let $l \in \N$, $\alpha_1, \ldots, \alpha_l \in \Non$ and $\beta_1, \ldots, \beta_l \in \Non$ with $|\alpha_j + \beta_j| = 1$ for all $j \in \{ 1, \ldots, l\}$, $|\alpha_1|+ \ldots + |\alpha_l| \leq M$ and $|\beta_1|+ \ldots + |\beta_l| \leq \tilde{m}$. Denoting $\alpha:= \alpha_1 + \ldots + \alpha_l$ and $\beta := \beta_1 + \ldots + \beta_l$ we get by mathematical induction with respect to $l$:
  \begin{align*}
    \ad(-ix)^{\alpha_1} \ad(D_x)^{\beta_1} \ldots \ad(-ix)^{\alpha_l} \ad(D_x)^{\beta_l} P^{-1} 
    = \sum_{ \substack{ (\vara{\alpha^1}{l}) +  \ldots + (\vara{ \alpha^{ l } }{l}) = \alpha \\
			(\vara{\beta^1}{l}) +  \ldots + (\vara{ \beta^{ l } }{l}) = \beta } }
	  \hspace{-1.7cm} R_{\alpha_1^1, \ldots, \alpha_l^{l}, \beta_1^1, \ldots, \beta_l^{l } }
  \end{align*}
  where
  \begin{align*}
    R_{\alpha_1^1, \ldots, \alpha_l^{l}, \beta_1^1, \ldots, \beta_l^{l } } 
    := &C_{\alpha_1^1, \ldots, \alpha_l^{l}, \beta_1^1, \ldots, \beta_l^{l } } P^{-1} \\
	&\quad \circ \left[ \ad(-ix)^{ \alpha^1_l} \ad(D_x)^{ \beta^1_l } \ldots \ad(-ix)^{ \alpha^1_1} \ad(D_x)^{ \beta^1_1 } P \right] P^{-1} \\
	&\quad \circ \ldots \circ
	\left[ \ad(-ix)^{ \alpha^{ l }_{l} } \ad(D_x)^{ \beta^{ l }_{l} } \ldots \ad(-ix)^{ \alpha^{ l }_1} \ad(D_x)^{ \beta^{ l }_1 } P \right] P^{-1}. 
  \end{align*} 
\end{bem}

With this remark at hand, we now are able to show Theorem \ref{thm:SpectralinvarianceS000}:

\begin{proof}[Proof of Theorem \ref{thm:SpectralinvarianceS000}:] 
  Let $l \in \N_0$, $\alpha_1, \ldots, \alpha_l \in \Non$ and $\beta_1, \ldots, \beta_l \in \Non$ with $|\alpha_j + \beta_j| = 1$ for all $j \in \{ 1, \ldots, n\}$ and $|\beta_1 + \ldots + \beta_l| \leq \hat{m}$ be arbitrary. 
  Due to $p(x,D_x) \in  \mathcal{A}^0_{0,0}(\hat{m}, 2)$ and  $P^{-1} \in \mathscr{L} (L^2(\Rn))$ we can apply Remark \ref{bem:Formal identity} and get if we define $P:=p(x, D_x)$:
  \begin{align*}
    \ad(-ix)^{\alpha_1} \ad(D_x)^{\beta_1} \ldots \ad(-ix)^{\alpha_l} \ad(D_x)^{\beta_l} P^{-1} 
    = \sum_{ \substack{ (\vara{\alpha^1}{l}) +  \ldots + (\vara{ \alpha^{ l } }{l}) = \alpha \\
			(\vara{\beta^1}{l}) +  \ldots + (\vara{ \beta^{ l } }{l}) = \beta } }
	  \hspace{-1.7cm} R_{\alpha_1^1, \ldots, \alpha_l^{l}, \beta_1^1, \ldots, \beta_l^{l } }
  \end{align*}
  where $R_{\alpha_1^1, \ldots, \alpha_l^{l}, \beta_1^1, \ldots, \beta_l^{l } }$ are defined as in Remark \ref{bem:Formal identity}.
  Since $P \in  \mathcal{A}^0_{0,0}(\hat{m}, 2)$ and $P^{-1} \in \mathscr{L}(L^2(\Rn))$, we obtain
  $P^{-1} \in \mathcal{A}^0_{0,0}(\hat{m}, 2)$. Considering $0 < s \leq \hat{m} - n/2$, $s \notin \N$, Theorem \ref{thm:classA00HmqUloc} and Lemma \ref{lemma:einbettungSymbolklassen} yields the claim.
\end{proof}

In the same way we can show

\begin{lemma}
  Let $\tilde{m} \in \N_0$ with $\tilde{m}>n/2$ and $0 < \tau < 1$. For every non-smooth symbol $p \in C^{\tilde{m},\tau} \Sn{-n/2}{0}{0} $ with $p(x,D_x)^{-1} \in \mathscr{L}(L^2(\Rn))$ we get
  \begin{align*}
    p(x,D_x)^{-1} \in \op W^{\tilde{m},2}_{uloc} S^0_{0,0}(\RnRn) \subseteq \op C^s \Sn{0}{0}{0}
  \end{align*}
   for all $s \in (0, \tilde{m} - n/2]$ with $ s \notin \N$.
\end{lemma}

\subsection{Properties of Difference Quotients}\label{section:DifferenceQuotients}

Our next aim is to prove the spectral invariance for pseudodifferential operators $P$ whose symbols are in the symbol-class $C^{\tau} \Sn{0}{1}{0}$, $\tau > 0$. The proof is again based on the formal identities for the iterated commutators of $P^{-1}$, cf.\;Remark \ref{bem:Formal identity}. In this case $\ad(-ix)^{\alpha} \ad(D_x)^{\beta} P$, $|\alpha| \neq 0$ are pseudodifferenial operators of negative order $-|\alpha|$. Hence the order of the Bessel potential space increases by applying $\ad(-ix)^{\alpha} \ad(D_x)^{\beta} P$, $|\alpha| \neq 0$. Therefore $P^{-1} \in \mathscr{L}(L^q(\Rn))$ is not sufficient. We even need $P^{-1} \in \mathscr{L}(H^{-s}_q(\Rn))$ for certain $s \in \N_0$. As we always try to restrict the assumptions to a minimal, we use the tools of difference quotients in order to get $P^{-1} \in \mathscr{L}(H^{-s}_q(\Rn))$ if $P^{-1} \in \mathscr{L}(L^q(\Rn))$ is assumed. \\

\begin{Def}
  Let $h \in \R \backslash \{ 0 \}$ and $j \in \{ 1, \ldots, n\}$. For $u \in H^s_p(\Rn)$ with $s \in \R$ and $1 < p<\infty$ we define the \textit{difference quotient} of $u$ by
  \begin{align*}
    \p^h_{x_j} u(x) := h^{-1}\{ u(x+h e_j) - u(x) \} \qquad \text{for all } x \in \Rn. 
  \end{align*}

\end{Def}

Difference quotients have the following useful properties:

\begin{lemma}\label{lemma:DifferencequotientPIsPDO}\label{lemma:CommutatorWithDifferenceQuotients}
  Let $m \in \R$, $\tilde{m} \in \N$, $0<\tau < 1$ and $M \in \N_0 \cup \{ \infty \}$. Considering $p \in C^{\tilde{m}, \tau} \Snn{m}{1}{0}{M}$, we get for all $j \in \{ 1, \ldots, n\}$:
  \begin{itemize}
    \item[i)] $\left\{ \differencequotient{h} p(x,\xi) : h\in \R \backslash \{ 0 \} \right\} \subseteq C^{\tilde{m}-1, \tau} \Snn{m}{1}{0}{M}$ is bounded, 
    \item[ii)] $[\p^h_{x_j}, p(x,D_x)] u(x) = \left[\left( \differencequotient{-h} p \right)(x, D_x) u \right] (x+he_j)$ for all $u \in \s, x \in \Rn$ and $h \in \R \backslash \{ 0 \}$.
    \item[iii)] Additionally let $M \in \N_0 \cup \{ \infty \}$ with $M>n/2$ for $q \geq 2$ and $M>n/q$ else and $s \in \R$ with $|s| < \tilde{m} - 1 + \tau$. Then we have for some $C>0$:
    \begin{align*}
    \| [ \p^h_{x_j}, p(x,D_x) ] u \|_{H^s_q} \leq C \|u\|_{H^{s+m}_q} \qquad \text{for all } u \in H^{s+m}_q(\Rn),h \in \R \backslash \{ 0 \} .
  \end{align*}
  \end{itemize}
\end{lemma}

\begin{proof}
  Due to the fundamental theorem of calculus we get claim $i)$. 
  In order to verify claim $ii)$ let $u \in \s$ be arbitrary. An application of 
  $\frac{e^{ihe_j\cdot \xi} - 1 }{h} \hat{u} (\xi) = \widehat{\differencequotient{h} u } (\xi)$ 
  yields for all $x \in \Rn$:
  \begin{align*}
    \differencequotient{h} [p(x, D_x)u(x)] 
    = \left[\left( \differencequotient{-h}  p \right)(x, D_x) u\right] (x+he_j) + \left[ p(x, D_x) \left( \differencequotient{h}  u \right)\right](x).
  \end{align*}
  Hence $ii)$ holds. Finally we can verify $iii)$ by means of Lemma \ref{lemma:DifferencequotientPIsPDO} i), Theorem \ref{thm:stetigInHoelderRaum} and the density of $\s$ in $ H^{s+m}_q(\Rn)$.
\end{proof}

\begin{thm}(Difference quotients and weak derivatives)\label{thm:DifferenceQuotientsWeakDerivatives}
  \begin{enumerate}
    \item[i)] We suppose $1 < p < \infty$ and $u \in H^{s+1}_p(\Rn)$ with $s \in \R$. Then there is a constant $C$, independent of $h \in \R \backslash \{0\}$ and $u \in H^{s+1}_p(\Rn)$, such that
      \begin{align*}
	\| \p^h_{x_j} u\|_{H^{s}_p} \leq C \| \p_{x_j} u \|_{H^{s}_p}
      \end{align*}
      for all $j \in \{ 1, \ldots, n \}$ and all $h \in \R \backslash \{0\}$.
    \item[ii)] Let $1 < p < \infty$ and $u \in H^{s}_p(\Rn)$ with $s \in \R$. Additionally we assume
      \begin{align*}
	\| \p^h_{x_j} u\|_{H^{s}_p} \leq C \qquad \text{for all } j \in \{ 1, \ldots, n \} \text{ and } h \in \R \backslash \{0\}. 
      \end{align*}
      Then $u \in H^{s+1}_p(\Rn)$ and $\| \p_{x_j} u \|_{H^{s}_p} \leq C$.
  \end{enumerate}
  Note that assertion $ii)$ is false for $p=1$ while $i)$ also holds for $p=1$.
\end{thm}

\begin{proof}
  The proof of the case $s=0$ is essentially the same as that one of Theorem 5.8.3 in \cite{Evans}. 
  Assuming an arbitrary $s \in \R \backslash \{ 0 \}$ and $u \in H^{s+1}_p(\Rn)$
  we know that $\<{D_x}^s u \in W^1_p(\Rn)$.  
  Moreover Lemma \ref{lemma:CommutatorWithDifferenceQuotients} provides $[\p^h_{x_j}, \<{D_x}^s] = 0$ for all $h \in \R \backslash \{ 0 \}$ and $j \in \{ 1, \ldots, n \}$.
  Hence an application of the case $s=0$ provides for each $j \in \{ 1, \ldots, n \}$:
  \begin{align*}
    \| \differencequotient{h} u \|_{H^s_p} 
    = \| \differencequotient{h} \<{D_x}^s u \|_{L^p} 
    \leq C \| \p_{x_j} \<{D_x}^s u \|_{L^p} 
    = C \|\p_{x_j} u \|_{H^s_p}
  \end{align*}
  for all $h \in \R \backslash \{0\}$ and $u \in H^{s+1}_p(\Rn)$. 
  Similary to $i)$ we obtain the case $s \in \R$ of $ii)$ as a consequence of case $s=0$ and Lemma \ref{lemma:CommutatorWithDifferenceQuotients}. 
\end{proof}

The previous theorem allows us to verify the following proposition:

\begin{prop}\label{prop:EigenschaftVonPundSeinerInversen}
  Let $k \in \N_0$, $r \in \R$ and $1 < q < \infty$. Moreover, let $P$ be an operator, which fulfills for all $s \in \{ r, r+1, \ldots, r+k\}$ the properties 
  \begin{enumerate}
    \item[i)] $P \in \mathscr{L}(H^s_q, H^s_q)$,
    \item[ii)] $P \in \mathscr{L}(H^{r+k+1}_q, H^{r+k+1}_q)$,
    \item[iii)] $ \{ [ P, \p^h_{x_j} ] : h \in \R \backslash \{0\} \} \subseteq \mathscr{L}(H^s_q, H^s_q)$ is bounded for all $j \in \{ 1, \ldots, n \}$,
    \item[iv)] $P^{-1} \in \mathscr{L}(H^r_q, H^r_q)$.
  \end{enumerate}
  Then $P^{-1} \in \mathscr{L}(H^s_q, H^s_q)$ for each $s \in \{ r, r+1, \ldots, r+k+1\}$.
\end{prop}

\begin{proof}
  We prove the claim by mathematical induction with respect to $s$. In the case $s=r$ there is nothing to show. For $s \in \{ r, r+1, \ldots, r+k\}$ we choose an arbitrary $j \in \{ 1, \ldots, n \}$ and $f \in H^{s+1}_q(\Rn) \subseteq H^s_q(\Rn)$. The induction hypothesis provides the existence of a $u \in H^s_q(\Rn)$ with $u=P^{-1}f$. Due to $P \in \mathscr{L}(H^s_q, H^s_q)$, we get $Pu \in H^s_q(\Rn)$ 
  and consequently $\p^h_{x_j} (Pu) \in H^s_q(\Rn)$. Similary we get $P(\p^h_{x_j} u) \in H^s_q(\Rn)$. An application of $P^{-1}$ to $P(\p^h_{x_j}u) = [P, \p^h_{x_j}] u + \p^h_{x_j} (Pu)$,
  the induction hypothesis,
  the assumptions
  and Theorem \ref{thm:DifferenceQuotientsWeakDerivatives} $i)$ yield
  \begin{align*}
    \| \p^h_{x_j} u \|_{H^s_q} 
    &= \| P^{-1} \{ [P, \p^h_{x_j}] u + \p^h_{x_j} (Pu) \} \|_{H^s_q}
    \leq C \| [P, \p^h_{x_j}] u \|_{H^s_q} + C \| \p^h_{x_j} f \|_{H^s_q}\\
    &\leq C \| u \|_{H^s_q} + C \| \p_{x_j} f \|_{H^s_q}
     \leq C \qquad \text{for all } h \in \R \backslash \{0\} , u \in H^s_q(\Rn).
  \end{align*}
  Therefore Theorem \ref{thm:DifferenceQuotientsWeakDerivatives} $ii)$ provides $u \in H^{s+1}_q(\Rn)$ which proves the surjectivity of the linear, bounded and injective operator $P:H^{s+1}_q(\Rn) \rightarrow H^{s+1}_q(\Rn)$. Then $P^{-1}$ is an element of $ \mathscr{L}(H^{s+1}_q, H^{s+1}_q)$ by means of the bounded inverse theorem. 
\end{proof}

By means of the previous proposition we obtain the central result of this subsection: 

\begin{thm}\label{thm:InverseOfPEvenBoundedInAroundR}
  Let $1 < q < \infty$, $0< \tau < 1$, $\tilde{m} \in \N$ and $N \in \N_0 \cup \{ \infty\}$ with $N>n/2$ for $q\geq 2$ and $N > n/q$ else. We define $k:= \max\{ l \in \N_0 : r+l < \tilde{m} + \tau \}$ for one $r \in \R$ with $|r| < \tilde{m} + \tau$. Considering $p \in C^{\tilde{m}, \tau} \Snn{0}{1}{0}{N}$, where $p(x,D_x)^{-1} \in \mathscr{L}(H^r_q, H^r_q)$ we obtain
  \begin{align}\label{p95}
    p(x,D_x)^{-1} \in \mathscr{L}(H^s_q; H^s_q) \qquad \text{for all } s \in [ -r-k, r+k ].
  \end{align}
\end{thm}

\begin{proof}
  On account of Theorem \ref{thm:stetigInHoelderRaum} and Lemma \ref{lemma:CommutatorWithDifferenceQuotients} we can apply Proposition \ref{prop:EigenschaftVonPundSeinerInversen} and get the claim for all $s \in \{ r, \ldots, r+k \}$.
  With $ (\differencequotient{h} )^* = -\differencequotient{-h}$ at hand we have
  $[P^*, \differencequotient{h}] = [P, \differencequotient{-h}]^*.$
  An application of Proposition \ref{prop:EigenschaftVonPundSeinerInversen} to $P^{\ast}$ provides the claim for all $s \in \{ -r-k, \ldots, r-1 \}$. Then the claim is verified for all $s \in [ -r-k, r+k ]$ by means of interpolation. 
\end{proof}

\subsection[Spectral Invariance of Operators with Symbols in $C^{\tilde{m}, \tau} S^{0}_{1,0} $]{Spectral Invariance of Pseudodifferential Operators in the Symbol-Class $C^{\tilde{m}, \tau} S^{0}_{1,0} $}\label{section:SpectralInvarianceHölderCaseS10}

We are now in the position to show the next spectral invariance result:

\begin{thm}\label{thm:Spektralinvarianz}
  Let $1< q_0 < \infty$, $0 < \tau < 1$ and $\tilde{m}, \hat{m} \in \N_0$ with $\tilde{m} \geq \hat{m} > n/ q_0$. Additionally let $M \in \N_0$ with $n<  M \leq \tilde{m} -\hat{m}$. We define $\tilde{M}:=M-(n+1)$. Furthermore, let $N \in \N \cup \{ \infty \}$ with $N-M > n/2$ if $q_0 \geq 2$ and $N-M > n/q_0$ else. Considering $p \in C^{\tilde{m}, \tau} \Snn{0}{1}{0}{N}$, where $p(x, D_x)^{-1} \in \mathscr{L}( H^{r}_{q_0}, H^{r}_{q_0} ) $ for one $|r| < \tilde{m} + \tau$ we get  	
  \begin{align*}
    p(x, D_x)^{-1} \in \op W^{\hat{m},q_0}_{uloc} \Snn{0}{1}{0}{\tilde{M}-1}.         
  \end{align*}
  In the case $\tilde{M}-1>n/\tilde{q}$ for one $1< \tilde{q} \leq 2$, we even have 
  \begin{align*}
    p(x, D_x)^{-1} \in \mathscr{L} (L^q, L^q) \qquad \text{for all } q \in [\tilde{q};\infty) \cup \{ q_0 \}.
  \end{align*}
\end{thm}

\begin{proof}
  An application of Theorem \ref{thm:InverseOfPEvenBoundedInAroundR} provides the boundedness of
  \begin{align}\label{p93}
    p(x, D_x)^{-1} \in \mathscr{L}(H^{-s}_{q_0}, H^{-s}_{q_0} ) \qquad \text{for all } s \in \{ 0, \ldots, M \}. 
  \end{align}
  Let $l \in \N_0$, $\varf{\alpha}{l} \in \Non$ and $\varf{\beta}{l} \in \Non$ with $|\alpha_j| + |\beta_j| = 1$ for all $j \in \{ 1, \ldots, l \}$, $|\alpha| \leq M$ and $|\beta| \leq \hat{m}$ where $\alpha:= \alpha_1 + \ldots + \alpha_l$ and $\beta := \beta_1 + \ldots + \beta_l$. Then Remark \ref{bem:SymbolOfIteratedCommutatorNonSmooth} and
  Theorem \ref{thm:stetigInHoelderRaum} yield for all  
  $s \in \{0, \ldots, M - |\alpha|\}$: 
  \begin{align}\label{p91}
    \ad(-ix)^{\alpha_l} \ad(D_x)^{\beta_l} \ldots \ad(-ix)^{\alpha_1} \ad(D_x)^{\beta_1} p(x, D_x)
      \in \mathscr{L}( H^{-s-|\alpha|}_{q_0}, H^{-s}_{q_0}).
  \end{align}
  Setting $P:= p(x, D_x)$ 
  we get due to Remark \ref{bem:Formal identity}
  \begin{align*}
    \ad(-ix)^{\alpha_1} \ad(D_x)^{\beta_1} \ldots \ad(-ix)^{\alpha_l} \ad(D_x)^{\beta_l} P^{-1} 
    = \sum_{ \substack{ (\vara{\alpha^1}{l}) +  \ldots + (\vara{ \alpha^{ l } }{l}) = \alpha \\
			(\vara{\beta^1}{l}) +  \ldots + (\vara{ \beta^{ l } }{l}) = \beta } }
	  \hspace{-1.7cm} R_{\alpha_1^1, \ldots, \alpha_l^{l}, \beta_1^1, \ldots, \beta_l^{l } }
  \end{align*}
  where $R_{\alpha_1^1, \ldots, \alpha_l^{l}, \beta_1^1, \ldots, \beta_l^{l } } $ is defined as in Remark  \ref{bem:SymbolOfIteratedCommutatorNonSmooth}.
  Together with (\ref{p93}) and (\ref{p91}) this provides 
  $P^{-1} \in \mathcal{A}^{0,M}_{1,0}(\hat{m},q_0)$. By means of Theorem \ref{thm:classificationA10HmqUloc} and Lemma \ref{lemma:einbettungSymbolklassen} we get for each $0 < \tilde{\tau} \leq \hat{m} - n/{q_0}$ with $\tilde{\tau} \notin \N$:
  \begin{align*}
    p(x, D_x)^{-1} &\in \op W^{\hat{m},q_0}_{uloc} \Snn{0}{1}{0}{\tilde{M}-1} \cap \mathscr{L}(L^{q_0}(\Rn)) \\
      &\subseteq C^{\tilde{\tau} } \Snn{0}{1}{0}{\tilde{M}-1}.
  \end{align*}
  Finally, considering $\tilde{M}-1>n/ \tilde{q}$ for one $1 < \tilde{q} \leq 2$ we obtain for every $q \in [\tilde{q}, \infty)$ due to Theorem \ref{thm:stetigInHoelderRaum} the boundedness of
  \begin{align*}
    P^{-1}: L^q(\Rn) \rightarrow L^q(\Rn).
  \end{align*}
  \vspace*{-1cm}

\end{proof}

The relation to the spectral invariance of Theorem \ref{thm:Spektralinvarianz} is emphasised in the next corollary which easily can be verified by means of Theorem \ref{thm:Spektralinvarianz}. For more details we refer to \cite[Corollary 6.12]{Diss}.

\begin{kor}
  We assume that all assumptions of Theorem \ref{thm:Spektralinvarianz} hold. Additionally we choose an arbitrary but fixed $\tilde{q} \in (1, 2]$ fulfilling the conditions of Theorem \ref{thm:Spektralinvarianz} and denote 
  $$P_{L^q}:= p(x, D_x): L^q(\Rn) \rightarrow L^q(\Rn) \qquad \text{for all } \tilde{q} \leq q < \infty.$$
  Then we obtain the spectral invariance of these operators:
  \begin{align*}
    \sigma(P_{L^{q} }) = \sigma(P_{L^{r} }) \qquad \text{for all } \tilde{q} \leq q, r < \infty.
  \end{align*}
\end{kor}

Now one may wonder whether it is possible to prove that $p(x, D_x)^{-1}$ is even an element of $\op W^{\hat{m},q}_{uloc} \Sn{0}{1}{0}$ if all assumptions of the Theorem \ref{thm:Spektralinvarianz} are fulfilled and additionally $p(x,D_x) \in \op C^{\tilde{m},\tau} \Sn{0}{1}{0} $. Unfortunately in general this is not the case as we see in the next example:

\begin{bsp}
  Let $s >0$, $1< q_0 < \infty$ and $\tau > s + \lfloor n/q_0 \rfloor + n +4$. Additionally let $p(\xi) \in S^{0}_{1,0}(\RnRnx{x}{\xi})$ be a symbol which is not constantly equal to zero and where $p(D_x)^{-1} \in \mathscr{L}(L^{q_0}(\Rn), L^{q_0}(\Rn))$. Moreover let $a \in C^{\tau}(\Rn)$ such that there is no open set $U \subseteq \Rn$, $U \neq \varnothing$ with $a|_{U} \in C^{\infty}(U)$
  and there are two constants $c, C >0$ with $C> a(x) > c$ for all $x \in \Rn$. Then $T:=a(x)p(D_x) \in C^{\tau} \Sn{0}{1}{0}$ fulfills all assumptions of Theorem \ref{thm:Spektralinvarianz} for $M= n+3$ and $\hat{m}:= \lfloor \tau \rfloor - (n+3)$. 
  Consequently $T^{-1} \in \op W^{\hat{m},q_0}_{uloc}  \Snn{0}{1}{0}{\tilde{M}-1}$, where $\tilde{M}:= M-(n+1)$,  but $T^{-1} \notin \op C^{ \tilde{\tau} } \Sn{0}{1}{0}$ with $\tilde{\tau} \in (0,\hat{m}-n/q_0]$. In particular $T^{-1} \notin \op W^{\hat{m},q}_{uloc}  \Sn{0}{1}{0}$ due to Lemma \ref{lemma:einbettungSymbolklassen}.
\end{bsp}

\begin{proof}
  We define 
  $b(x):=\left( a(x) \right)^{-1}$ for all $x \in \Rn$. Then $b \in C^{\tau}(\Rn)$ and $T \in C^{\tau} \Sn{0}{1}{0}$. 
  Due to Theorem \ref{thm:SpektralinvarianzGlatterFall} $p(D_x)^{-1} \in \op \Sn{0}{1}{0}$. Hence $T^{-1}= p(D_x)^{-1} b(x)$. In particular the boundedness of $b$ and $p(D_x)^{-1} \in \mathscr{L}(L^{q_0}(\Rn), L^{q_0}(\Rn))$ imply $T^{-1} \in \mathscr{L}(L^{q_0}(\Rn), L^{q_0}(\Rn))$. Therefore all assumptions of Theorem \ref{thm:Spektralinvarianz} are fulfilled for $M= n+3$ and $\hat{m}:= \lfloor \tau \rfloor - (n+3)$. Let $\tilde{\tau} \in (0,\hat{m}-n/q_0]$. Assuming $T^{-1} \in C^{ \tilde{\tau} } \Sn{0}{1}{0}$ there is a kernel $\tilde{k}: \Rn \times \Rn \backslash \{ 0 \} \rightarrow \C$ such that $\tilde{k}(x,.) \in C^{\infty}( \Rn \backslash \{ 0 \} ) $ for each $x \in \Rn$ and 
  \begin{align}\label{p99}
    T^{-1} f (x) = \int \tilde{k}(x, x-y) f(y) dy
  \end{align}
  for all $f \in \s$ and $x \notin \supp f$ due to Theorem \ref{thm:KernelRepresentationNonsmoothCase}. An application of Remark \ref{bem:KernelRepresentationSmoothCase} provides the existence of a kernel $k \in C^{\infty} (\Rn \backslash \{ 0 \} )$ such that
  \begin{align}\label{p103}
    p(D_x)^{-1} u(x) = \int k( x-y) u(y) dy \qquad \text{for all } x \notin \supp u
  \end{align}
  for all $u \in \s$. Now let $(\delta_{\e} )_{\e>0} \subseteq C^{\infty}_c(\Rn)$ be a Dirac family, i.e.\,, for all $\e > 0$ we have $\delta_{\e} \geq 0$, $\int \delta_{\e}(x) dx = 1$ and $\lim_{\e \rightarrow 0} \int_{|x|\geq d} \delta_{\e} (x) dx = 0$ for every $d>0$. Then
  $\delta_{\e} \ast b \in C^{\infty}(\Rn)$ for each $\e >0$. 
  The boundedness of $b$, Theorem 10.7 in \cite{Kaballo} and $\p_x^{\alpha} \delta_{\e} \in C^{\infty}_c(\Rn) \subseteq L^1(\Rn) $ provides for every $\alpha \in \Non$
  \begin{align*}
    \left| \p_x^{\alpha} (\delta_{\e} \ast b)(x) \right| 
    \leq \int  \left| (\p_x^{\alpha} \delta_{\e})(y) \right| \left| b(x-y) \right| dy 
    \leq \|b\|_{L^{\infty} } \| \p_x^{\alpha} \delta_{\e} \|_{L^1} 
    \leq C_{\alpha, \e}
  \end{align*}
  for all $x \in \Rn$. In the case $|\alpha|=0$ the constant $C_{\alpha,\e}$ is even independent of $\e>0$.  
  In particular  $\delta_{\e} \ast b \in C^{\infty}_b(\Rn)$ 
  for every $\e >0$ and therefore $(\delta_{\e} \ast b)f \in \s$ for all $f \in \s$. Additionally we obtain for all $f \in \s$ with $x \notin \supp f$ the existence of a constant $C$, independent of $\e>0$, such that
  \begin{align*}
    |k( x-y) (\delta_{\e} \ast b)(y) f(y)| \leq C |k(x-y) f(y)| \in L^1(\Rn_y).
  \end{align*}
  and 
  \begin{align}\label{p104}
    (\delta_{\e} \ast b)(y) f(y) \xrightarrow[]{\e \rightarrow 0} b(y) f(y) \qquad \text{for all } y \in \Rn.
  \end{align} 
  Using (\ref{p99}), $p(D_x)^{-1} \in \mathscr{L}(L^2(\Rn))$ and (\ref{p104}) first we obtain together with (\ref{p103}) and an application of Lebesgue's theorem for all $f \in \s$:
  \begin{align}\label{p105}
    &\int \tilde{k}(x, x-y) f(y) dy =
    T^{-1} f(x) 
    = p(D_x)^{-1} \left[ \lim_{\e \rightarrow 0} (\delta_{\e} \ast b)(x) f(x) \right] \notag \\
    &\qquad \quad = \lim_{\e \rightarrow 0} \int k( x-y) (\delta_{\e} \ast b)(y) f(y) dy
    = \int k( x-y) b(y) f(y) dy 
  \end{align}
  for all $x \notin \supp f$.  
  Now we fix $x \in \Rn$ such that $\tilde{k}(x, .)$ is not constantly equal to zero. 
  An application of the fundamental lemma of calculus of variations
  yields 
  \begin{align*}
    k( x-y) b(y) = \tilde{k}(x, x-y) \qquad \text{for all } y \in \Rn \backslash \{ x\}
  \end{align*}
  since $k( x-y)$, $\tilde{k}(x, x-y)$ and $b(y)$ are continuous with respect to $y \in \Rn \backslash \{x\}$. By means of variable transformation we obtain
  \begin{align}\label{eq32}
    k( z) = a(x-z) \tilde{k}(x, z) \in C^{\infty}(\Rn_z \backslash \{ 0 \} ).
  \end{align}
  Now we choose $z \in \Rn \backslash \{0\}$ with $\tilde{k}(x,z) \neq 0$. Due to $\tilde{k}(x,.) \in C^{\infty}(\Rn \backslash \{0\})$, there some a $\delta>0$ such that $\tilde{k}(x,\tilde{z}) \neq 0$ for all $\tilde{z} \in B_{\delta}(z)$ and $0 \notin B_{\delta}(z)$. Together with $\tilde{k}(x,.) \in C^{\infty}(\Rn \backslash \{0\})$ and (\ref{eq32}) we obtain $a \in C^{\infty}(B_{\delta}(x-z))$. This is a contradiction to the choice of $a$. Therefore $T^{-1} \notin C^{ \tilde{\tau} } \Sn{0}{1}{0}$.
\end{proof}

\subsection[Spectral Invariance of Operators with Symbols in $W^{\tilde{m},q}_{uloc} S^{0}_{1,0} $]{Spectral Invariance of Pseudodifferential Operators in the Symbol-Class $W^{\tilde{m},q}_{uloc} S^{0}_{1,0} $}\label{section:SpectralInvarianceWmqUlocCase}

The present subsection serves to improve Theorem \ref{thm:Spektralinvarianz} for non-smooth pseudodifferential operators of the order zero with coefficients in $W^{\tilde{m},q}_{uloc}(\Rn)$.

Verifying the proof of Theorem \ref{thm:Spektralinvarianz} we see that we need similar results for pseudodifferential operators whose symbols are in $W^{\tilde{m},q}_{uloc} \Sn{m}{1}{0}$ instead of Theorem \ref{thm:InverseOfPEvenBoundedInAroundR} and Remark \ref{bem:SymbolOfIteratedCommutatorNonSmooth}:

\begin{lemma}\label{lemma:DifferencequotientPIsPDOwithCoefficientInLocalSobolevSpace}
  Let $1< q <\infty$, $m \in \R$ and $\tilde{m} \in \N$ with $\tilde{m} > 1+ n/q$. Considering $p \in W^{\tilde{m},q}_{uloc} S^{m}_{1,0}(\RnRn)$, we get the boundedness of
  \begin{align*}
    \left\{ \differencequotient{h} p(x,\xi) : h\in \R \backslash \{ 0 \} \right\} \subseteq W^{\tilde{m}-1, q}_{uloc} S^{m}_{1,0}(\RnRn)
  \end{align*}
  for all $j \in \{ 1, \ldots, n\}$.  
\end{lemma}

\begin{proof}
  Let $\psi \in C_c^{\infty}(\Rn)$ be such that $\supp \psi \subseteq \overline{B_2(0)}$ and $\psi(x)=1$ for all $x \in B_1(0)$. Assuming an arbitrary $\alpha \in \Non$ we get due to Theorem \ref{thm:DifferenceQuotientsWeakDerivatives} and $p \in W^{\tilde{m},q}_{uloc} S^{m}_{1,0}(\RnRn)$:
  \begin{align*}
    \| \pa{\alpha} \differencequotient{h} p(x,\xi) \|_{W^{\tilde{m}-1, q}_{uloc} (\Rn_x)}
    &\leq \sum_{|\beta| \leq \tilde{m}-1} \sup_{y \in \Rn} \| \differencequotient{h} \p_x^{\beta} \pa{\alpha} [p(x,\xi) \psi(x-y)] \|_{L^q(\Rn_x)} \notag \\
    &\leq C \sum_{|\beta| \leq \tilde{m}-1} \sup_{y \in \Rn} \| \p_{x_j} \p_x^{\beta} \pa{\alpha} [p(x,\xi) \psi(x-y)] \|_{L^q(\Rn_x)} \notag \\
    &\leq C \sum_{|\beta| \leq \tilde{m} } \sup_{y \in \Rn} \| \p_x^{\beta} \pa{\alpha} p(.,\xi) \|_{L^q(B_2(y))} \notag \\
    &\leq C \| \pa{\alpha} p(x,\xi) \|_{W^{\tilde{m},q}_{uloc}(\Rn_x)}
    \leq C \<{\xi}^{m-|\alpha|}
  \end{align*}
  for all $\xi \in \Rn$ and $h\in \R \backslash \{ 0 \}$. This implies the claim.
\end{proof}

\begin{lemma}\label{lemma:CommutatorWithDifferenceQuotientsLocallySobolevCase}
  Let $1< \tilde{q},q <\infty$, $m \in \R$ and $\tilde{m} \in \N$ with $\tilde{m} > 1+ n/q$. Assuming $p \in W^{\tilde{m},q}_{uloc} S^{m}_{1,0}(\RnRn)$ we get for every $j \in \{ 1, \ldots, n\} $ and all $h \in \R \backslash \{ 0 \}$:
  \begin{align*}
    [\p^h_{x_j}, p(x,D_x)] u(x) = \left( \differencequotient{-h} p \right)(x, D_x) u(x+he_j) \qquad \text{for all } u \in \s, x \in \Rn.
  \end{align*}
  Moreover, for all $ - \tilde{m} +1 + n/q < s \leq \tilde{m} - 1 - n(1/q - 1/{\tilde{q}})^{+}$ there is a constant $C$, independent of $h \in \R \backslash \{ 0 \}$, such that 
  \begin{align*}
    \| [ \p^h_{x_j}, p(x,D_x) ] u \|_{H^s_{\tilde{q}}} \leq C \|u\|_{H^{s+m}_{\tilde{q}}} \qquad \text{for all } u \in H^{s+m}_{\tilde{q}}(\Rn), 
  \end{align*}
  where $j \in \{ 1, \ldots, n\} $.
\end{lemma}

\begin{proof}
  The proof of the lemma is essentially the same as that one of Lemma \ref{lemma:CommutatorWithDifferenceQuotients}. We just have to replace Lemma \ref{lemma:DifferencequotientPIsPDO} with Lemma \ref{lemma:DifferencequotientPIsPDOwithCoefficientInLocalSobolevSpace} and Theorem \ref{thm:stetigInHoelderRaum} with Theorem \ref{thm:boundednessOfSymbolsWithCoefficientsInUniformlyBoundedSobolevSpaces}.
\end{proof}

\begin{lemma}\label{lemma:InverseOfPEvenBoundedInAroundRLocallySobolevCase}
  Let $1 < q, \tilde{q} < \infty$ and $\tilde{m} \in \N$ with $\tilde{m} >1+ n/q$. Considering  $p \in W^{\tilde{m},q}_{uloc} \Sn{m}{1}{0}$, where the inverse operator $p(x,D_x)^{-1} \in \mathscr{L}(H^r_{\tilde{q}}, H^r_{\tilde{q}})$ for one $ -\tilde{m} + n/q < r \leq \tilde{m} - n(1/q - 1/{\tilde{q}} )^{+} $, we obtain
  \begin{align*}
    p(x,D_x)^{-1} \in \mathscr{L}(H^s_{\tilde{q}}, H^s_{\tilde{q}}) \qquad \text{for all } s \in [ r-l, r+k ].
  \end{align*}
  Here $k$ and $l$ are defined by $k:= \max\{ \tilde{k} \in \N_0 : r+ \tilde{k} \leq \tilde{m} - n(1/q - 1/{\tilde{q}} )^{+} \}$ and $l:= \max \{ \tilde{l} \in \N_0 : -\tilde{m} + n/q < r-\tilde{l} \}$.
\end{lemma}

\begin{proof}
  Using Theorem \ref{thm:boundednessOfSymbolsWithCoefficientsInUniformlyBoundedSobolevSpaces} instead of Theorem \ref{thm:stetigInHoelderRaum} and Lemma \ref{lemma:CommutatorWithDifferenceQuotientsLocallySobolevCase} instead of Lemma \ref{lemma:CommutatorWithDifferenceQuotients} the statement follows in the same way as that one of Theorem \ref{thm:InverseOfPEvenBoundedInAroundR}.
\end{proof}

Comparing the previous result with that one of Theorem \ref{thm:InverseOfPEvenBoundedInAroundR} the difference lies in the choice of the neighbourhood of $r$. The previous lemma allows us 
to improve Theorem \ref{thm:Spektralinvarianz}:

\begin{thm}\label{thm:SpectralInvarianceWmqUloc}
  Let $1< q, q_0 < \infty$ and $\tilde{m} \in \N_0$ with $\tilde{m} > \max \{ 1+ n/q, n/ q_0 \}$. Additionally let $\hat{m} \in \N_0$ with $n/{q_0} < \hat{m} \leq \max \{ r \in \N_0 : r < \tilde{m} - n/q \}$. Moreover, let $M \in \N_0$ with $ n < M < \tilde{m} -\hat{m}- n/q$. We define $\tilde{M}:=M-(n+1)$. Considering $p \in W^{\tilde{m},q}_{uloc} S^{0}_{1,0}(\RnRn)$, where $p(x, D_x)^{-1} \in \mathscr{L}( H^{r}_{q_0}, H^{r}_{q_0} ) $ for one $ -\tilde{m} + n/q < r \leq \tilde{m} - n(1/q - 1/{q_0} )^{+} $ we get
  \begin{align*}
    p(x, D_x)^{-1} \in \op W^{\hat{m},q_0}_{uloc} \Snn{0}{1}{0}{\tilde{M}-1}.
  \end{align*}
  In the case $\tilde{M}-1>n/ \tilde{q}$ for one $1 < \tilde{q} \leq 2$, we even have 
  \begin{align*}
    p(x, D_x)^{-1} \in \mathscr{L} (L^{\tilde{q}}, L^{\tilde{q}}) \qquad \text{for all } \tilde{q} \in [\tilde{q}; \infty) \cup \{q_0\}.
  \end{align*}
\end{thm}

\begin{proof}
  We get the statement in the same way as that one of Theorem \ref{thm:Spektralinvarianz}. We just have to replace Theorem \ref{thm:InverseOfPEvenBoundedInAroundR} with Lemma \ref{lemma:InverseOfPEvenBoundedInAroundRLocallySobolevCase} and Remark \ref{bem:SymbolOfIteratedCommutatorNonSmooth} with Remark \ref{bem:SymbolOfIteratedCommutatorNonSmoothInUniformlyLocallySobolevSpaces}. Moreover, we have to use Theorem \ref{thm:boundednessOfSymbolsWithCoefficientsInUniformlyBoundedSobolevSpaces} instead of Theorem \ref{thm:stetigInHoelderRaum}. 
\end{proof}

Theorem \ref{thm:SpectralInvarianceWmqUloc} in fact is an improvement for $P \in \op W_{uloc}^{\tilde{m},q} S^0_{1,0} (\Rn \times \Rn)$ because Theorem \ref{thm:SpectralInvarianceWmqUloc} holds for the less strict assumption $-\tilde{m} + n/q < r \leq \tilde{m} - n(1/q - 1/{q_0} )^{+} $. For more details, see \cite[Section 6.4]{Diss}.
In the same manner as for non-smooth pseudodifferential operators of the symbol-class $W_{uloc}^{\tilde{m},q} S_{1,0}^0 (\Rn \times \Rn)$ we get a better result for the subsets $H_q^{\tilde{m}} S_{1,0}^0 (\Rn \times \Rn)$ and $W_{uloc}^{\tilde{m},q} S_{cl}^0 (\Rn \times \Rn)$ of $W_{uloc}^{\tilde{m},q} S_{1,0}^0 (\Rn \times \Rn)$.

\begin{lemma}\label{lemma:DifferencequotientPIsPDOwithCoefficientInSobolevSpace} \label{lemma:DifferencequotientPIsPDOwithCoefficientInLocalSobolevSpaceClassical}
  Let $1< q <\infty$, $m \in \R$ and $\tilde{m} \in \R$ with $\tilde{m} > 1+ n/q$. Then
  \begin{itemize}
    \item[i)] $\left\{ \differencequotient{h} p(x,\xi) : h\in \R \backslash \{ 0 \} \right\} \subseteq H^{\tilde{m}-1}_q \Sn{m}{1}{0}$ if $p \in H^{\tilde{m}}_q \Sn{m}{1}{0}$,
    \item[ii)] $\left\{ \differencequotient{h} p(x,\xi) : h\in \R \backslash \{ 0 \} \right\} \subseteq W^{\tilde{m}-1, q}_{uloc} S^{m}_{cl}(\RnRn)$ if $p \in W^{\tilde{m},q}_{uloc} S^{m}_{cl}(\RnRn)$ and  $\tilde{m} \in \N$
  \end{itemize}
  for all $j \in \{ 1, \ldots, n\}$.  
\end{lemma}

\begin{proof}
  By means of Theorem \ref{thm:DifferenceQuotientsWeakDerivatives} we can verify $i)$.
   Additionally we know from Lemma \ref{lemma:DifferencequotientPIsPDOwithCoefficientInLocalSobolevSpace} the boundedness of the set 
  $$\left\{ \differencequotient{h} p(x,\xi) : h\in \R \backslash \{ 0 \} \right\} \subseteq W^{\tilde{m}-1, q}_{uloc} \Sn{m}{1}{0}.$$
  Therefore it remains to show for all $h\in \R \backslash \{ 0 \}$ that $\differencequotient{h} p(x,\xi)$ is even an element of $W^{\tilde{m},q}_{uloc} S^{m}_{cl}(\RnRn)$. Hence let $h\in \R \backslash \{ 0 \}$ and $j \in \{ 1, \ldots, n\}$ be arbitrary. Since $p \in W^{\tilde{m},q}_{uloc} S^{m}_{cl}(\RnRn)$, there exists an expansion where $p_k$ are homogeneous of degree $m-k$ in $\xi$ (for $|\xi| \geq 1$) such that we have for all $N \in \N_0$
  \begin{align*}
    p(x, \xi) - \sum_{k=0}^{ N} p_k(x, \xi) \in W^{\tilde{m}, q}_{uloc} S^{m-N-1}_{1,0}(\Rn_x \times \Rn_{\xi}).
  \end{align*}
  Let $N \in \N_0$ be arbitrary. In the same way as in the proof of Lemma \ref{lemma:DifferencequotientPIsPDOwithCoefficientInLocalSobolevSpace} we get the boundedness of:
  \begin{align*}
    \left\{ \differencequotient{h} p(x, \xi) - \sum_{k=0}^{ N} \differencequotient{h} p_k(x, \xi) : h\in \R \backslash \{ 0 \} \right\} 
    \subseteq W^{\tilde{m}-1, q}_{uloc} S^{m-N-1}_{1,0}(\Rn_x \times \Rn_{\xi}).
  \end{align*}
  Since $\differencequotient{h} p_k(x, \xi)$ is homogeneous of degree $m-k$ in $\xi$ (for $|\xi| \geq 1$) for every $k \in \N_0$ claim $ii)$ holds.
\end{proof}

With the previous lemma at hand, we are able to prove

\begin{lemma}\label{lemma:CommutatorWithDifferenceQuotientsSobolevCase}
  Let $1< \tilde{q},q <\infty$, $m \in \R$ and $\tilde{m} \in \N$ with $\tilde{m} > 1+n/q$. Additionally let $p$ be either an element of the symbol-class $H^{\tilde{m}}_q \Sn{m}{1}{0}$ or of $W^{\tilde{m},q}_{uloc} S^{m}_{cl}(\RnRn)$. Then we get for every $j \in \{ 1, \ldots, n\} $ and $h \in \R \backslash \{ 0 \}$:
  \begin{align*}
    [\p^h_{x_j}, p(x,D_x)] u(x) = \left( \differencequotient{-h} p \right)(x, D_x) u(x+he_j) \qquad \text{for all } u \in \s, x \in \Rn.
  \end{align*}
  Moreover, for all $s \in \R$ with $ n(1/{\tilde{q}} + 1/q - 1)^{+} - \tilde{m} +1 < s \leq \tilde{m} - 1 - n(1/q - 1/{\tilde{q}})^{+}$ there is a constant $C$, independent of $h \in \R \backslash \{ 0 \}$, such that 
  \begin{align*}
    \| [ \p^h_{x_j}, p(x,D_x) ] u \|_{H^s_{\tilde{q}}} \leq C \|u\|_{H^{s+m}_{\tilde{q}}} \qquad \text{for all } u \in H^{s+m}_{\tilde{q}}(\Rn), 
  \end{align*}
  where $j \in \{ 1, \ldots, n\} $.
\end{lemma}

\begin{proof}
  The proof of the lemma is essentially the same as that one of Lemma \ref{lemma:CommutatorWithDifferenceQuotients} ii) and iii). 
  We just have to replace Lemma \ref{lemma:DifferencequotientPIsPDO} i) with Lemma \ref{lemma:DifferencequotientPIsPDOwithCoefficientInSobolevSpace} and Theorem \ref{thm:stetigInHoelderRaum} with Theorem \ref{thm:Marschall,Thm2.2} in the case $p \in H^{\tilde{m}}_q \Sn{m}{1}{0}$ and with Theorem \ref{thm:boundednessOfClassicalSymbolsInSobolevSpaceUniformly} else. 
\end{proof}

The previous lemma enables us to verify 

\begin{lemma}\label{lemma:InverseOfPEvenBoundedInAroundRSobolevCase}
  Let $1 < q, \tilde{q} < \infty$ and $\tilde{m} \in \N$ with $\tilde{m} > 1+ n/q$. We consider either $p \in H^{\tilde{m}}_q \Sn{0}{1}{0}$ or $p \in W^{\tilde{m},q}_{uloc} S^{0}_{cl}(\RnRn)$ such that $p(x,D_x)^{-1} \in \mathscr{L}(H^r_{\tilde{q}}, H^r_{\tilde{q}})$ for one $r \in \R$ with $ n(1/{\tilde{q}} + 1/q - 1)^+ -\tilde{m} < r \leq \tilde{m} - n(1/q - 1/{\tilde{q}} )^{+} $, then we obtain
  \begin{align*}
    p(x,D_x)^{-1} \in \mathscr{L}(H^s_{\tilde{q}}, H^s_{\tilde{q}}) \qquad \text{for all } s \in [ r-l, r+k ].
  \end{align*}
  Here $k$ and $l$ are defined by $k:= \max\{ \tilde{k} \in \N_0 : r+ \tilde{k} \leq \tilde{m} - n(1/q - 1/{\tilde{q}} )^{+} \}$ and $l:= \max \{ \tilde{l} \in \N_0 : n(1/{\tilde{q}} + 1/q - 1)^+ -\tilde{m} < r- \tilde{l} \}$.
\end{lemma}

\begin{proof}
  The statement follows in the same way as that one of Theorem \ref{thm:InverseOfPEvenBoundedInAroundR}. We merely have to use Lemma \ref{lemma:CommutatorWithDifferenceQuotientsSobolevCase} instead of Lemma \ref{lemma:CommutatorWithDifferenceQuotients}. Moreover, we have to apply Theorem \ref{thm:Marschall,Thm2.2} instead of Theorem \ref{thm:stetigInHoelderRaum} in the case  $p \in H^{\tilde{m}}_q \Sn{0}{1}{0}$ and Theorem \ref{thm:boundednessOfClassicalSymbolsInSobolevSpace} otherwise.
\end{proof}

Now we are in the position to prove the spectral invariance of non-smooth pseudodifferential operators whose symbols are in $H^{\tilde{m}}_q \Sn{0}{1}{0}$ or in $W^{\tilde{m},q}_{uloc} S^{0}_{cl}(\RnRn)$:

\begin{thm}\label{thm:spectralInvarianceCaseBesselPotentialSpace}
  Let $1< q, q_0 < \infty$ and $\tilde{m} \in \N_0$ with $\tilde{m} > \max \{1+ n/q, n/ q_0 \}$. Furthermore, let $\hat{m} \in \N_0$ with $n/{q_0} < \hat{m} \leq \max \{ r \in \N_0 : r < \tilde{m} - n/q \}$. Additionally let $M \in \N_0$ with $n <  M < \tilde{m} -\hat{m}- n(1/q + 1/{q_0} -1)^{+}$. We define $\tilde{M}:=M-(n+1)$. Considering $p \in H^{\tilde{m}}_q \Sn{0}{1}{0}$ or $p \in W^{\tilde{m},q}_{uloc} S^{0}_{cl}(\RnRn)$, where $p(x, D_x)^{-1} \in \mathscr{L}( H^{r}_{q_0}, H^{r}_{q_0} ) $ for one $r \in \R$ with $ n(1/{q_0} + 1/q - 1)^+ -\tilde{m} < r \leq \tilde{m} - n(1/q - 1/{q_0} )^{+} $ we get 	
  \begin{align*} 			
    p(x, D_x)^{-1} \in \op W^{\hat{m},q_0}_{uloc}
    \Snn{0}{1}{0}{ \tilde{M}-1}.
  \end{align*}
  In the case $\tilde{M}-1>n/\tilde{q}$ for one $1 < \tilde{q} \leq 2$, we even have 
  \begin{align*}
    p(x, D_x)^{-1} \in \mathscr{L} (L^{\tilde{q}}, L^{\tilde{q}}) \qquad \text{for all } \tilde{q} \in [\tilde{q}; \infty) \cup \{q_0\}.
  \end{align*}
\end{thm}

\begin{proof}
  We get the statement in the same way as that one of Theorem \ref{thm:Spektralinvarianz}. We just have to replace Theorem \ref{thm:InverseOfPEvenBoundedInAroundR} with Lemma \ref{lemma:InverseOfPEvenBoundedInAroundRSobolevCase} and Remark \ref{bem:SymbolOfIteratedCommutatorNonSmooth} with Remark \ref{bem:SymbolOfIteratedCommutatorNonSmoothInSobolevSpaces} in the case $p \in H^{\tilde{m}}_q \Sn{m}{1}{0}$ and with Remark \ref{bem:SymbolOfIteratedCommutatorNonSmoothInUniformlyLocallySobolevSpacesClassical} else. Moreover, we have to use Theorem \ref{thm:Marschall,Thm2.2} instead of Theorem \ref{thm:stetigInHoelderRaum} in the case $p \in H^{\tilde{m}}_q \Sn{m}{1}{0}$ and Theorem \ref{thm:boundednessOfClassicalSymbolsInSobolevSpace} otherwise. 
\end{proof}


\begin{thebibliography}{10}

\bibitem{PDO}
H.~Abels.
\newblock {\em Pseudodifferential and Singular Integral Operators: An
  Introduction with Applications}.
\newblock De Gruyter, Berlin/Boston, 2012.

\bibitem{Paper1}
H.~{Abels} and C.~{Pfeuffer}.
\newblock {Characterization of Non-Smooth Pseudodifferential Operators}.
\newblock {\em Preprint arXiv:1512.01127}, 2015.

\bibitem{AlvarezHounie}
J.~{Alvarez} and J.~{Hounie}.
\newblock {Spectral invariance and tameness of pseudo-differential operators on
  weighted Sobolev spaces.}
\newblock {\em {J. Oper. Theory}}, 30(1):41--67, 1993.

\bibitem{AmannPaper}
H.~{Amann}.
\newblock {Compact embeddings of vector-valued Sobolev and Besov spaces.}
\newblock {\em {Glas. Mat., III. Ser.}}, 35(1):161--177, 2000.

\bibitem{Beals}
R.~Beals.
\newblock Characterization of pseudodifferential operators and applications.
\newblock {\em Duke Math. J.}, 44(1):45--57, 1977.

\bibitem{Elstrodt}
J.~Elstrodt.
\newblock {\em Maß- und Integrationstheorie}.
\newblock Springer-Verlag Berlin [u.a.], 2011.

\bibitem{Evans}
L.C. Evans.
\newblock {\em Partial Differential Equations}.
\newblock Americal Mathematical Society, Second Edition, 2010.

\bibitem{Grubb}
G.~{Grubb}.
\newblock {Parameter-elliptic and parabolic pseudodifferential boundary
  problems in global $L\sb p$ Sobolev spaces.}
\newblock {\em {Math. Z.}}, 218(1):43--90, 1995.

\bibitem{Kaballo}
W.~{Kaballo}.
\newblock {\em {Einf\"uhrung in die Analysis III.}}
\newblock Heidelberg: Spektrum Akademischer Verlag, 1999.

\bibitem{Kryakvin}
V.D. Kryakvin.
\newblock {Characterization of pseudodifferential operators in H\"older-Zygmund
  spaces.}
\newblock {\em Differ. Equ.}, 49(3):306--312, 2013.

\bibitem{SchroheLeopold1992}
H.-G. {Leopold} and E.~{Schrohe}.
\newblock {Spectral invariance for algebras of pseudodifferential operators on
  Besov spaces of variable order of differentiation.}
\newblock {\em {Math. Nachr.}}, 156:7--23, 1992.

\bibitem{LeopoldTriebel}
H.-G. {Leopold} and H.~{Triebel}.
\newblock {Spectral invariance for pseudodifferential operators on weighted
  function spaces.}
\newblock {\em {Manuscr. Math.}}, 83(3-4):315--325, 1994.

\bibitem{Marschall}
J.~Marschall.
\newblock Pseudodifferential operators with nonregular symbols of the class
  {$S^m_{\rho\delta}$}.
\newblock {\em Comm. Partial Differential Equations}, 12(8):921--965, 1987.

\bibitem{Marschall2}
J.~Marschall.
\newblock Pseudodifferential operators with coefficients in {S}obolev spaces.
\newblock {\em Trans. Amer. Math. Soc.}, 307(1):335--361, 1988.

\bibitem{Diss}
C.~Pfeuffer.
\newblock Characterization of non-smooth pseudodifferential operators.
\newblock {urn:nbn:de:bvb:355-epub-31776. available in the internet:
  http://epub.uni-regensburg.de/31776/ [23.06.2015]}, May 2015.

\bibitem{Ruzicka}
M.~{R\r u\v zi\v cka}.
\newblock {\em {Nichtlineare Funktionalanalysis. Eine Einf\"uhrung.}}
\newblock Berlin: Springer, 2004.

\bibitem{Schrohe1990}
E.~{Schrohe}.
\newblock {Boundedness and spectral invariance for standard pseudodifferential
  operators on anisotropically weighted $L\sp p$-Sobolev spaces.}
\newblock {\em {Integral Equations Oper. Theory}}, 13(2):271--284, 1990.

\bibitem{Schrohe1992}
E.~{Schrohe}.
\newblock {Spectral invariance, ellipticity, and the Fredholm property for
  pseudodifferential operators on weighted Sobolev spaces.}
\newblock {\em {Ann. Global Anal. Geom.}}, 10(3):237--254, 1992.

\bibitem{Taylor2}
M.E. Taylor.
\newblock {\em Pseudodifferential Operators and Nonlinear PDE}.
\newblock Birkhäuser Boston, Basel, Berlin, 1991.

\bibitem{Ueberberg}
J.~Ueberberg.
\newblock Zur {S}pektralinvarianz von {A}lgebren von
  {P}seudodifferentialoperatoren in der {$L^p$}-{T}heorie.
\newblock {\em Manuscripta Math.}, 61(4):459--475, 1988.

\end{thebibliography}

\end{document}